\documentclass[reqno]{amsart}

\usepackage{amssymb,amsxtra,latexsym}

\usepackage[sc]{mathpazo}
\usepackage{mathrsfs}
\usepackage[T1]{fontenc}
\usepackage{textcomp}

\usepackage[cmtip,matrix,arrow,curve]{xy}
\CompileMatrices
\newdir{((}{{\lhook\kern-.1em}}
\newdir{))}{{\rhook\kern-.1em}}
\newdir{ >}{{}*!/-5pt/@{>}}

\usepackage{multicol}

\usepackage{booktabs}

\usepackage{manfnt}

\makeatletter

\def\theglossary{\@restonecoltrue\if@twocolumn\@restonecolfalse\fi
\columnseprule\z@ \columnsep 35\p@
\let\@makessectionhead\indexsec
\@xp\section\@xp*\@xp{\glossaryname}%
\let\item\@idxitem
\parindent\z@  \parskip\z@\@plus.3\p@\relax
\footnotesize}

\def\glossaryname{Index of Notation}

\makeatother

\makeglossary

\numberwithin{equation}{section}

\renewcommand\dots{\relax\ifmmode\ldots\else$\,\ldots\,$\fi}

\newcommand\note[1]%
{$^\dagger$\marginpar{\footnotesize{$^\dagger${#1}}}}

\divide\count253 by 60 %hours after midnight, rounded down
\divide\count255 by 60
\multiply\count255 by 60 %minutes after midnight, rounded down
\advance\count254 by -\count255 %minutes after the hour

\def\today{\number\year-\ifnum\month<10
0\fi\number\month-\ifnum\day<10 0\fi\number\day}

\def\hour{\ifnum\count253<10
0\number\count253\else\number\count253\fi}

\def\minute{\ifnum\count254<10
0\number\count254\else\number\count254\fi}

\swapnumbers
\newtheorem{theorem}{Theorem}[subsection]
\newtheorem{lemma}[theorem]{Lemma}
\newtheorem{proposition}[theorem]{Proposition}
\newtheorem*{proposition*}{Proposition}

\newtheorem{section-theorem}{Theorem}[section]
\newtheorem{section-lemma}[section-theorem]{Lemma}
\newtheorem{section-proposition}[section-theorem]{Proposition}
\newtheorem{section-corollary}[section-theorem]{Corollary}
\newtheorem{section-claim}[section-theorem]{Claim}
\newtheorem{section-conjecture}[section-theorem]{Conjecture}

\theoremstyle{definition}
\newtheorem{definition}[theorem]{Definition}
\newtheorem{example}[theorem]{Example}

\newtheorem{remark}[theorem]{Remark}

\newtheorem{section-definition}[section-theorem]{Definition}
\newtheorem{section-example}[section-theorem]{Example}
\newtheorem{section-problem}[section-theorem]{Problem}
\newtheorem{section-remark}[section-theorem]{Remark}

\newcommand\lie{\mathfrak}

\newcommand{\g}{\lie{g}}
\newcommand{\h}{\lie{h}}
\newcommand{\m}{\lie{m}}
\newcommand{\n}{\lie{n}}

\newcommand\bb[1]{{\text{\bf#1}}}

\newcommand\Z{\bb{Z}} 
\newcommand\Q{\bb{Q}}
\newcommand\R{\bb{R}} 
\newcommand\C{\bb{C}}

\newcommand\K{\bb{K}}

\renewcommand\P{\bb{P}}

\newcommand\ca{\mathscr}

\newcommand\X{\ca{X}}
\newcommand\Y{\ca{Y}}

\newcommand\func[1]{\operatorname{\mathrm{#1}}}

\renewcommand\det{\func{det}}
\newcommand\End{\func{End}}
\newcommand\Ext{\func{Ext}}

\newcommand\Hom{\func{Hom}}
\newcommand\id{\func{id}}

\newcommand\ind{\func{ind}}
\renewcommand\index{\func{index}}

\newcommand\Pic{\func{Pic}}

\newcommand\pr{\func{pr}}
\newcommand\rank{\func{rank}}
\newcommand\Rep{\operatorname{\lie{Rep}}}

\newcommand\sym{\func{symbol}}

\newcommand\trace{\func{trace}}

\newcommand\group[1]{{\text{\bf#1}}}

\newcommand\Spin{\group{Spin}}

\newcommand\SO{\group{SO}}
\newcommand\GL{\group{GL}}

\newcommand\SU{\group{SU}}
\newcommand\U{\group{U}}

\newcommand\G{\group{G}}

\newcommand\abs[1]{\lvert#1\rvert}
\newcommand\bigabs[1]{\bigl\lvert#1\bigr\rvert}

\newcommand\quot[1][\kern.3ex]{/\kern-.7ex/_{\kern-.4ex#1}}
\newcommand\bigquot[1][\,\,]{\big/\kern-.85ex\big/_{\!\!#1}}

\newcommand\powl{[\kern-.3ex[}
\newcommand\powr{]\kern-.3ex]}
\newcommand\bigpowl{\bigl[\kern-.6ex\bigl[}
\newcommand\bigpowr{\bigr]\kern-.6ex\bigr]}

\newcommand\sur{\mathrel{\to\kern-1.8ex\to}}
\newcommand\iso{\mathrel{\hookrightarrow\kern-1.8ex\to}}

\newcommand\longto{\longrightarrow}
\newcommand\longhookrightarrow{\lhook\joinrel\longrightarrow}

\newcommand\longinj{\longhookrightarrow}
\newcommand\longsur{\mathrel{\longrightarrow\kern-1.8ex\to}}
\newcommand\dirac{\textit{\dh}} 
\newcommand\Dirac{\textit{\DH}}
\newcommand\Cl{\group{Cl}}
\newcommand\cliff{\mathrm{cliff}}

\renewcommand\th{\group{th}}
\renewcommand\d{\group{d}}
\newcommand\e{\group{e}}

\newcommand\zerodots%
{\smash{\makebox[0pt]{$_{\lower8pt\hbox{\Large$0$}}%
\ddots^{\lower3pt\hbox{\Large$0$}}$}}}
\newcommand\bigzerodots%
{\smash{\makebox[0pt]{$_{\lower6pt\hbox{\Large$0$}}%
\ddots^{\hbox{\Large$0$}}$}}}

\newcommand\eps{\varepsilon}

\newcommand\pt{{\rm pt}}

\renewcommand\subset{\subseteq}

\newcommand\even{{\mathrm{even}}}
\newcommand\odd{{\mathrm{odd}}}
\newcommand\op{{\mathrm{op}}}

\begin{document}

%%%%%%%%%%%%%%%%%%%%%%%%%%%%%%%%%%%%%%%%%%%%%%%%%%%%%%%%%%%%%%%%%%%%%%%%
%%%%%%%%%%%%%%%%%%%%%%%%%%%%%%%%%%%%%%%%%%%%%%%%%%%%%%%%%%%%%%%%%%%%%%%%

\title[Character formul{\ae} and GKRS multiplets]{Character
formul{\ae} and GKRS multiplets in equivariant K-theory}

\author{Gregory D. Landweber}

\email{gregland@bard.edu}

\address{Department of Mathematics, Bard College, Annandale-on-Hudson,
NY 12504-5000, USA}

\author{Reyer Sjamaar}

\email{sjamaar@math.cornell.edu}

\address{Department of Mathematics, Cornell University, Ithaca, NY
14853-4201, USA}

\subjclass[2000]{19L47 (58J05)}

\date{2012-10-01}

%%%%%%%%%%%%%%%%%%%%%%%%%%%%%%%%%%%%%%%%%%%%%%%%%%%%%%%%%%%%%%%%%%%%%%%%
%%%%%%%%%%%%%%%%%%%%%%%%%%%%%%%%%%%%%%%%%%%%%%%%%%%%%%%%%%%%%%%%%%%%%%%%

\begin{abstract}
Let $G$ be a compact Lie group, $H$ a closed subgroup of maximal rank
and $X$ a topological $G$-space.  We obtain a variety of results
concerning the structure of the $H$-equivariant K-ring $K_H^*(X)$
viewed as a module over the $G$-equivariant K-ring $K_G^*(X)$.  One
result is that the module has a nonsingular bilinear pairing; another
is that the module contains multiplets which are analogous to the
Gross-Kostant-Ramond-Sternberg multiplets of representation theory.
\end{abstract}

%%%%%%%%%%%%%%%%%%%%%%%%%%%%%%%%%%%%%%%%%%%%%%%%%%%%%%%%%%%%%%%%%%%%%%%%
%%%%%%%%%%%%%%%%%%%%%%%%%%%%%%%%%%%%%%%%%%%%%%%%%%%%%%%%%%%%%%%%%%%%%%%%

\maketitle

%%%%%%%%%%%%%%%%%%%%%%%%%%%%%%%%%%%%%%%%%%%%%%%%%%%%%%%%%%%%%%%%%%%%%%%%
%%%%%%%%%%%%%%%%%%%%%%%%%%%%%%%%%%%%%%%%%%%%%%%%%%%%%%%%%%%%%%%%%%%%%%%%

%%%%%%%%%%%%%%%%%%%%%%%%%%%%%%%%%%%%%%%%%%%%%%%%%%%%%%%%%%%%%%%%%%%%%%%%
\section*{Introduction}
%%%%%%%%%%%%%%%%%%%%%%%%%%%%%%%%%%%%%%%%%%%%%%%%%%%%%%%%%%%%%%%%%%%%%%%%

Frobenius \cite[Band~III, pp.\
82-103]{frobenius;gesammelte-abhandlungen} showed how to ``extend'' a
character of a subgroup to a character of the ambient group.  He
considered only finite groups, but his method, which we will refer to
as \emph{formal induction}, was soon adapted in various ways to
infinite and topological groups, and to this day is one of the most
important technical tools of representation theory.

Induction methods for a compact Lie group $G$ and a closed subgroup
$H$ were systematically studied by Bott
\cite{bott;homogeneous-differential} as an application of the index
theory of equivariant elliptic operators.  These index theory methods
work well only if the subgroup $H$ is of maximal rank.  The purpose of
this paper is to make some applications of Bott's work and later work
building on it to the equivariant K-theory of topological $G$-spaces.
Our main results, all valid under the assumption that $H$ is of
maximal rank, are a relative duality theorem and a multiplet theorem.
% We also obtain a number of results on representation rings, one of
% which is a ``Borel-Weil-Bott'' theorem. 
Many of our results are known in important special cases, which
explains the largely expository nature of this paper.

The relative duality theorem, Theorem \ref{theorem;duality}, states
that for every compact $G$-space $X$ the $H$-equivariant K-group
$K_H^*(X)$ is equipped with a $K_G^*(X)$-bilinear nonsingular pairing.
This result contains as a special case a form of Poincar\'e duality
for the K-group of the homogeneous space $G/H$, and it generalizes
results of Pittie \cite{pittie;homogeneous-vector}, Steinberg
\cite{steinberg;pittie}, Shapiro
\cite{shapiro;duality-representation},
\cite{shapiro;algebraic-k-theory}, McLeod
\cite{mcleod;kunneth-equivariant}, and Kazhdan and Lusztig
\cite{kazhdan-lusztig;deligne-langlands-hecke}.  For the special case
of a maximal torus this result can be found in our earlier paper
\cite{harada-landweber-sjamaar;divided-differences-character}.

The multiplet theorem, Theorem \ref{theorem;gkrs}, generalizes a
result of Gross et al.\ \cite{gross-kostant-ramond-sternberg}, which
says that the irreducible representations of $H$ are naturally
partitioned into multiplets which have the property that the
alternating sum of the dimensions of the modules in each multiplet is
zero.  Our theorem states that the $H$-equivariant K-group $K_H^*(X)$
of a compact $G$-space $X$ is similarly divided into multiplets.
Under the forgetful map the classes in each multiplet map to classes
in the ordinary K-group $K^*(X)$ with the property that the
alternating sum is zero.

We have included some elementary background material.  For instance,
because the homogeneous space $G/H$ is not always K-orientable (i.e.\
does not always have a $\Spin^c$-structure), it is necessary to
incorporate an orientation twist in the statement of our theorems, and
in fact also in the definition of the induction map.  This necessity
was overlooked by Bott, which led to a number of minor errors in his
paper.  In \S\,\ref{section;induction} we offer a review of Bott's
work, in which we take care to correct these mistakes, and in
\S\S\,\ref{section;dirac}--\ref{section;properties} we develop the
resulting theory of twisted induction.  In Appendix \ref{section;spin}
we classify $\Spin^c$-structures on $G/H$, which is an easy exercise,
but which to our surprise we could not find in the literature.  In a
separate paper \cite{landweber-sjamaar;spin} we will list examples
where such structures do not exist (the simplest of which is the
Grassmannian of oriented $3$-planes in $\R^7$).

Much of the material in this paper grew out of discussions with Megumi
Harada.  The authors are grateful for her help and encouragement.  We
also thank Eckhard Meinrenken, Peter Landweber and the referee for
several helpful suggestions.

%%%%%%%%%%%%%%%%%%%%%%%%%%%%%%%%%%%%%%%%%%%%%%%%%%%%%%%%%%%%%%%%%%%%%%%%
\tableofcontents
%%%%%%%%%%%%%%%%%%%%%%%%%%%%%%%%%%%%%%%%%%%%%%%%%%%%%%%%%%%%%%%%%%%%%%%%

%%%%%%%%%%%%%%%%%%%%%%%%%%%%%%%%%%%%%%%%%%%%%%%%%%%%%%%%%%%%%%%%%%%%%%%%
\section{Induction maps}\label{section;induction}
%%%%%%%%%%%%%%%%%%%%%%%%%%%%%%%%%%%%%%%%%%%%%%%%%%%%%%%%%%%%%%%%%%%%%%%%

This section is a review of Bott's paper
\cite{bott;homogeneous-differential} on equivariant index theory for
compact homogeneous spaces, where we extend his work by allowing for
certain twists of representation rings.  This amounts to working not
with ordinary representations but with projective representations with
a suitable central character.  Although this is a routine
generalization, the particulars are confusing and not always correctly
recorded in the literature, which is why we provide a detailed
treatment.  The results are most conveniently formulated in terms of
twisted K-theory in the sense of Donovan and Karoubi
\cite{donovan-karoubi;graded-local}, which is reviewed in Appendix
\ref{section;twist}.  (We will only need to twist K-theory by torsion
classes; we do not require Rosenberg's more general version of this
theory.)  An induction map is then a wrong-way or pushforward
homomorphism between equivariant twisted K-groups.  In ordinary
K-theory pushforward maps are commonly defined by means of a Dirac
operator; in twisted K-theory we employ what we call a twisted Dirac
operator.  This term usually means a Dirac operator with coefficients
in a bundle, but in this paper we will use it, following Murray and
Singer \cite{murray-singer;gerbes-index}, for a special type of
transversely elliptic operator in the sense of Atiyah
\cite{atiyah;elliptic-operators-compact-groups}, which lives not on
the base manifold but on a suitable principal bundle over it.  A
definition of pushforwards in twisted K-theory very similar to ours
was given earlier by Mathai, Melrose and
Singer~\cite{mathai-melrose-singer;equivariant-fractional}.

%%%%%%%%%%%%%%%%%%%%%%%%%%%%%%%%%%%%%%%%%%%%%%%%%%%%%%%%%%%%%%%%%%%%%%%%
\subsubsection*{Notation}
%%%%%%%%%%%%%%%%%%%%%%%%%%%%%%%%%%%%%%%%%%%%%%%%%%%%%%%%%%%%%%%%%%%%%%%%

Our notational conventions in
\S\S\,\ref{section;induction}--\ref{section;k} are as follows.  See
also the index at the end.

Let $G$%
\glossary{G@$G$, compact connected Lie group}
be a compact connected Lie group, $\X(G)=\Hom(G,\U(1))$%
\glossary{XG@$\X(G)$, character group $\Hom(G,\U(1))$}
\glossary{U1@$\U(1)$, unit circle}
the character group of $G$, and $R(G)$%
\glossary{R(G)@$R(G)$, representation ring}
the Grothendieck ring of the category of finite-dimensional complex
$G$-modules.  A character $\chi\in\X(G)$ determines a one-dimensional
$G$-module $\C_\chi$%
\glossary{C_chi@$\C_\chi$, one-dimensional $G$-module defined by
character $\chi\in\X(G)$}
and hence a class in $R(G)$, which we denote by~$e^\chi$.%
\glossary{e^chi]@$e^\chi$, class of $\C_\chi$ in $R(G)$}

(We use the term \emph{Grothendieck group} in the sense of
\cite[\S\,2, Theorem~1]{quillen;higher-algebraic-k-theory-I}: if
$\lie{E}$ is an exact category, then the Grothendieck or K-group
$K(\lie{E})$%
\glossary{KE@$K(\lie{E})$, Grothendieck group of category $\lie{E}$}
is the abelian group with one generator $[E]$ for every object $E$ and
one relation $[E_0]-[E_1]+[E_2]=0$ for every short exact sequence
$0\to E_0\to E_1\to E_2\to0$ in $\lie{E}$.  If $\lie{E}$ has tensor
products, then $K(\lie{E})$ is in fact a commutative ring.  In this
paper $\lie{E}$ will always be a category of modules or vector bundles
equipped with the usual notion of an exact sequence.)

We choose once and for all a maximal torus $T$%
\glossary{T@$T$, maximal torus of $G$}
of $G$.  We denote the inclusion map by $j_G\colon T\to G$%
\glossary{j_G@$j_G$, inclusion $T\to G$}
and the Weyl group by $W_G=N_G(T)/T$.%
\glossary{WG@$W_G$, Weyl group of $G$}
The map $\X(T)\to R(T)$ defined by $\chi\mapsto e^\chi$ identifies
$R(T)$ with the group ring $\Z[\X(T)]$.  The restriction homomorphism
$j_G^*\colon R(G)\to R(T)$ induces an isomorphism $R(G)\cong
R(T)^{W_G}$.  (See e.g.\ \cite[\S\,IX.3]{bourbaki;groupes-algebres}.)
In particular the ring $R(G)$ has no zero divisors.  Let
$\ca{R}_G\subset\X(T)$%
\glossary{R_G@$\ca{R}_G$, root system of $G$}
be the root system of $(G,T)$.  We fix a basis $\ca{B}_G$%
\glossary{BG@$\ca{B}_G$, basis of $\ca{R}_G$}
of $\ca{R}_G$, which determines a set of positive roots $\ca{R}_G^+$,%
\glossary{R_G+@$\ca{R}_G^+$, positive roots of $G$}
and we let%
\glossary{rhoG@$\rho_G$, half-sum of positive roots of $G$}
\glossary{alpha@$\alpha$, root of $G$}
$$\rho_G=\frac12\sum_{\alpha\in\ca{R}_G^+}\alpha\in\frac12\X(T)$$
be the half-sum of the positive roots.  (Here we regard the character
group $\X(T)$ as a lattice in the vector space of rational characters
$\X(T)_\Q=\Q\otimes_\Z\X(T)$, and we denote by $\frac12\X(T)$ the
collection of rational characters $\chi\in\X(T)_\Q$ such that
$2\chi\in\X(T)$.)

We denote by $H$%
\glossary{H@$H$, closed subgroup of $G$, from
\S\,\ref{subsection;maximal} onward connected and containing $T$}
a closed subgroup of $G$ with inclusion map $i\colon H\to G$.%
\glossary{i@$i$, inclusion $H\to G$}
From \S\,\ref{subsection;maximal} on we will assume that $H$ is
connected and contains $T$.  We let $M$%
\glossary{M@$M$, homogeneous space $G/H$}
be the homogeneous space $G/H$ and $\bar1=1H$%
\glossary{1@$\bar1$, identity coset $1H\in M$}
the identity coset.  For any $H$-module $V$, we denote by%
\glossary{G^HV@$G\times^HV$, homogeneous vector bundle}
$$G\times^HV$$
the homogeneous vector bundle on $M$ with fibre at $\bar1$ equal to
$V$.  This notation is not to be confused with
\glossary{ACB@$A\times_CB$, fibred product}
$$A\times_CB,$$%
which we use for the fibred product of two objects $A$ and $B$ in some
category (e.g.\ the category of groups) over a third object $C$.  The
tangent space to $M$ at $\bar1$ is denoted by $\m$%
\glossary{m@$\m$, tangent space $T_{\bar1}M$}
and the isotropy representation by%
\glossary{eta@$\eta$, tangent representation $H\to\GL(\m)$}
$$\eta\colon H\longto\GL(\m).$$
The $H$-module $\m$ is naturally isomorphic to $\g/\h$;%
\glossary{gfrak@$\g$, Lie algebra of $G$}
and the tangent bundle $TM$ is naturally isomorphic to $G\times^H\m$.
We choose a $G$-invariant inner product on $\g$, and we identify $\g$
with $\g^*$ and $\m$ with the orthogonal complement of $\h$ in $\g$.

%%%%%%%%%%%%%%%%%%%%%%%%%%%%%%%%%%%%%%%%%%%%%%%%%%%%%%%%%%%%%%%%%%%%%%%%
\subsection{Formal induction}\label{subsection;formal}
%%%%%%%%%%%%%%%%%%%%%%%%%%%%%%%%%%%%%%%%%%%%%%%%%%%%%%%%%%%%%%%%%%%%%%%%
 
Let $R(G)^*=\Hom_\Z(R(G),\Z)$%
\glossary{R(G)*@$R(G)^*$, dual $\Z$-module of $R(G)$}
be the dual abelian group of $R(G)$.  As in
\cite[\S\,1]{bott;homogeneous-differential} we will identify $R(G)^*$
with the group of possibly infinite $\Z$-linear combinations
$\sum_ka_k[E_k]$ of isomorphism classes of irreducible $G$-modules
$E_k$, and also with the Grothendieck group of the category of
admissible $G$-modules.  (A vector in a $G$-module $E$ is called
\emph{$G$-finite} if it is contained in a finite-dimensional
$G$-submodule of $E$.  A $G$-module $E$ is \emph{admissible} if every
vector in $E$ is $G$-finite and every irreducible $G$-module has
finite multiplicity in $E$.)  If $E$ is an admissible $G$-module and
$F$ a finite-dimen\-sional $G$-module, then the dual pairing between
the classes $[E]\in R(G)^*$ and $[F]\in R(G)$ is given by
\begin{equation}\label{equation;pairing}
\langle[E],[F]\rangle=\dim\Hom_G(F,E).
\end{equation}

Let $V$ be a finite-dimen\-sional $H$-module.  We define the
\emph{formal pushforward} or \emph{formal induced module}
$i_!(V)=\ind_H^G(V)$%
\glossary{i"!@$i_"!$, formal induction}% 
\glossary{ind@$\ind_H^G$, formal induction}
of $V$ to be the module of $G$-finite vectors of the $G$-module of
smooth global sections $\Gamma(M,G\times^HV)$.%
\glossary{Gamma@$\Gamma$, smooth global sections functor}
It follows from the Peter-Weyl theorem that $i_!(V)$ is an admissible
$G$-module.  (See \cite[\S\,3]{bott;homogeneous-differential}.)  This
fact enables us to define the \emph{formal pushforward homomorphism}
or \emph{formal induction map}
$$i_!\colon R(H)\to R(G)^*$$
by $i_!([V])=[i_!(V)]$.  This map is $R(G)$-linear in the sense that
$i_!(i^*(b)\phi)=bi_!(\phi)$ for all $b\in R(G)$ and $\phi\in R(H)^*$,
where $i^*\colon R(G)\to R(H)$ is the restriction homomorphism induced
by the inclusion $i\colon H\to G$.  Formal induction satisfies
Frobenius reciprocity, which can be stated by saying that $i_!$ is the
composition of the two maps
\begin{equation}\label{equation;composition}
i_!\colon R(H)\longto R(H)^*\overset{^t\mspace{-1mu}i^*}\longto
R(G)^*.
\end{equation}
Here $R(H)\to R(H)^*$ is the natural inclusion and
$^t\mspace{-1mu}i^*$ is the transpose of $i^*$.  (See
\cite[\S\,2]{bott;homogeneous-differential} or
\cite[\S\,2]{segal;representation-ring}.)  

A less desirable property is that the module $i_!(V)$ is often
infinite-dimensional even if $V$ is finite-dimensional.  We wish to
look for induction maps that preserve finite-dimensionality, which
forces us to give up Frobenius reciprocity.  Accordingly, by an
\emph{induction map} we will mean any $R(G)$-linear map $R(H)\to
R(G)$, i.e.\ any element of%
\glossary{R(H)v@$R(H)^\vee$, dual $R(G)$-module of $R(H)$}
$$R(H)^\vee=\Hom_{R(G)}(R(H),R(G)),$$
the dual of the $R(G)$-module $R(H)$.

%%%%%%%%%%%%%%%%%%%%%%%%%%%%%%%%%%%%%%%%%%%%%%%%%%%%%%%%%%%%%%%%%%%%%%%%
\subsection{Twisted induction}\label{subsection;twist}
%%%%%%%%%%%%%%%%%%%%%%%%%%%%%%%%%%%%%%%%%%%%%%%%%%%%%%%%%%%%%%%%%%%%%%%%

We require a slightly more general notion involving projective
representations.  Part of the following material is taken from
\cite{landweber;twisted}.  Let%
\glossary{sigma@$\sigma$, central extension of $G$ by $\U(1)$}
\glossary{G^sigma@$G^{(\sigma)}$, central extension of $G$ by $\U(1)$}
\begin{equation}\label{equation;extension}
\sigma\colon1\longto\U(1)\longto G^{(\sigma)}\longto G\longto1
\end{equation}
be a central extension of $G$ by the circle $\U(1)$.  (A notational
convention: the label $\sigma$ will refer to the exact sequence
\eqref{equation;extension} as a whole, but when there is no danger of
confusion, we will speak of ``the extension $G^{(\sigma)}$.'')  A
complex $G^{(\sigma)}$-module $V$ has \emph{central character} or
\emph{level} $k\in\Z$ if the subgroup $\U(1)$ acts on $V$ by $z\cdot
v=z^kv$.  Let $\Rep(G,\sigma)$ be the category of finite-dimensional
complex $G^{(\sigma)}$-modules of level~$1$.  We call the Grothendieck
group $R(G,\sigma)$%
\glossary{R(Gsigma)@$R(G,\sigma)$, twisted representation module}
of $\Rep(G,\sigma)$ the \emph{$\sigma$-twisted representation module}
of $G$.  This is the $G$-equivariant twisted K-group of a point; see
Example \ref{example;equivariant-extension} in Appendix
\ref{section;twist}.  In the case of the trivial extension $\sigma=0$
(which is defined by $G^{(0)}=\U(1)\times G$) there is an evident
equivalence of categories $\Rep(G,0)\to\Rep(G)$, which identifies
$R(G,0)$ with $R(G)$.

Recall that the \emph{sum} of two central extensions $\sigma$ and
$\upsilon$ of $G$ by $\U(1)$ is the central extension
$\sigma+\upsilon$ of $G$ by $\U(1)$ defined by
$$G^{(\sigma+\upsilon)}=G^{(\sigma,\upsilon)}/K.$$
Here $G^{(\sigma,\upsilon)}$ denotes the fibred product
$G^{(\sigma)}\times_G G^{(\upsilon)}$ and $K$ is a copy of $\U(1)$
anti-diagonally embedded in $G^{(\sigma,\upsilon)}$.  The extension
\emph{opposite} to $\sigma$ is the extension $-\sigma$ obtained by
precomposing the inclusion $\U(1)\to G^{(\sigma)}$ with the
automorphism $z\mapsto z^{-1}$ of $\U(1)$.  The tensor product functor
$$\Rep(G,\sigma)\times\Rep(G,\upsilon)\longto\Rep(G,\sigma+\upsilon)$$
induces a bi-additive map
\begin{equation}\label{equation;sum}
R(G,\sigma)\times R(G,\upsilon)\longto R(G,\sigma+\upsilon).
\end{equation}
We will use multiplicative notation for this map; so if $a=[E]\in
R(G,\sigma)$ and $b=[F]\in R(G,\upsilon)$, then $ab=[E\otimes F]\in
R(G,\sigma+\upsilon)$.  By taking $\sigma=0$, we see that
$R(G,\upsilon)$ is an $R(G)$-module for all $\upsilon$.  With respect
to this module structure, the multiplication law \eqref{equation;sum}
is $R(G)$-bilinear.

\begin{remark}\label{remark;ext}
If two extensions $\sigma$ and $\upsilon$ are equivalent, then the
twisted representation modules $R(G,\sigma)$ and $R(G,\upsilon)$ are
isomorphic, but the isomorphism depends on the choice of the
equivalence.  For this reason we do \emph{not} identify $R(G,\sigma)$
with $R(G,\upsilon)$ unless an explicit equivalence
$\sigma\sim\upsilon$ has been specified.
\end{remark}

For any extension $G^{(\sigma)}$ as in \eqref{equation;extension} and
any Lie group homomorphism $f\colon L\to G$ we can form the
\emph{pullback} extension
$$
f^*\sigma\colon1\longto\U(1)\longto L^{(f^*\sigma)}\longto L\longto1,
$$
where $L^{(f^*\sigma)}=G^{(\sigma)}\times_GL$ is the fibred product of
$G^{(\sigma)}$ and $L$ with respect to the homomorphisms
$G^{(\sigma)}\to G$ and $f\colon L\to G$.  There is an induced
$R(L)$-linear homomorphism
$$
f^*\colon R(G,\sigma)\longto R(L,f^*\sigma).
$$

For instance, taking $f$ to be the inclusion $i\colon H\to G$ we get a
central extension $H^{(i^*\sigma)}$ of $H$, which is simply the
preimage of $H$ under the projection $G^{(\sigma)}\to G$.  To
economize on notation we will write $H^{(\sigma)}$ instead of
$H^{(i^*\sigma)}$.  The map $G^{(\sigma)}\to G$ descends to a
diffeomorphism $G^{(\sigma)}/H^{(\sigma)}\cong M$, which allows us to
identify $M$ with the homogeneous space $G^{(\sigma)}/H^{(\sigma)}$.

\begin{lemma}\label{lemma;twist-module}
Let $\sigma$ and $\upsilon$ be central extensions of $G$ by $\U(1)$.
\begin{enumerate}
\item\label{item;isotype}
The $R(G)$-module $R\bigl(G^{(\sigma)}\bigr)$ is the direct sum of the
submodules $R(G,k\sigma)$ over all levels $k\in\Z$.  Each summand
$R(G,k\sigma)$ is nonzero.
\item\label{item;extension-invariant}
The group $T^{(\sigma)}$ is a maximal torus of $G^{(\sigma)}$; the
submodule $R(T,\sigma)$ of $R(T^{(\sigma)})$ is preserved by the
$W_G$-action; and the restriction homomorphism $R(G,\sigma)\to
R(T,\sigma)$ is an isomorphism onto $R(T,\sigma)^{W_G}$.
\item\label{item;extension-sum}
Let $a\in R(G,\sigma)$ and $b\in R(G,\upsilon)$.  If $ab\in
R(G,\sigma+\upsilon)$ is equal to $0$, then $a=0$ or $b=0$.
\end{enumerate}
\end{lemma}

\begin{proof}
\eqref{item;isotype} and \eqref{item;extension-invariant} are special
cases of Lemma \ref{lemma;extension} in the Appendix (where
$G^{(\sigma)}$ is denoted by $\check{G}$ and $R(G,k\sigma)$ by
$R^k(\check{G})$).  The group
$G^{(\sigma,\upsilon)}=G^{(\sigma)}\times_G G^{(\upsilon)}$ is a
central $\U(1)$-extension of $G^{(\sigma)}$ as well as of
$G^{(\upsilon)}$.  Therefore, by~\eqref{item;isotype}, the modules
$R(G,\sigma)$ and $R(G,\upsilon)$ are naturally isomorphic to
submodules of $R(G^{(\sigma,\upsilon)})$.  Under this isomorphism, the
bilinear map \eqref{equation;sum} corresponds to multiplication in
$R(G^{(\sigma,\upsilon)})$.  Since $G^{(\sigma,\upsilon)}$ is a
central extension of $G$ by the torus $\U(1)^2$, it is connected, and
therefore $R(G^{(\sigma,\upsilon)})$ has no zero divisors.
\end{proof}

\begin{example}\label{example;torus}
Suppose $G=T$ is a torus.  Choose a character $\mu$ of $T^{(\sigma)}$
of level $1$.  Then the character group $\X\bigl(T^{(\sigma)}\bigr)$
is the direct sum of $\X(T)$ and $\Z\mu$, and therefore
$R\bigl(T^{(\sigma)}\bigr)=R(T)[e^\mu,e^{-\mu}]$ is a Laurent
polynomial algebra over $R(T)$ in one variable.  The level $k$
submodule is $R(T,k\sigma)=R(T)e^{k\mu}$, which is a free
$R(T)$-module of rank $1$ on the generator $e^{k\mu}$.
\end{example}

\begin{lemma}\label{lemma;sections}
Let $\sigma$ be a central extension of $G$ by $\U(1)$ and let $V$ be
an $H^{(\sigma)}$-module of level~$1$.  Then the space of smooth
sections of the $G^{(\sigma)}$-homogeneous vector bundle
$G^{(\sigma)}\times^{H^{(\sigma)}}V$ over $M$ is a
$G^{(\sigma)}$-module of level~$1$.
\end{lemma}

\begin{proof}
Put $E=G^{(\sigma)}\times^{H^{(\sigma)}}V$.  A section of $E$ is a
function $f\colon G^{(\sigma)}\to V$ satisfying $f(gh^{-1})=h\cdot
f(g)$ for all $h\in H^{(\sigma)}$.  The action of an element $k\in
G^{(\sigma)}$ on $f$ is defined by $(k\cdot f)(g)=f(k^{-1}g)$ for
$g\in G^{(\sigma)}$.  In particular, a central element
$z\in\U(1)\subset H^{(\sigma)}\subset G^{(\sigma)}$ acts by
$$
(z\cdot f)(g)=f(z^{-1}g)=z\cdot f(g)=zf(g),
$$
where the last equality follows from the assumption that $V$ is an
$H^{(\sigma)}$-module of level $1$.  Therefore $z\cdot f=zf$.
\end{proof}

\begin{definition}\label{definition;induction}
A \emph{twisted induction map} is an element of
$$\Hom_{R(G)}\bigl(R(H,\tau),R(G,\sigma)\bigr),$$
where $\sigma$ is a central extension of $G$ by $\U(1)$ and $\tau$%
\glossary{tau@$\tau$, central extension of $H$ by $\U(1)$}
\glossary{H@$H^{(\tau)}$, central extension of $H$ by $\U(1)$}
is a central extension of $H$ by $\U(1)$.
\end{definition}

%%%%%%%%%%%%%%%%%%%%%%%%%%%%%%%%%%%%%%%%%%%%%%%%%%%%%%%%%%%%%%%%%%%%%%%%
\subsection{Elliptic operators}\label{subsection;elliptic}
%%%%%%%%%%%%%%%%%%%%%%%%%%%%%%%%%%%%%%%%%%%%%%%%%%%%%%%%%%%%%%%%%%%%%%%%

We will obtain twisted induction maps from trans\-verse\-ly elliptic
differential operators.  To motivate Definition
\ref{definition;twist-operator} below let us first consider the
untwisted case.  Let $P$ be compact $G\times H$-manifold equipped with
a $\Z/2\Z$-graded $G\times H$-equivariant vector bundle $E=E^0\oplus
E^1$ and a $G\times H$-equivariant differential operator
$D\colon\Gamma(P,E^0)\to\Gamma(P,E^1)$ which is transversely elliptic
with respect to $H$.  (Recall that ``$\Gamma$'' denotes smooth
sections.)  There are a number of equivalent ways to define the
\emph{equivariant index} of $D$.  We will define it as the module of
$G$-finite vectors of $\ker(D)$ minus the module of $G$-finite vectors
of $\ker(D^*)$.  By the discussion in \cite[Lecture~2,
p.~17]{atiyah;elliptic-operators-compact-groups} the index is an
element of
\begin{equation}\label{equation;index}
\index(D)\in R(H)^*\otimes_\Z R(G),
\end{equation}
and so can be viewed as a $\Z$-linear map from $R(H)$ to $R(G)$.  It
follows from
\cite[Theorem~3.5]{atiyah;elliptic-operators-compact-groups} (see also
Lemma \ref{lemma;elliptic}\eqref{item;twist} below) that this map is
in fact $R(G)$-linear and is therefore an induction map in our sense.

Now suppose that the $H$-action on $P$ is free, so that the quotient
$X=P/H$ is a manifold and the quotient map $\pr\colon P\to X$ is a
principal $H$-bundle.  Then the quotient $F=E/H$ is a $G$-equivariant
vector bundle on $X$, and we have a natural identification
$\Gamma(X,F)\cong\Gamma(P,E)^H$.  It follows that $D$ restricts to an
operator $D_0\colon\Gamma(X,F^0)\to\Gamma(X,F^1)$, which is a
$G$-equivariant elliptic operator on $X$ called the operator
\emph{induced} by $D$.  As shown in
\cite[Lecture~3]{atiyah;elliptic-operators-compact-groups}, every
$G$-equivariant elliptic operator $D_0$ on $X$ is induced by a
$G\times H$-equivariant differential operator $D$ on $P$ which is
transversely elliptic with respect to $H$.  There are many possible
choices for such an operator $D$, but its principal symbol satisfies
$$
\sym(D)|T_H^*P=\pr^*\sym(D_0)
$$
(see the proof of
\cite[Theorem~3.1]{atiyah;elliptic-operators-compact-groups}), where
$T_H^*P$ denotes the horizontal cotangent bundle of the principal
$H$-bundle $P$, and $\sym(D)|T_H^*P$ denotes the restriction of the
symbol to $T_H^*P$.  It follows from this that the equivariant index
of $D$, and hence the associated induction map, are uniquely
determined by $D_0$.  Since we are primarily interested not in $D$ but
in its index, we will sometimes allow ourselves to blur the
distinction between the operators $D$ and $D_0$.

To get twisted induction maps we will modify this set-up by replacing
$G$ with a central extension $G^{(\sigma)}$ and $H$ with a central
extension $H^{(\tau)}$.  We will take $P$ to be a $G^{(\sigma)}\times
H^{(\sigma)}$-manifold on which $H^{(\sigma)}$ acts freely.  However,
the bundle $E$ and the operator $D\colon\Gamma(P,E^0)\to\Gamma(P,E^1)$
will not be equivariant with respect to $G^{(\sigma)}\times
H^{(\sigma)}$, but with respect to $G^{(\sigma)}\times
H^{(\sigma,\tau)}$, where
$$
H^{(\sigma,\tau)}=H^{(\sigma)}\times_HH^{(\tau)}.
$$
Therefore $H^{(\sigma)}$ does not act on $E$, and so $D$ does not
descend to an operator on the quotient $P/H^{(\sigma)}$ (except when
$\tau=0$).  Nevertheless we will call $D$ a ``twisted'' operator on
$P/H^{(\sigma)}$.

Specifically, we take $P=G^{(\sigma)}$, so that $P/H^{(\sigma)}\cong
G/H=M$.  We view the group $H^{(\sigma,\tau)}$ as an iterated
$\U(1)$-central extension of $H$, namely
$$H^{(\sigma,\tau)}=\bigl(H^{(\sigma)}\bigr)^{(\tau)},$$
the pullback of the central extension $\tau$ via the homomorphism
$H^{(\sigma)}\to H$.  Thus, by an $H^{(\sigma,\tau)}$-module of level
$1$ we mean an $H^{(\sigma,\tau)}$-module on which the second factor
of the central torus $\U(1)\times\U(1)$ acts with weight~$1$.  We take
$E$ to be a product bundle $E=G^{(\sigma)}\times U$, where $U$ is a
$\Z/2\Z$-graded $H^{(\sigma,\tau)}$-module $U$ of level $1$.  Note
that $E$ is $G^{(\sigma)}\times H^{(\sigma,\tau)}$-equivariant, where
$G^{(\sigma)}$ acts on the base $G^{(\sigma)}$ by left multiplication,
and $H^{(\sigma,\tau)}$ acts on the base by right multiplication via
the homomorphism $H^{(\sigma,\tau)}\to H^{(\sigma)}$ and linearly on
the fibre $U$.

\begin{definition}\label{definition;twist-operator}
Let $U$ be a $\Z/2\Z$-graded $H^{(\sigma,\tau)}$-module $U$ of level
$1$.  A \emph{$(\sigma,\tau)$-twisted equivariant elliptic
differential operator} on $M$ (or a \emph{$(\sigma,\tau)$-twisted
operator} for short) is a $G^{(\sigma)}\times
H^{(\sigma,\tau)}$-equivariant linear differential operator
$$
D\colon\Gamma\bigl(G^{(\sigma)},G^{(\sigma)}\times
U^0\bigr)\longto\Gamma\bigl(G^{(\sigma)},G^{(\sigma)}\times U^1\bigr)
$$
on $G^{(\sigma)}$ which is transversely elliptic relative to
$H^{(\sigma,\tau)}$.
\end{definition}

If $\tau=0$ is the trivial extension, then a $(\sigma,\tau)$-twisted
operator descends to a $G^{(\sigma)}$-equivariant elliptic
differential operator on $M$ in the usual sense.

Following \eqref{equation;index} we view the index of a
$(\sigma,\tau)$-twisted operator $D$ as a $\Z$-linear map
$\index(D)\colon R\bigl(H^{(\sigma,\tau)}\bigr)\to
R\bigl(G^{(\sigma)}\bigr)$.

\begin{lemma}\label{lemma;elliptic}
Let $D$ be a $(\sigma,\tau)$-twisted operator.
\begin{enumerate}
\item\label{item;twist}
The index of $D$ is $R\bigl(G^{(\sigma)}\bigr)$-linear and maps
$R(H,\tau-\sigma)$ to $R(G,\sigma)$.
\item\label{item;elliptic}
Let $V$ be an $H^{(\tau-\sigma)}$-module of level $1$.  Let $E_V$ be
the $\Z/2\Z$-graded $G^{(\sigma)}$-homogeneous vector bundle
$G^{(\sigma)}\times^{H^{(\sigma)}}(U\otimes V^*)$ over $M$.  Then $D$
descends to a $G^{(\sigma)}$-equivariant elliptic differential
operator
$$
D_V\colon\Gamma\bigl(M,E_V^0\bigr)\longto\Gamma\bigl(M,E_V^1\bigr)
$$
with the property that $\index(D_V)=\langle\index(D),[V]\rangle$.
\end{enumerate}
\end{lemma}

\begin{proof}
The $R\bigl(G^{(\sigma)}\bigr)$-linearity is a special case of the
multiplicativity property of the index,
\cite[Theorem~3.5]{atiyah;elliptic-operators-compact-groups}.  (Take
$X=\pt$ and $Y=G$ in the statement of that theorem.)  Let $V$ be an
$H^{(\tau-\sigma)}$-module of level $1$.  Then for $p=0$, $1$ the
tensor product $U^p\otimes V^*$ is an $H^{(\sigma)}$-module of level
$1$, and the bundle $G^{(\sigma)}\times(U^p\otimes V^*)$ over
$G^{(\sigma)}$ is $G^{(\sigma)}\times H^{(\sigma)}$-equivariant.  This
shows that the vector bundle $E_V^p$ is well-defined.  It follows from
\eqref{equation;pairing} that
\begin{equation}\label{equation;family-index}
\langle\index(D),[V]\rangle
=\bigl[\Hom_{H^{(\sigma,\tau)}}(V,\ker(D))\bigr]
-\bigl[\Hom_{H^{(\sigma,\tau)}}(V,\ker(D^*))\bigr]\in
R\bigl(G^{(\sigma)}\bigr).
\end{equation}
Now $\Hom_{H^{(\sigma,\tau)}}(V,\ker(D))\cong
(V^*\otimes\ker(D))^{H^{(\sigma,\tau)}}$ is a $G^{(\sigma)}$-submodule
of
\begin{equation}\label{equation;sections}
\bigl(V^*\otimes\Gamma\bigl(G^{(\sigma)},G^{(\sigma)}\times
U^0\bigr)\bigr)^{H^{(\sigma,\tau)}}\cong
\Gamma\bigl(G^{(\sigma)},G^{(\sigma)}\times(U^0\otimes
V^*)\bigr)^{H^{(\sigma,\tau)}}\cong\Gamma(M,E^0),
\end{equation}
and is therefore of level $1$ by Lemma \ref{lemma;sections}.
Similarly, $\Hom_{H^{(\sigma,\tau)}}(V,\ker(D^*))$ is isomorphic to a
$G^{(\sigma)}$-submodule of $\Gamma(M,E^1)$ and so is of level $1$.
It now follows from \eqref{equation;family-index} that
$\langle\index(D),[V]\rangle\in R(G,\sigma)$, which proves
\eqref{item;twist}.  The operator
$$
\id_{V^*}\otimes D\colon
V^*\otimes\Gamma\bigl(G^{(\sigma)},G^{(\sigma)}\times U^0\bigr)\longto
V^*\otimes\Gamma\bigl(G^{(\sigma)},G^{(\sigma)}\times U^1\bigr)
$$
is $H^{(\sigma,\tau)}$-equivariant.  Taking
$H^{(\sigma,\tau)}$-invariants on both sides and composing with the
natural isomorphism \eqref{equation;sections} we get an operator
$D_V\colon\Gamma\bigl(M,E_V^0\bigr)\longto\Gamma\bigl(M,E_V^1\bigr)$,
which is a $G^{(\sigma)}$-equivariant differential operator because
$D$ is.  The principal symbol of $D$ is a morphism of
$G^{(\sigma)}\times H^{(\sigma,\tau)}$-equivariant bundles over
$T_{H^{(\sigma,\tau)}}^*G^{(\sigma)}=G^{(\sigma)}\times\m$,
\begin{equation}\label{equation;symbol}
\sym(D)\colon G^{(\sigma)}\times\m\times U^0\longto
G^{(\sigma)}\times\m\times U^1,
\end{equation}
which is an isomorphism off the zero section
$G^{(\sigma)}\times\m\times\{0\}$.  This transversely elliptic symbol
induces an elliptic symbol
$$
\bigl(G^{(\sigma)}\times\m\bigr)\times^{H^{(\sigma)}}(U^0\otimes
V^*)\longto
\bigl(G^{(\sigma)}\times\m\bigr)\times^{H^{(\sigma)}}(U^1\otimes V^*)
$$
on $T^*M=\bigl(G^{(\sigma)}\times\m\bigr)\big/H^{(\sigma)}$, which is
the symbol of $D_V$.  Hence $D_V$ is elliptic.  If $V$ is irreducible,
then under the isomorphism \eqref{equation;sections} the kernel of
$D_V$ corresponds to the $V$-isotypical subspace of $\ker(D)$.
Similarly, $\ker(D_V^*)$ corresponds to the $V$-isotypical subspace of
$\ker(D^*)$.  Therefore, by \eqref{equation;family-index},
$\index(D_V)=\langle\index(D),[V]\rangle$ for irreducible $V$.  The
index of $D_V$ is additive with respect to $V$, so we conclude that
$\index(D_V)=\langle\index(D),[V]\rangle$ for general $V$.
\end{proof}

\begin{definition}\label{definition;elliptic-induction}
Let $D$ be a $(\sigma,\tau)$-twisted operator.  The \emph{associated
induction map} is the map
\glossary{iD@$i_D$, induction defined by operator $D$}
$$i_D\in\Hom_{R(G)}\bigl(R(H,\sigma-\tau),R(G,\sigma)\bigr)$$
defined by $i_D([V])=\index(D_{V^*})=\langle\index(D),[V^*]\rangle$.
\end{definition}

Let $BM$%
\glossary{BM@$BM$, unit ball bundle of cotangent bundle $T^*M$}
be the unit ball bundle and $SM$%
\glossary{SM@$SM$, unit sphere bundle of cotangent bundle $T^*M$}
the unit sphere bundle of the cotangent bundle $T^*M$.  The principal
symbol \eqref{equation;symbol} of a $(\sigma,\tau)$-twisted operator
$D$ defines a class $[D]$ in the twisted relative K-group of the pair
$(BM,SM)$ called the \emph{symbol class},
\glossary{D@$[D]$, symbol class of operator $D$}
\begin{equation}\label{equation;symbol-class}
[D]\in K_{G^{(\sigma)}}^0(BM,SM,\tau).
\end{equation}
(The relative twisted K-group occurring here is explained in Appendix
\ref{subsection;relative}.)  Let $\zeta\colon M\to T^*M$%
\glossary{zeta@$\zeta$, zero section $M\to T^*M$}
be the zero section of $T^*M$.  We will call the restriction of the
symbol class to the zero section,%
\glossary{eD@$\e(D)$, Euler class of operator $D$}
\begin{equation}\label{equation;euler}
\e(D)=\zeta^*([D])=[U^0]-[U^1]\in K_{G^{(\sigma)}}^0(M,\tau)\cong
K_{H^{(\sigma)}}^0(\pt,\tau)\cong R\bigl(H^{(\sigma)},\tau\bigr),
\end{equation}
the \emph{($G$-equivariant) Euler class} of $D$.  Note that the Euler
class depends only on the coefficient module $U$ of $D$.  (If $D$ is
the twisted $\Spin^c$ Dirac operator on $M$, which is defined in
\S\,\ref{subsection;thom}, then $\e(D)$ is the usual equivariant Euler
class of $M$.  The isomorphisms in \eqref{equation;euler} are
explained in Appendix \ref{subsection;bundle}, Examples
\ref{example;equivariant-extension}--\ref{example;induced-gerbe}.)

The following statement, which summarizes and extends to the twisted
case results of Bott \cite{bott;homogeneous-differential}, says that
an induction map defined by a twisted operator depends only on the
Euler class.  Moreover, all such induction maps vanish if the subgroup
is not of maximal rank.  The situation in the maximal rank case is
diametrically opposite and is described in Theorem
\ref{theorem;dirac-induction} below.  For simplicity we denote the
inclusion $H^{(\sigma)}\to G^{(\sigma)}$ by $i$ and the formal
induction map $R\bigl(H^{(\sigma)}\bigl)\to
R\bigl(G^{(\sigma)}\bigr)^*$ by $i_!$.  It follows from Lemma
\ref{lemma;sections} that $i_!$ maps $R(H,\sigma)$ to $R(G,\sigma)^*$.

\begin{theorem}\label{theorem;linear}
Let $D$ be a $(\sigma,\tau)$-twisted operator on $M$ and let
$i_D\colon R(H,\sigma-\tau)\to R(G,\sigma)$ be the associated
induction map.
\begin{enumerate}
\item\label{item;push}
$i_D(a)=i_!(\e(D)a)$ for all $a\in R(H,\sigma-\tau)$.
\item\label{item;zero}
$i_D=0$ if $H$ does not contain a maximal torus of $G$.
\item\label{item;zero-twist}
If $\tau=0$, then $i_D(1)=i_!(\e(D))$ and $i_Di^*(b)=i_D(1)b$ for all
$b\in R(G,\sigma)$.
\end{enumerate}
\end{theorem}

\begin{proof}
Let $V$ be an $H^{(\sigma-\tau)}$-module of level $1$.  The Euler
class of the operator $D_{V^*}$ defined in Lemma
\ref{lemma;elliptic}\eqref{item;elliptic} is
$$
\e({D_{V^*}})=[U^0\otimes V]-[U^1\otimes V]=\e(D)\cdot[V]\in
R(H,\sigma).
$$
By \cite[Theorem~I]{bott;homogeneous-differential}, the element
$i_!(\e({D_{V^*}}))\in R(G,\sigma)^*$ is in $R(G,\sigma)$, and the
index of $D_{V^*}$ is equal to $i_!(\e({D_{V^*}}))$.  Therefore, by
Lemma \ref{lemma;elliptic}\eqref{item;elliptic},
$$
i_D([V])=\index(D_{V^*})=i_!(\e(D)\cdot[V]),
$$
which proves \eqref{item;push}.  \eqref{item;zero} follows from
\cite[Theorem~II]{bott;homogeneous-differential} if $H$ is connected.
In the general case we argue as in
\cite[\S\,2]{segal;representation-ring}: the character of the virtual
$G^{(\sigma)}$-module $i_D(a)$ is given by the Lefschetz fixed point
formula of Atiyah and Bott
\cite{atiyah-bott;lefschetz-fixed-elliptic-complex}, but if $H$ does
not contain a maximal torus of $G$, then $H^{(\sigma)}$ does not
contain a maximal torus of $G^{(\sigma)}$, so $T^{(\sigma)}$ acts on
$M$ without fixed points, so $i_D(a)=0$.  Finally,
\eqref{item;zero-twist} follows immediately from \eqref{item;push} and
the $R(G)$-linearity of $i_D$ (Lemma
\ref{lemma;elliptic}\eqref{item;twist}).
\end{proof}

%%%%%%%%%%%%%%%%%%%%%%%%%%%%%%%%%%%%%%%%%%%%%%%%%%%%%%%%%%%%%%%%%%%%%%%%
\subsection{Twisted Dirac operators}\label{subsection;dirac}
%%%%%%%%%%%%%%%%%%%%%%%%%%%%%%%%%%%%%%%%%%%%%%%%%%%%%%%%%%%%%%%%%%%%%%%%

The operators of most interest to us will be twisted versions of
Atiyah and Singer's generalized Dirac operators.  We state the
definition here and we will compute the Euler class of some twisted
Dirac operators in
% Example \ref{example;euler}, 
Lemma \ref{lemma;weight-euler}\eqref{item;euler-twist} and Proposition
\ref{proposition;principal-chern-euler}.

Let $X$ be an oriented $G$-manifold equipped with an invariant
Riemannian metric and let%
\glossary{Cl@$\Cl(E)$, Clifford algebra of a vector space or vector
bundle $E$}
$\Cl(X)=\Cl(T^*X)$ be the Clifford bundle of
the cotangent bundle $T^*X$.  Let $\pr\colon P\to X$ be a
$G$-equivariant principal bundle over $X$ with structure group $H$.
Let $\tau$ be a central extension of $H$ by $\U(1)$ and let
$E=E^0\oplus E^1$ be a $\Z/2\Z$-graded $G\times
H^{(\tau)}$-equivariant vector bundle over $P$ on which the central
circle $\U(1)$ of $H^{(\tau)}$ acts by fibrewise scalar
multiplication.  (In the language of Appendix \ref{section;twist}, $E$
is a $G$-equivariant twisted vector bundle over $X$.)  Suppose also
that $E$ is a graded module for the bundle of Clifford algebras
$\pr^*\Cl(X)$ and is equipped with a connection $\nabla$.  Assume that
the module structure and the connection are $G\times
H^{(\tau)}$-equivariant.  Choose a $G$-invariant connection
$\theta\in\Omega^1(P,\h)$ on the $H$-bundle $P$ and let $p\colon
T^*P\to\pr^*T^*X$ be the associated projection.  Let
$\cliff\colon\pr^*T^*X\times E^0\to E^1$ the Clifford multiplication.
Form a first-order operator on $P$,
\begin{equation}\label{equation;twisted-dirac}
D\colon\Gamma(P,E^0)\overset\nabla\longto\Gamma(P,T^*P\otimes E^0)
\overset{p}\longto\Gamma(P,\pr^*T^*X\otimes
E^0)\overset{\cliff}\longto \Gamma(P,E^1).
\end{equation}
This is a $G\times H^{(\tau)}$-equivariant operator on $P$ which is
transversely elliptic relative to $H^{(\tau)}$.

Let $\sigma$ be a central extension of $G$ by $\U(1)$.  Replacing $G$
with $G^{(\sigma)}$ and $H^{(\tau)}$ with $H^{(\sigma,\tau)}$ in this
definition, we obtain a $G^{(\sigma)}\times
H^{(\sigma,\tau)}$-equivariant operator $D$ on $P$ which is
transversely elliptic relative to $H^{(\sigma,\tau)}$.  We refer to
$D$ as the \emph{twisted Dirac operator} determined by the equivariant
Clifford module $E$ and the connection $\theta$.  Its symbol defines a
class $[D]\in K_{G^{(\sigma)}}^0(BT^*X,ST^*X,\tau)$, which is
independent of the choice of $\theta$.

As a special case we take $X$ to be the homogeneous space
$M=G/H=G^{(\sigma)}/H^{(\sigma)}$ and $P$ to be the group
$G^{(\sigma)}$.  The Clifford bundle $\pr^*\Cl(M)$ is then the product
bundle $G^{(\sigma)}\times\Cl(\m)$.  For $E$ we take a product bundle
$E=G^{(\sigma)}\times U$, where $U$ is a $\Z/2\Z$-graded
$H^{(\sigma,\tau)}$-module of level~$1$.  To make $E$ an equivariant
$\pr^*\Cl(M)$ bundle we assume that $U$ is an $H$-equivariant
$\Cl(\m)$-module.  The orthogonal splitting $\g=\h\oplus\m$ defines a
connection $\theta\in\Omega^1\bigl(G^{(\sigma)},\h^{(\sigma)}\bigr)$
on the principal $H^{(\sigma)}$-bundle $G^{(\sigma)}\to M$.  A
connection $\nabla=d+\theta$ on the bundle $E$ is then defined by
$$\nabla_vf=df(v)+\theta(v)\cdot f$$
for all tangent vectors $v$ to $G^{(\sigma)}$ and smooth functions
$f\colon G^{(\sigma)}\to U$.  In this formula
$\theta(v)\in\h^{(\sigma)}$ acts on $f$ by combining the infinitesimal
right multiplication action on $G^{(\sigma)}$ with the action
$$
\h^{(\sigma)}\longto\h\longto\lie{o}(\m)\longto\lie{cl}(\m)
\longto\lie{gl}(U)
$$
on $U$.  (Here $\lie{cl}(\m)$ is the Lie algebra of the Clifford
group, i.e.\ the vector space $\Cl(\m)$ equipped with the commutator
bracket $[x,y]=xy-yx$.)  Thus the data $U$ and $\theta$ define a
twisted Dirac operator on $M$, which is a $(\sigma,\tau)$-twisted
equivariant elliptic differential operator in the sense of Definition
\ref{definition;twist-operator}.  By Theorem \ref{theorem;linear} the
induction map $i_D$ associated with $D$ depends only on the Euler
class $\e(D)=[U^0]-[U^1]$.

%%%%%%%%%%%%%%%%%%%%%%%%%%%%%%%%%%%%%%%%%%%%%%%%%%%%%%%%%%%%%%%%%%%%%%%%
\subsection{The maximal rank case}\label{subsection;maximal}
%%%%%%%%%%%%%%%%%%%%%%%%%%%%%%%%%%%%%%%%%%%%%%%%%%%%%%%%%%%%%%%%%%%%%%%%

In view of Theorem \ref{theorem;linear}\eqref{item;zero} we assume for
the remainder of the paper that the subgroup $H$ of $G$ contains $T$.
For simplicity we will also assume $H$ to be connected.  Then the
homogeneous space $M=G/H$ is simply connected (\cite[Ch.~IX,
\S\,2.4]{bourbaki;groupes-algebres}).  In this section we will
calculate the pushforward homomorphism defined by a twisted elliptic
operator on $M$ by means of the Lefschetz formula of Atiyah and Bott.

%%%%%%%%%%%%%%%%%%%%%%%%%%%%%%%%%%%%%%%%%%%%%%%%%%%%%%%%%%%%%%%%%%%%%%%%
\subsubsection*{Notation}
%%%%%%%%%%%%%%%%%%%%%%%%%%%%%%%%%%%%%%%%%%%%%%%%%%%%%%%%%%%%%%%%%%%%%%%%

We denote the various inclusion maps by%
\glossary{j_H@$j_H$, inclusion $T\to H$}
$$
T\overset{j_G}\longto G,\qquad T\overset{j_H}\longto
H\overset{i}\longto G.
$$
We let $G^{(\sigma)}$ be a central $\U(1)$-extension of $G$ and
$H^{(\tau)}$ a central $\U(1)$-extension of $H$.  Then we have a
pullback extension $H^{(\sigma)}$ of $H$ and two pullback extensions
$T^{(\sigma)}$ and $T^{(\tau)}$ of $T$, and corresponding inclusion
maps
$$
T^{(\sigma)}\overset{j_G}\longto G^{(\sigma)},\qquad
T^{(\sigma)}\overset{j_H}\longto H^{(\sigma)},\qquad
T^{(\tau)}\overset{j_H}\longto H^{(\tau)},\qquad
H^{(\sigma)}\overset{i}\longto G^{(\sigma)}.
$$

The root system of $G$ contains that of $H$.  We let $\ca{B}_H$ be the
unique basis of $\ca{R}_H$ such that the associated positive roots
$\ca{R}_H^+$ are positive for $G$.  We define%
\glossary{WH@$W^H$, shortest representatives for $W_G/W_H$}
$$W^H=\{\,w\in W_G\mid w(\ca{R}_H^+)\subset\ca{R}_G^+\,\}.$$
Then $W^H$ is a system of coset representatives for $W_G/W_H$.  The
set $\ca{R}_M=\ca{R}_G\setminus\ca{R}_H$%
\glossary{R_M@$\ca{R}_M$, weights $\ca{R}_G\setminus\ca{R}_H$ of
$\m_\C$}
(which is usually not a root system) is the set of weights of the
complexified $T$-module $\m_\C$.  We call the weights in
$\ca{R}_M^+=\ca{R}_G^+\setminus\ca{R}_H^+$%
\glossary{R_M+@$\ca{R}_M^+$, positive weights
$\ca{R}_G^+\setminus\ca{R}_H^+$ of $\m_\C$}
\emph{positive}; the set $\ca{R}_M$ is the disjoint union of
$\ca{R}_M^+$ and $-\ca{R}_M^+$.  The dimension of $M$ is the
cardinality of $\ca{R}_M$, which is even.  For any root $\alpha$ of
$G$, let $\g_\C^\alpha\subset\g_\C$%
\glossary{gfrakalpha@$\g_\C^\alpha$, root space of $\g_\C$}
be the root space corresponding to $\alpha$.  We have
$$
\m_\C=\bigoplus_{\alpha\in\ca{R}_M}\g_\C^\alpha,\qquad
\m=\bigoplus_{\alpha\in\ca{R}_M^+}\m^\alpha,
$$
where
$\m^\alpha=\m\cap\bigl(\g_\C^\alpha\oplus\g_\C^{-\alpha}\bigr)$.%
\glossary{malpha@$\m^\alpha$, weight space}
The projection $\m_\C\to\m$ which sends a vector to its real part
induces $T$-equivariant $\R$-linear isomorphisms
$\g_\C^\alpha\to\m^\alpha$ for all $\alpha\in\ca{R}_M^+$.  We endow
$\m$ with the $T$-invariant complex structure provided by this
isomorphism (which depends on the choice of the positive weights
$\ca{R}_M^+$).

\begin{theorem}\label{theorem;lefschetz}
Let $D$ be a $(\sigma,\tau)$-twisted operator on $M$ and let
$i_D\colon R(H,\sigma-\tau)\to R(G,\sigma)$ be the induction map
determined by $D$.  Then
$$
j_G^*i_D(a)=\sum_{w\in W^H}w\biggl(\frac{j_H^*(\e(D)a)}
{\prod_{\alpha\in\ca{R}_M}{(1-e^\alpha)}}\biggr)
$$
for all $a\in R(H,\sigma-\tau)$.
\end{theorem}

\begin{proof}
First assume that $\sigma=0$.  Then $D$ is a $G\times
H^{(\tau)}$-equivariant operator
$$D\colon\Gamma(G,G\times U^0)\to\Gamma(G,G\times U^1)$$
that is transversely elliptic with respect to $H^{(\tau)}$, where
$U^0$ and $U^1$ are $H^{(\tau)}$-modules of level $1$.  Choose a
representative $\dot{w}\in N_G(T)$ for each $w\in W^H$.  The map
$W^H\to M^T$ defined by $w\mapsto\dot{w}H/H$ is a bijection.  Let $V$
be an $H^{(\tau)}$-module of level $-1$ and let $\chi_V\colon
H^{(\tau)}\to\C$ be its character.  Let $\chi$ be the character of the
virtual $G$-module $i_D([V])=\index(D_{V^*})$.  Let $t$ be a generic
element of $T$.  By
\cite[\S\,II.5]{atiyah-bott;lefschetz-fixed-elliptic-complex},
$\chi(t)=\sum_{w\in W^H}\chi_w(t)$, where
$$
\chi_w(t)=\frac{\trace\bigl(\kappa^0(w^{-1}(t))\bigr)
-\trace\bigl(\kappa^1(w^{-1}(t))\bigr)}
{\bigabs{\det_\R\bigl(1-\eta(w^{-1}(t^{-1}))\bigr)}},
$$
and $\kappa^p$ is the representation $H\to\GL(U^p\otimes V)$ for $p=0$
or $1$.  We have
$$
\trace\bigl(\kappa^p(w^{-1}(t))\bigr)
=\chi_{U^p}\bigl(w^{-1}(t)\bigr)\chi_V\bigl(w^{-1}(t)\bigr),
$$
and arguing as in
\cite[\S\,II.5]{atiyah-bott;lefschetz-fixed-elliptic-complex} we find
$$
\det_\R\bigl(1-\eta(w^{-1}(t^{-1}))\bigr) =\prod_{\alpha\in\ca{R}_M^+}
\bigl(1-t^{w(\alpha)}\bigr)\bigl(1-t^{-w(\alpha)}\bigr)>0.
$$
Therefore
$$
\chi(t)=\sum_{w\in W^H}\frac{\bigl[\chi_{U^0}\bigl(w^{-1}(t)\bigr)
-\chi_{U^1}\bigl(w^{-1}(t)\bigr)\bigr]\chi_V\bigl(e^{w^{-1}(\xi)}\bigr)}
{\prod_{\alpha\in\ca{R}_M}\bigl(1-t^{w(\alpha)}\bigr)}.
$$
This identity, valid for generic $t$, amounts to an identity in the
fraction field of the representation ring $R(T)$, namely
$$
j_G^*i_D([V])=\sum_{w\in
W^H}\frac{wj_H^*\bigl(([U^0]-[U^1])\cdot[V]\bigr)}
{\prod_{\alpha\in\ca{R}_M}\bigl(1-e^{w(\alpha)}\bigr)},
$$
from which the assertion follows immediately.  The case of a general
extension $\sigma$ is handled by replacing $G$ with $G^{(\sigma)}$,
$H$ with $H^{(\sigma)}$, and $H^{(\tau)}$ with
$H^{(\sigma)}\times_HH^{(\tau)}$.
\end{proof}

The following example is well-known; see e.g.\
\cite[\S\,2]{segal;representation-ring}.

\begin{example}\label{example;euler}
Let $\sigma=\tau=0$.  The exterior algebra $\Lambda(\m_\C)$ is a
$\Z/2\Z$-graded $H$-equivariant $\Cl(\m)$-module.  The corresponding
Dirac operator defined with respect to the Levi-Civita connection (see
\S\,\ref{subsection;dirac}) is the Hodge-de Rham operator
$$
D=d+d^*\colon\Gamma\bigl(M,\Lambda^\even(T_\C^*M)\bigr)\to
\Gamma\bigl(M,\Lambda^\odd(T_\C^*M)\bigr),
$$
where $d^*$ is the adjoint of the exterior derivative $d$ with respect
to the Riemannian metric on $M$.  Also,
$\e(D)=\zeta^*([\sigma(D)])=\lambda_{-1}([\m_\C])$, and hence
$$
j_H^*(\e(D))=\lambda_{-1}(j_H^*([\m_\C]))=\prod_{\alpha\in\ca{R}_M^+}
(1-e^\alpha)(1-e^{-\alpha}).
$$
This expression cancels against the denominator in Theorem
\ref{theorem;lefschetz} and therefore
$$
j_G^*i_D(a)=\sum_{w\in W^H}w(j_H^*(a))
$$
for all $a\in R(H)$.  In particular
$$i_!(\e(D))=i_D(1)=\sum_{w\in W^H}1=\abs{W_G/W_H}$$
is the Euler characteristic of $M$.
\end{example}

%%%%%%%%%%%%%%%%%%%%%%%%%%%%%%%%%%%%%%%%%%%%%%%%%%%%%%%%%%%%%%%%%%%%%%%%
\section{Twisted $\Spin^c$-induction: the character formula}
\label{section;dirac}
%%%%%%%%%%%%%%%%%%%%%%%%%%%%%%%%%%%%%%%%%%%%%%%%%%%%%%%%%%%%%%%%%%%%%%%%

Among all twisted equivariant elliptic differential operators on a
maximal-rank homogeneous space $M=G/H$ there is a preferred one, the
\emph{twisted $\Spin^c$ Dirac operator} $\Dirac_M=\Dirac$,%
\glossary{DH@$\Dirac$, twisted $\Spin^c$ Dirac operator}
defined in \eqref{equation;dirac} below.  (The definition easily
extends to give a natural twisted $\Spin^c$ Dirac operator on every
oriented Riemannian manifold, which is closely analogous to the
twisted $\Spin$ Dirac operator of Murray and Singer
\cite[\S~3.3]{murray-singer;gerbes-index}.)  Thanks to the Thom
isomorphism every elliptic induction map can be expressed in terms of
$\Dirac$ (Theorem \ref{theorem;induction} below).  In this section we
develop the properties of twisted $\Spin^c$-induction up to the ``Weyl
character formula'' (Theorem \ref{theorem;weyl-character}).

%%%%%%%%%%%%%%%%%%%%%%%%%%%%%%%%%%%%%%%%%%%%%%%%%%%%%%%%%%%%%%%%%%%%%%%%
\subsubsection*{Notation}
%%%%%%%%%%%%%%%%%%%%%%%%%%%%%%%%%%%%%%%%%%%%%%%%%%%%%%%%%%%%%%%%%%%%%%%%

We retain the hypotheses and the notational conventions stated at the
beginning of \S\,\ref{section;induction} and
\S\,\ref{subsection;maximal}.  In addition we define%
\glossary{rhoM@$\rho_M$, half-character $\rho_G-\rho_H$}
$$
\rho_M=\rho_G-\rho_H=\frac12\sum_{\alpha\in\ca{R}_M^+}\alpha
\in\frac12\X(T).
$$
%

%%%%%%%%%%%%%%%%%%%%%%%%%%%%%%%%%%%%%%%%%%%%%%%%%%%%%%%%%%%%%%%%%%%%%%%%
\subsection{The Thom isomorphism}\label{subsection;thom}
%%%%%%%%%%%%%%%%%%%%%%%%%%%%%%%%%%%%%%%%%%%%%%%%%%%%%%%%%%%%%%%%%%%%%%%%

Let
$$\Spin^c(\m)=(\U(1)\times\Spin(\m))/K$$
be the complex spinor group of $\m=T_{\bar1}M$.  Here $K\cong\Z/2\Z$
denotes the central subgroup $\{(1,1),(-1,x)\}$ of
$\U(1)\times\Spin(\m)$, with $x$ being the nontrivial element in the
kernel of the double cover $\Spin(\m)\to\SO(\m)$.  The central
extension
\begin{equation}\label{equation;spin-c}
1\longto\U(1)\longto\Spin^c(\m)\longto\SO(\m)\longto1
\end{equation}
pulls back via the tangent representation $\eta\colon H\to\SO(\m)$ to
a central extension%
\glossary{omegaM@$\omega_M$, orientation system of $M$}
\begin{equation}\label{equation;orientation}
\omega=\omega_M\colon1\longto\U(1)\longto H\times_{\SO(\m)}\Spin^c(\m)
\longto H\longto1,
\end{equation}
which we call the \emph{orientation system} of $M$.  (See Appendix
\ref{section;twist}, Example \ref{example;equivariant-orientation}.
The orientation system is trivial if and only if $\eta$ lifts to a
representation $H\to\Spin^c(\m)$, which is equivalent to $M$ having a
$G$-invariant $\Spin^c$-structure, i.e.\ an orientation in equivariant
K-theory.  See Appendix \ref{section;spin} for a discussion of
$\Spin^c$-structures on $M$.)

Let us write elements of $\Spin^c(\m)$ as equivalence classes $[z,g]$
with $z\in\U(1)$ and $g\in\Spin(\m)$.  The involution
$f([z,g])=[z^{-1},g]$ of $\Spin^c(\m)$ restricts to the identity on
the subgroup $\Spin(\m)$ and therefore induces the identity on
$\SO(\m)$.  The involution $\eta^*(f)$ of $H^{(\omega)}$ therefore
defines an equivalence between the orientation system $\omega_M$ and
its opposite, and hence an isomorphism of $R(H)$-modules
$$R(H,\omega_M)\cong R(H,-\omega_M),$$
which we will use to identify these two modules.  (See Remark
\ref{remark;ext}.)

\begin{remark}\label{remark;stiefel-whitney}
Let $\Ext(G,\U(1))$ be the group of equivalence classes of central
extensions of $G$ by $\U(1)$.  There is an isomorphism of abelian
groups
$$
\Ext(G,\U(1))\cong H_G^3(\pt,\Z).
$$
(See \cite[\S\,6]{atiyah-segal;twisted-k-theory} or
\cite[\S\,2.3]{tu-xu-laurent-gengoux;twisted-differentiable-stacks}.)
Under the isomorphism $H_G^3(M,\Z)\cong H_H^3(\pt,\Z)$, the class of
$\omega_M$ in $\Ext(H,\U(1))\cong H_H^3(\pt,\Z)$ corresponds to the
integral equivariant Stiefel-Whitney class $W_G^3(M)\in H_G^3(M,\Z)$,
which is a $2$-torsion element.  If $G$ is simply connected, the
forgetful map $H_G^3(M,\Z)\to H^3(M,\Z)$ is injective (see
\cite[Lemma~3.3]{krepski;pre-quantization-flat}), so $M$ is $\Spin^c$
if and only if it is $G$-invariantly $\Spin^c$.
\end{remark}

Let $S=S^0\oplus S^1$%
\glossary{S@$S$, spinor module $S^0\oplus S^1$ of $\Cl(\m)$}
be the spinor module of the Clifford algebra $\Cl(\m)$.  This is a
$\Z/2\Z$-graded $H^{(\omega)}$-equivariant Clifford module, and the
$H^{(\omega)}$-action is of level $1$.  The associated Dirac operator
\begin{equation}\label{equation;dirac}
\Dirac=\Dirac_M\colon\Gamma(G^{(\sigma)},G^{(\sigma)}\times
S^0)\longto\Gamma(G^{(\sigma)},G^{(\sigma)}\times S^1)
\end{equation}
is a $(\sigma,\omega)$-twisted equivariant elliptic operator (see
\S\,\ref{subsection;dirac}), which we refer to as the \emph{twisted
$\Spin^c$ Dirac operator} on $M$.  We denote the corresponding
induction map by%
\glossary{i*@$i_*$, twisted $\Spin^c$-induction}
$$i_*=i_\Dirac\colon R(H,\sigma+\omega_M)\longto R(G,\sigma).$$

Let $\pi\colon T^*M\to M$%
\glossary{pi@$\pi$, projection $T^*M\to M$}
be the cotangent bundle projection.  The Thom isomorphism theorem,
Theorem \ref{theorem;thom} in Appendix \ref{subsection;isomorphism},
states that for every central $\U(1)$-extension $\tau$ of $H$ the map
$$
\zeta_*\colon R\bigl(H^{(\sigma)},\tau\bigr)\cong
K_{G^{(\sigma)}}^*(M,\tau)\longto
K_{G^{(\sigma)}}^*\bigl(BM,SM,\pi^*(\tau+\omega_M)\bigr)
$$
defined by $\zeta_*(a)=\pi^*(a)\th(M)$ is an isomorphism of
$R(H)$-modules.  We denote the inverse of $\zeta_*$ by $\pi_*$.  Here
$\th(M)$ is the Thom class of $M$, which by \eqref{equation;thom} is
equal to the symbol class of the operator $\Dirac$.

It was observed by Atiyah and Singer that the Thom isomorphism theorem
reduces the index theorem for general elliptic operators to the case
of the Dirac operator.  In exactly the same way the Thom isomorphism
enables us to express every induction map defined by a twisted
elliptic operator in terms of twisted $\Spin^c$-induction.

\begin{theorem}\label{theorem;induction}
Let $D$ be a $(\sigma,\tau)$-twisted operator on $M$.  There is a
unique element $a_D\in R\bigl(H^{(\sigma)},\tau+\omega_M\bigr)$ such
that $\e(D)=a_D\e(\Dirac)$, namely $a_D=\pi_*([D])$.  This element has
the property that $i_D(a)=i_*(a_Da)$ for all $a\in R(H,\sigma-\tau)$.
\end{theorem}

\begin{proof}
Put $a_D=\pi_*([D])$.  Using the definition of $\zeta_*$ and the fact
that $\th(M)=[\Dirac]$ we obtain $[D]=\pi^*(a_D)[\Dirac]$.  Pulling
back to the zero section gives
$$
\e(D)=\zeta^*([D])=a_D\zeta^*([\Dirac])=a_D\e(\Dirac).
$$
It follows from Lemma
\ref{lemma;twist-module}\eqref{item;extension-sum} that $a_D$ is the
only class that satisfies this identity.  Using this and applying
Theorem \ref{theorem;linear}\eqref{item;push} twice we find that
$$
i_D(a)=i_!(\e(D)a)=i_!(\e(\Dirac)a_Da)=i_\Dirac(a_Da)
$$
for all $a\in R(H,\sigma-\tau)$.
\end{proof}

\begin{example}\label{example;twisted-dirac}
Let $V$ be an $H^{(\tau-\omega_M)}$-module of level $1$.  The twisted
Dirac operator $D$ defined by the $H^{(\tau)}$-equivariant Clifford
module $S\otimes V$ is a $(\sigma,\tau)$-twisted operator. (See
\S\,\ref{subsection;dirac}.)  The symbol class of $D$ is
$[D]=[\Dirac]\pi^*([V])$, so the class $a_D$ is equal to
$$
a_D=\pi_*([D])=\pi_*\bigl([\Dirac]\pi^*([V])\bigr)=\pi_*([\Dirac])[V]=[V].
$$
Hence $i_D(a)=i_*([V]a)$ for all $a\in R(H,\sigma-\tau)$ by Theorem
\ref{theorem;induction}.
\end{example}

%%%%%%%%%%%%%%%%%%%%%%%%%%%%%%%%%%%%%%%%%%%%%%%%%%%%%%%%%%%%%%%%%%%%%%%%
\subsection{The character formula}\label{subsection;character}
%%%%%%%%%%%%%%%%%%%%%%%%%%%%%%%%%%%%%%%%%%%%%%%%%%%%%%%%%%%%%%%%%%%%%%%%

The centre of $\Spin^c(\m)$ is $\U(1)$ and its commutator subgroup is
$\Spin(\m)$, the universal covering group of $\SO(\m)$.  We therefore
have a natural infinitesimal splitting
$\lie{spin}^c(\m)=\lie{u}(1)\oplus\lie{so}(\m)$ of the extension
\eqref{equation;spin-c}.  Pulling back to $H$ we obtain an
infinitesimal splitting of the orientation system $\omega$, and then
restricting to the maximal torus $T$ of $H$ we get an infinitesimal
splitting of the extension
$$
1\longto\U(1)\longto T^{(\omega)}\longto T\longto1,
$$
where $T^{(\omega)}\subset H^{(\omega)}$ is the inverse image of $T$.
Dually we have a rational splitting
$q\colon\X(U(1))_\Q\to\X(T^{(\omega)})_\Q$ of the exact sequence of
character groups
$$
0\longto\X(T)\longto\X(T^{(\omega)})\longto\X(\U(1))\longto0.
$$
We identify $\X(T)$ with its image in $\X(T^{(\omega)})$ and
$\X(\U(1))$ with its image $q(\X(\U(1)))$ in $\X(T^{(\omega)})_\Q$.
Let $\eps_0=\id_{\U(1)}$%
\glossary{eps0@$\eps_0$, standard generator of $\X(\U(1))\cong\Z$}
be the standard generator of $\X(\U(1))$.  Since $\Spin(\m)$ is a
double covering of $\SO(\m)$, we have
$\eps_0\in\frac12\X(T^{(\omega)})$.

In the next lemma we compute the Euler class $\e(\Dirac)\in
R(H,\omega_M)$, or rather its restriction to the maximal torus
$T^{(\omega)}$ of $H^{(\omega)}$.  For ease of notation we denote the
inclusion $T^{(\omega)}\to H^{(\omega)}$ by $j_H$.

\begin{lemma}\label{lemma;weight-euler}
\begin{enumerate}
\item\label{item;rho}
$\eps_0-\rho_M\in\X(T^{(\omega)})$.
\item\label{item;omega}
$\X(T^{(\omega)})=\X(T)\oplus\Z\cdot(\eps_0-\rho_M)$.
\item\label{item;Romega}
$R(T,\omega)=R(T)\cdot e^{\eps_0-\rho_M}$.
\item\label{item;euler-twist}
$j_H^*(\e(\Dirac))
=e^{\eps_0-\rho_M}\prod_{\alpha\in\ca{R}_M^+}(1-e^\alpha)
=e^{\eps_0}\prod_{\alpha\in\ca{R}_M^+}(e^{-\alpha/2}-e^{\alpha/2})$.
\end{enumerate}
\end{lemma}

\begin{proof}
Let $\ca{R}_M^+=\{\alpha_1,\alpha_2,\dots,\alpha_l\}$ and put
$\m_k=\m^{\alpha_k}$.  The group
$$\group{T}=\prod_{k=1}^l\SO(\m_k)$$
is a maximal torus of $\SO(\m)$ that contains the torus $\eta(T)$.
Let $\eps_k$ be the weight of the $\group{T}$-action on $\m_k$.
(Recall our convention whereby we identify $\m_k$ with the root space
$\g^{\alpha_k}$, so turning $\m$ into a unitary $\group{T}$-module.)
Let $\hat{\group{T}}$ be the inverse image of $\group{T}$ in
$\Spin^c(\m)$.  Using \cite[Ch.~VI,
Planche~IV]{bourbaki;groupes-algebres} we compute the character groups
of $\group{T}$ and $\hat{\group{T}}$ to be
$$
\X(\group{T})=\bigoplus_{k=1}^l\Z\eps_k,\qquad
\X(\hat{\group{T}})=\Z\hat{\eps}\oplus\bigoplus_{k=1}^l\Z\eps_k,
$$
where $\hat{\eps}=\eps_0-\frac12\sum_{k=1}^l\eps_k$.  Because $T$ acts
with weight $\alpha_k$ on $\m_k$, the homomorphisms
$$
\eta^*\colon\X(\group{T})\to\X(T),\qquad
\hat{\eta}^*\colon\X(\hat{\group{T}})\to\X(T^{(\omega)})
$$
induced by $\eta\colon T\to\group{T}$ and its lift $\hat{\eta}\colon
T^{(\omega)}\to\hat{\group{T}}$ are given by
$$
\eta^*(\eps_k)=\alpha_k,\qquad\hat{\eta}^*(\hat{\eps})=\eps_0-\rho_M.
$$
By definition $T^{(\omega)}$ is the pullback
$T\times_{\group{T}}\hat{\group{T}}$, so dually its character group
$\X(T^{(\omega)})$ is the pushout
$$
\X(T^{(\omega)})\cong
\bigl(\X(T)\oplus\X(\hat{\group{T}})\bigr)\big/\X(\group{T}).
$$
Hence $\X(T^{(\omega)})$ is generated by $\X(T)$ and $\eps_0-\rho_M$,
which proves \eqref{item;rho} and \eqref{item;omega}.  Let
$\lambda\in\X(T^{(\omega)})$ and write $\lambda=\mu+k(\eps_0-\rho_M)$
with $\mu\in\X(T)$ and $k\in\Z$.  Then the element
$e^\lambda\in\Z[\X(T^{(\omega)})]\cong R(T^{(\omega)})$ is of level
$1$ if and only if $k=1$.  Thus the $R(T)$-module $R(T,\omega)$ is
spanned by the element $e^{\eps_0-\rho_M}$, which proves
\eqref{item;Romega}.  According to \eqref{equation;euler}, the Euler
class of $\Dirac$ is 
\begin{equation}\label{equation;euler-twist}
\e(\Dirac)=\hat{\eta}^*\bigl([S^0]-[S^1]\bigr)\in R(H,\omega).
\end{equation}
To find its restriction to $R(T,\omega)$, we must describe the
$T$-action on the half-spin representations $S^0$ and $S^1$.
As vector spaces,
$$
S=\Lambda_\C(\m),\qquad S^0=\Lambda_\C^\even(\m),\qquad
S^1=\Lambda_\C^\odd(\m).
$$
For $1\le k\le l$, let $v_k\in\m_k$ be a vector of unit length.  Then
the products
\begin{equation}\label{equation;weight-vector}
v_{\group{i}}=v_{i_1}\wedge\cdots\wedge v_{i_k},
\end{equation}
where $\group{i}=(i_1,i_2,\dots,i_k)$ ranges over all increasing
multi-indices $\group{i}$ of length
$\abs{\group{i}}=k\in\{0,1,\dots,l\}$, form an orthonormal basis of
$\Lambda_\C(\m)$.  Let $\tilde{\group{T}}$ be the inverse image in
$\Spin(\m)$ of the maximal torus $\group{T}$.  Its character group is
$$
\X(\tilde{\group{T}})=\Z\tilde{\eps}+\bigoplus_{k=1}^l\Z\eps_k,
$$
where $\tilde{\eps}=\frac12\sum_{k=1}^l\eps_k$.  The vector
$v_{\group{i}}$ is a weight vector of the $\tilde{\group{T}}$-action
with weight $-\tilde{\eps}+\eps_{\group{i}}$, where
$\eps_{\group{i}}=\eps_{i_1}+\cdots+\eps_{i_k}$.  (Cf.\
\cite[\S\,VIII.13.4]{bourbaki;groupes-algebres}.)  The centre $\U(1)$
of $\Spin^c(\m)$ acts by scalar multiplication, so the weight of
$v_{\group{i}}$ for the $\hat{\group{T}}$-action is
$\eps_0-\tilde{\eps}+\eps_{\group{i}}=\hat{\eps}+\eps_{\group{i}}$.
Hence the weight of $v_{\group{i}}$ for the $T$-action is
\begin{equation}\label{equation;weight}
\hat{\eta}(\hat{\eps}+\eps_{\group{i}})
=\eps_0-\rho_M+\alpha_{\group{i}}.
\end{equation}
Thus the class of $\Lambda^k_\C(\m)$ in $R(T,\omega)$ is
$e^{\eps_0-\rho_M}\sum_{\abs{\group{i}}=k}e^{\alpha_{\group{i}}}$, and
hence
$$
j_H^*(\e(\Dirac))=e^{\eps_0-\rho_M}\sum_{k=1}^l(-1)^k
\sum_{\abs{\group{i}}=k}e^{\alpha_{\group{i}}}
=e^{\eps_0-\rho_M}\prod_{k=1}^l(1-e^{\alpha_k}),
$$
which establishes \eqref{item;euler-twist}.
\end{proof}

\begin{remark}\label{remark;shift}
Let $\Z[\rho_M+\X(T)]$ be the free abelian group generated by the
affine lattice $\rho_M+\X(T)$ in $\X(T)_\Q$.  The basis elements of
$\Z[\rho_M+\X(T)]$ are of the form $e^{\mu+\rho_M}$ with
$\mu\in\X(T)$.  The group
$$
\Z[\rho_M+\X(T)]=\Z[\X(T)]\,e^{\rho_M}\cong R(T)\cdot e^{\rho_M}
$$
is not a ring, but a module over the ring $\Z[\X(T)]\cong R(T)$.  By
Lemma \ref{lemma;weight-euler}\eqref{item;Romega}, the affine map
$\rho_M+\X(T)\to\X(T^{(\omega)})$ defined by
$\lambda\mapsto\eps_0+\lambda$ induces an isomorphism
$$
R(T)\cdot e^{\rho_M}\overset\cong\longto R(T,\omega),\qquad
a\longmapsto e^{\eps_0}\cdot a
$$
of $R(T)$-modules, which we will call the \emph{level shift}.  The
preimage of $j_H^*(\e(\Dirac))$ under the level shift is the class
$$
e^{-\rho_M}\prod_{\alpha\in\ca{R}_M^+}(1-e^\alpha)
\in\Z[\rho_M+\X(T)].
$$
Since $\rho_M-w(\rho_M)$ is in the root lattice of $G$ for all $w\in
W_H$, the $W_H$-action on $\X(T)_\Q$ preserves $\rho_M+\X(T)$.
Because $\eps_0$ is central, the level shift is $W_H$-equivariant and
hence restricts to an isomorphism of $R(H)$-modules
$$\Z[\rho_M+\X(T)]^{W_H}\overset\cong\longto R(H,\omega),$$
which we also refer to as the level shift.  (If $\rho_M\not\in\X(T)$,
i.e.\ if $M$ does not have a $G$-invariant $\Spin$-structure, we can
view $R(H)\oplus R(H,\omega)$ as the representation ring of the double
cover $H\times_{\SO(\m)}\Spin(\m)$ of $H$.)  In the sequel we will
make the identifications
$$
R(T,\omega)=R(T)\cdot e^{\rho_M},\qquad R(H,\omega)=\bigl(R(T)\cdot
e^{\rho_M}\bigr)^{W_H}.
$$
Formally this amounts to setting $\eps_0=0$, which has the desirable
effect of cleaning up several formulas.
\end{remark}

The \emph{antisymmetrizer} of $W_G$ is the element $J_G$ of the group
ring $\Z[W_G]$ defined by%
\glossary{Jg@$J_G$, antisymmetrizer of $W_G$}
\begin{equation}\label{equation;antisymmetrizer}
J_G=\sum_{w\in W_G}\det(w)w.
\end{equation}
Let $A$ be a $W_G$-module and let $A^{-W_G}$ denote the set of
\emph{anti-invariant} elements of $A$, i.e.\ those $a\in A$ satisfying
$w(a)=\det(w)a$ for all $w\in W_G$.  Then $J_G$ defines a
$\Z[W_G]$-linear operator $J_G\colon A\to A^{-W_G}$, and we have
$J_G(a)=\abs{W_G}\,a$ for all anti-invariant $a\in A$.  Similarly, we
have the antisymmetrizer $J_H\in\Z[W_H]$ with respect to the Weyl
group $W_H$.  We also define a \emph{relative} antisymmetrizer
$J_M\in\Z[W_G]$ by%
\glossary{Jm@$J_M$, relative antisymmetrizer}
$$J_M=\sum_{w\in W^H}\det(w)w.$$
The fact that each $w\in W$ can be written uniquely as a product
$w'w''$ with $w'\in W^H$ and $w''\in W_H$ implies that
\begin{equation}\label{equation;anti-invariant}
J_G=J_MJ_H.
\end{equation}

Recall that the $W_G$-anti-invariant element $J_G(e^{\rho_G})$ of
$\Z\bigl[\frac12\X(T)\bigr]$ is equal to the \emph{Weyl denominator}
\begin{equation}\label{equation;weyl-denominator}
\glossary{dG@$\d_G$, Weyl denominator of $G$}
\d_G=e^{\rho_G}\prod_{\alpha\in\ca{R}_G^+}(1-e^{-\alpha}).
\end{equation}
(See e.g.\ \cite[\S\,VI.3, Proposition~2]{bourbaki;groupes-algebres}.)
The \emph{duality homomorphism} $R(H,\tau)\to R(H,-\tau)$ is the map
$b\mapsto b^*$ defined on generators by $[V]^*=[V^*]$.

\begin{lemma}\label{lemma;weyl-denominator}
Let $\Dirac_{G/T}$ and $\Dirac_{H/T}$ be the twisted $\Spin^c$ Dirac
operators of the flag varieties $G/T$, resp.\ $H/T$.  We have the
identities
$$
\d_G=\e(\Dirac_{G/T})^*,\quad j_H^*(\e(\Dirac_M)^*)=\d_G/\d_H,\quad
\e(\Dirac_{G/T})=j_H^*(\e(\Dirac_{G/H}))\e(\Dirac_{H/T}).
$$
\end{lemma}

\begin{proof}
This follows immediately from Lemma
\ref{lemma;weight-euler}\eqref{item;euler-twist} and
\eqref{equation;weyl-denominator} (after applying a level shift; see
Remark \ref{remark;shift}).
\end{proof}

A version of the following character formula was stated by Bott
\cite[p.~ 179]{bott;homogeneous-differential}, but his hypotheses and
proof are not entirely correct.  Incorporating the $\omega$-twist
resolves the problems in Bott's treatment.  A version for Lie algebras
(which does not require any twists) can be found in
\cite{gross-kostant-ramond-sternberg}.  Note that the character
formula depends on the choice of the positive roots.  The induction
map $i_*$, however, is independent of this choice.

\begin{theorem}\label{theorem;weyl-character}
Let $\Dirac_M$ be the twisted $\Spin^c$ Dirac operator on $M$ and let
$i_*\colon R(H,\sigma+\omega_M)\to R(G,\sigma)$ be the associated
induction map.  Then
$$
j_G^*i_*(a)=\frac{J_M\bigl(\d_Hj_H^*(a)\bigr)}{\d_G}
$$
for all $a\in R(H,\sigma+\omega_M)$.
\end{theorem}

\begin{proof}
It follows from Lemma \ref{lemma;weight-euler}\eqref{item;euler-twist}
that
$$
j_H^*(\e(\Dirac)\e(\Dirac)^*)=\prod_{\alpha\in\ca{R}_M}(1-e^\alpha).
$$
Substituting this expression into the formula of Theorem
\ref{theorem;lefschetz} yields  
\begin{equation}\label{equation;weyl}
j_G^*i_*(a)=\sum_{w\in
W^H}w\Bigl(j_H^*\Bigl(\frac{a}{\e(\Dirac)^*}\Bigr)\Bigr)
\end{equation}
for all $a\in R(H,\sigma+\omega_M)$.  Lemma
\ref{lemma;weyl-denominator} shows that
$$
w\Bigl(j_H^*\Bigl(\frac{a}{\e(\Dirac)^*}\Bigr)\Bigr)
=w\Bigl(\frac{\d_Hj_H^*(a)}{\d_G}\Bigr)
=\frac{\det(w)w\bigl(\d_Hj_H^*(a)\bigr)}{\d_G}
$$
for all $w\in W^H$, and substituting this into \eqref{equation;weyl}
proves the result.
\end{proof}

Let us explain the extent to which this formula is parallel to Weyl's
character formula.  For simplicity let $\sigma=0$.  The following
result generalizes a well-known result for semisimple simply connected
groups.  The proof is in Appendix \ref{section;anti-invariant}.

\begin{proposition}\label{proposition;anti-invariant}
\begin{enumerate}
\item
The $W_G$-action on $R\bigl(T^{(\omega_{G/T})}\bigl)$ preserves
$R(T,\omega_{G/T})$.
\item
The set of anti-invariants $R(T,\omega_{G/T})^{-W_G}$ is a free
$R(G)$-module of rank $1$ generated by $\d_G$.
\item
The elements $J_G(e^\lambda)$, with $\lambda\in\rho_G+\X(T)$ strictly
dominant, form a basis of the $\Z$-module $R(T,\omega_{G/T})^{-W_G}$.
\end{enumerate}
\end{proposition}

Applying this to $G$ and $H$ and using the equivalence
$\omega_{G/H}+\omega_{H/T}\sim\omega_{G/T}$ (see Lemma
\ref{lemma;functor} below), we can interpret the character formula for
twisted $\Spin^c$-induction as a three-step process,
$$
R(H,\omega_{G/H})\underset\cong{\overset{\cdot\d_H}\longto}
R(T,\omega_{G/T})^{-W_H}\overset{J_M}\longto R(T,\omega_{G/T})^{-W_G}
\underset\cong{\overset{/\d_G}\longto}R(G).
$$
%

%%%%%%%%%%%%%%%%%%%%%%%%%%%%%%%%%%%%%%%%%%%%%%%%%%%%%%%%%%%%%%%%%%%%%%%%
\section{Further properties of twisted $\Spin^c$-induction}
\label{section;properties}
%%%%%%%%%%%%%%%%%%%%%%%%%%%%%%%%%%%%%%%%%%%%%%%%%%%%%%%%%%%%%%%%%%%%%%%%

We continue our discussion of the twisted induction map defined by
means of the twisted $\Spin^c$ Dirac operator on the homogeneous space
$M=G/H$.  We retain the notational conventions of
\S\S\,\ref{section;induction} and \ref{section;dirac}.  We show that
twisted $\Spin^c$-induction is functorial with respect to the subgroup
$H$ (Theorem \ref{theorem;functor}).  If $M$ has an invariant
$\Spin^c$-structure (which is not always the case), then there is a
transparent relationship between $\Spin^c$-induction and twisted
$\Spin^c$-induction (Theorem \ref{theorem;versus}).  Twisted induction
of irreducible modules behaves according to a ``Borel-Weil-Bott
formula'' (Theorem \ref{theorem;irreducible}).  Twisted
$\Spin^c$-induction gives rise to a bilinear pairing between twisted
representation modules.  These modules are twisted equivariant
K-groups, and the pairing is analogous to an intersection pairing in
ordinary equivariant cohomology.  Under a mild condition on $G$ the
pairing is nonsingular (Theorem \ref{theorem;nonsingular}), which
enables us to characterize all twisted induction maps from $H$ to $G$
(Theorem \ref{theorem;dirac-induction}).

%%%%%%%%%%%%%%%%%%%%%%%%%%%%%%%%%%%%%%%%%%%%%%%%%%%%%%%%%%%%%%%%%%%%%%%%
\subsection{Twisted induction in stages}\label{subsection;natural}
%%%%%%%%%%%%%%%%%%%%%%%%%%%%%%%%%%%%%%%%%%%%%%%%%%%%%%%%%%%%%%%%%%%%%%%%

Twisted $\Spin^c$-induction is functorial with respect to the subgroup
$H$.  Consider two closed connected subgroups $H_2\subset H_1$ of $G$,
both of which contain $T$.  We have a diagram of inclusions
$$
\xymatrix@=3em{
&H_1\ar[d]^{i_1}\\
H_2\ar[r]_{i_2}\ar[ur]^k&G.
}
$$
Let $\m_1=\g/\h_1$, $\m_2=\g/\h_2$ and $\n=\h_1/\h_2$.  Being a sum of
root spaces, $\m_2$ is naturally isomorphic to the orthogonal direct
sum $\m_1\oplus\n$.  We have three orientation systems,
\begin{align*}
\omega_{G/H_p}\colon1\longto\U(1)\longto
H_p\times_{\SO(\m_p)}\Spin^c(&\m_p) \longto H_p\longto1\quad(p=1,2),\\
\omega_{H_1/H_2}\colon1\longto\U(1)\longto
H_2\times_{\SO(\n)}\Spin^c(&\n) \longto H_2\longto1.
\end{align*}

\begin{lemma}\label{lemma;functor}
The natural homomorphism $f\colon\SO(\m_1)\times\SO(\n)\to\SO(\m_2)$
induces an equivalence of extensions $f^*\colon
k^*(\omega_{G/H_1})+\omega_{H_1/H_2}\overset\sim\longto\omega_{G/H_2}$.
\end{lemma}

\begin{proof}
For $p=0$ or $1$, put $\G_p=\SO(\m_p)$ and $\hat{\G}_p=\Spin^c(\m_p)$.
Put $\G=\SO(\n)$ and $\hat{\G}=\Spin^c(\n)$.  The homomorphism
$f\colon\G_1\times\G\to\G_2$ lifts to a morphism of $\U(1)$-extensions
of $\G_1$,
\begin{equation}\label{equation;big-extension}
(\hat{\G}_1\times\hat{\G})/K\longto\hat{\G}_2,
\end{equation}
where $K$ is a copy of $\U(1)$ anti-diagonally embedded in
$\hat{\G}_2\times\hat{\G}$.  The morphism
\eqref{equation;big-extension} is an isomorphism by the five-lemma.
Pulling the extension $H_1^{(\omega_{G/H_1})}$ of $H_1$ back to $H_2$
gives the $\U(1)$-extension
$$
H_2\times_{H_1}\bigl(H_1\times_{\G_1}\hat{\G}_1\bigr)\cong
H_2\times_{\G_1}\hat{\G}_1
$$
of $H_2$.  Adding this extension to $H_2^{(\omega_{H_1/H_2})}$ gives
the extension
\begin{multline*}
\bigl[\bigl(H_2\times_{\G_1}\hat{\G}_1\bigr)
\times_{H_2}\bigl(H_2\times_{\G}\hat{\G}\bigr)\bigr]\big/K
\cong\bigl[\bigl(H_2\times_{\G_1}\hat{\G}_1\bigr)
\times_{\G}\hat{\G}\bigr]\big/K
\\
\cong H_2\times_{\G_1\times\G}\bigl[\bigl(\hat{\G}_1
\times\hat{\G}\bigr)\big/K\bigr]
\end{multline*}
of $H_2$, which is isomorphic to $H_2\times_{\G_2}\hat{\G}_2$ by
\eqref{equation;big-extension}.
\end{proof}

For the sake of brevity we will write this equivalence as
$$\omega_{G/H_1}+\omega_{H_1/H_2}\sim\omega_{G/H_2}$$
and use it to make the identification
$$
R\bigl(H_2,\sigma+\omega_{G/H_2}\bigr)
=R\bigl(H_2,\sigma+\omega_{G/H_1}+\omega_{H_1/H_2}\bigr).
$$

A version of the next theorem, under the assumption that $G/H_1$ and
$H_1/H_2$ are $\Spin$, was proved in \cite[Part~II,
\S\,4]{slebarski;dirac-homogeneous}.

\begin{theorem}\label{theorem;functor}
\begin{enumerate}
\item\label{item;multiplicative}
The Euler class of the twisted $\Spin^c$ Dirac operator is
multiplicative in the sense that
$\e(\Dirac_{G/H_2})=k^*(\e(\Dirac_{G/H_1}))\e(\Dirac_{H_1/H_2})$.
\item\label{item;functorial}
Twisted $\Spin^c$-induction is functorial in the sense that the
diagram
$$
\xymatrix{
R\bigl(H_2,\sigma+\omega_{G/H_1}+\omega_{H_1/H_2}\bigr)
\ar[r]^-{k_*}\ar@{=}[d]&
R\bigl(H_1,\sigma+\omega_{G/H_1}\bigr)\ar[d]^{i_{1,*}}
\\
R\bigl(H_2,\sigma+\omega_{G/H_2}\bigr)\ar[r]^-{i_{2,*}}&R(G,\sigma)
}
$$
commutes.
\end{enumerate}
\end{theorem}

\begin{proof}
\eqref{item;multiplicative} follows from Lemma
\ref{lemma;weight-euler}\eqref{item;euler-twist}.  Let $a\in
R\bigl(H_2,\sigma+\omega_{G/H_2}\bigr)$.  It follows from
\eqref{equation;composition} that the formal pushforward homomorphism
is functorial in the sense that $i_{2,!}=i_{1,!}\circ k_!$.  Hence,
using Theorem \ref{theorem;linear}\eqref{item;push}, the
$R(H_1)$-linearity of the formal pushforward, and the multiplicativity
of the Euler class, we obtain
\begin{align*}
i_{2,*}(a)&=i_{2,!}\bigl(\e(\Dirac_{G/H_2})a\bigr)
\\
&=i_{1,!}k_!\bigl(k^*(\e(\Dirac_{G/H_1}))\e(\Dirac_{H_1/H_2})a\bigr)
\\
&=i_{1,!}\bigl(\e(\Dirac_{G/H_1})k_!(\e(\Dirac_{H_2/H_1})a)\bigr)
\\
&=i_{1,*}k_*(a)
\end{align*}
for all $a$, so $i_{2,*}=i_{1,*}\circ k_*$.
\end{proof}

%%%%%%%%%%%%%%%%%%%%%%%%%%%%%%%%%%%%%%%%%%%%%%%%%%%%%%%%%%%%%%%%%%%%%%%%
\subsection{$\Spin^c$ versus twisted $\Spin^c$}\label{subsection;versus}
%%%%%%%%%%%%%%%%%%%%%%%%%%%%%%%%%%%%%%%%%%%%%%%%%%%%%%%%%%%%%%%%%%%%%%%%

We call a character $\gamma\in\X(H)$%
\glossary{gamma@$\gamma$, c-spinorial character of $H$}
\emph{c-spinorial} if the homomorphism $\gamma\times\eta\colon
H\to\U(1)\times\SO(\m)$ lifts to a homomorphism $\hat{\eta}\colon
H\to\Spin^c(\m)$.  Such characters $\gamma$ exist if and only if $M$
admits a $G$-invariant $\Spin^c$-structure, and they classify such
structures up to equivalence.  See Appendix \ref{section;spin} and
\cite{landweber-sjamaar;spin} for a discussion of invariant
$\Spin^c$-structures.  Let $\gamma$ be a c-spinorial character.  Then
we have the untwisted elliptic operator%
\glossary{dh@$\dirac$, $\Spin^c$ Dirac operator}
$$
\dirac=\dirac_\gamma\colon\Gamma\bigl(M,G\times^HS^0\bigr)
\overset\nabla\longto\Gamma\bigl(M,G\times^H(\m\times
S^0)\bigr)\overset{\cliff}\longto\Gamma\bigl(M,G\times^HS^1\bigr),
$$
called the \emph{$\Spin^c$ Dirac operator} on $M$ associated with
$\gamma$.  Coupling the Dirac operator with untwisted $H$-modules
gives rise to an induction map
\glossary{i_Dirac@$i_\dirac$, $\Spin^c$-induction}
$$i_\dirac=i_{\dirac_\gamma}\colon R(H)\longto R(G).$$

\begin{theorem}\label{theorem;versus}
Let $\gamma$ be a c-spinorial character of $H$.  Then the
$R(H)$-module $R(H,\omega)$ is freely generated by $e^{\gamma/2}$, and
the diagram
$$
\xymatrix{
R(H)\ar[r]^-{i_\dirac}\ar[d]_-{e^{\gamma/2}}^-{\cong}&R(G)
\\
R(H,\omega)\ar[ur]_-{i_*}
}
$$
commutes.  Hence
$j_G^*i_\dirac(a)=J_M\bigl(e^{\gamma/2}\d_Ha\bigr)\big/\d_G$ for all
$a\in R(H)$.
\end{theorem}

\begin{proof}
As in Remark \ref{remark;shift}, we identify $R(H,\omega)$ with the
$R(H)$-module
$$
\bigl(R(T)\cdot e^{\rho_M}\bigr)^{W_H}\cong\bigl(\Z[\X(T)]\cdot
e^{\rho_M}\bigr)^{W_H}.
$$
By Proposition \ref{proposition;spin-c}\eqref{item;lifting}, the
element $\rho_M-\frac12j_H^*(\gamma)$ is in $\X(T)$, and therefore
$e^{j_H^*(\gamma)/2}$ generates the $\Z[\X(T)]$-module $\Z[\X(T)]\cdot
e^{\rho_M}$.  Being the restriction of a character of $H$, the element
$e^{j_H^*(\gamma)/2}$ is invariant under $W_H$.  Therefore
$$
\bigl(\Z[\X(T)]\cdot e^{\rho_M}\bigr)^{W_H}=\bigl(\Z[\X(T)]\cdot
e^{j_H^*(\gamma)/2}\bigr)^{W_H}=\Z[\X(T)]^{W_H}\cdot
e^{j_H^*(\gamma)/2},
$$
which implies that $e^{\gamma/2}$ generates $R(H,\omega)$.  Since is
$e^{\gamma/2}$ a unit, this shows that $R(H,\omega)$ is free.  By
Proposition \ref{proposition;principal-chern-euler}\eqref{item;euler},
the Euler class of $\dirac_\gamma$ is $e^{\gamma/2}\e(\Dirac)$.
Therefore, by Theorem \ref{theorem;induction},
$i_\dirac(a)=i_*(e^{\gamma/2}a)$ for all $a\in R(H)$.  The last
assertion now follows from the character formula, Theorem
\ref{theorem;weyl-character}.
\end{proof}

The best-known special cases of $\Spin^c$-induction are
$\Spin$-induction and holomorphic induction.  The comparison with
twisted $\Spin^c$-induction is as follows.

\begin{example}\label{example;spin-induction}
Suppose that $M$ possesses a $G$-invariant $\Spin$-structure.  This is
the case if and only if $\rho_M\in\X(T)$.  The c-spinorial character
of the corresponding $\Spin^c$-structure is $\gamma=0$.  (See Example
\ref{example;spin}.)  Therefore, in this case twisted
$\Spin^c$-induction is the same as $\Spin$-induction.  (However,
twisted induction is defined even if $M$ has no invariant
$\Spin$-structure.)  The character formula reads
$j_G^*i_\dirac(a)=J_M(\d_Ha)\big/\d_G$ for all $a\in R(H)$.  In
particular, setting $H=T$ and supposing that $\rho_G\in\X(T)$, we find
$j_G^*j_{G,\dirac}(a)=J_G(a)/\d_G$ for all $a\in R(T)$.
\end{example}

\begin{example}\label{example;holomorphic-induction}
Suppose that $M$ possesses a $G$-invariant complex structure.  This is
the case if and only if $H$ is the centralizer of a subtorus of $T$.
The c-spinorial character for the associated $\Spin^c$-structure is
given by $j_H^*(\gamma)=2\rho_M$.  (See Example
\ref{example;complex}.)  Therefore holomorphic induction is given by
$i_\dirac(a)=i_*(e^{\rho_M}a)$.  Holomorphic induction depends on the
choice of the invariant complex structure (in other words, the choice
of the basis of the root system), but twisted induction does not.  The
character formula is $j_G^*i_\dirac(a)=J_M(e^{\rho_M}\d_Ha)/\d_G$ for
all $a\in R(H)$.  In particular, setting $H=T$ gives
$j_G^*j_{G,\dirac}(a)=J_G(e^{\rho_G}a)/\d_G$ for all $a\in R(T)$,
which is the usual Weyl character formula.
\end{example}

%%%%%%%%%%%%%%%%%%%%%%%%%%%%%%%%%%%%%%%%%%%%%%%%%%%%%%%%%%%%%%%%%%%%%%%%
\subsection{Irreducibles}\label{subsection;irreducible}
%%%%%%%%%%%%%%%%%%%%%%%%%%%%%%%%%%%%%%%%%%%%%%%%%%%%%%%%%%%%%%%%%%%%%%%%

The pushforward of an irreducible $H$-module to $G$ is given by a
well-known Borel-Weil-Bott type formula, Theorem
\ref{theorem;irreducible} below.  Such a formula was obtained for
homogeneous $\Spin$-manifolds by Slebarski \cite[Part~II, \S\,4,
Theorem~2]{slebarski;dirac-homogeneous} and for homogeneous
$\Spin^c$-manifolds by
Landweber~\cite{landweber;harmonic-spinors-homogeneous},%
~\cite{landweber;twisted}.  Versions for Lie algebra representations
were given by Gross et al.\ \cite{gross-kostant-ramond-sternberg} and
Kostant \cite{kostant;cubic-dirac-multiplet},
\cite{kostant;bott-borel-weil-multiplet}.  Our modest contribution is
to state a version of the formula which holds at the group level, but
which does not hypothesize the existence of a $\Spin^c$-structure.
Since it is an equality of characters, the formula actually follows
from the Lie algebra version.  To illustrate our techniques we will
derive it from the functoriality theorem, Theorem
\ref{theorem;functor}.

An irreducible $G^{(\sigma)}$-module $V$ is determined up to
isomorphism by its highest weight, which is a dominant character
$\lambda\in\X\bigl(T^{(\sigma)}\bigr)$.  If $V$ is of level $1$, then
its highest-weight space $V_\lambda$ is a one-dimensional
$T^{(\sigma)}$-module of level $1$.  We will call such a $\lambda$ a
\emph{$G$-dominant character of level $1$}, and denote the module $V$
by $V_G(\lambda,\sigma)$.%
\glossary{V@$V_G(\lambda,\sigma)$, irreducible $G^{(\sigma)}$-module
of level $1$ with highest weight $\lambda$}
If $\sigma=0$, we write $V_G(\lambda,\sigma)=V_G(\lambda)$.

\begin{theorem}\label{theorem;irreducible}
Let $\mu$ be an $H$-dominant level $1$ character of
$T^{(\sigma+\omega_M)}$.  There exists at most one $w\in W^H$ such
that $w(\mu+\rho_H)-\rho_G$ is $G$-dominant.  If no such $w$ exists,
then $i_*([V_H(\mu,\sigma+\omega_M)])=0$; if $w$ exists, then
$$
i_*([V_H(\mu,\sigma+\omega_M)])
=\det(w)[V_G(w(\mu+\rho_H)-\rho_G,\sigma)].
$$
\end{theorem}

\begin{proof}
The Borel-Weil-Bott theorem for the group $G$ implies the following
assertion: for each $\kappa\in\X(T)$ there is at most one $v\in W_G$
such that $v(\kappa+\rho_G)-\rho_G$ is $G$-dominant; if no such $v$
exists, then $j_{G,\dirac}(e^\kappa)=0$; if $v$ exists, then
$$
j_{G,\dirac}(e^\kappa)=\det(v)[V_G(v(\kappa+\rho_G)-\rho_G)].
$$
Here $j_{G,\dirac}\colon R(T)\to R(G)$ is the holomorphic induction
map.  (See e.g.\ \cite{demazure;simple-bott}.)  Let us apply this to
the group $G^{(\sigma)}$ and recall that holomorphic induction
preserves the level.  (This follows from Lemma \ref{lemma;sections}.)
Let us also use the formula
$j_{G,*}(e^\lambda)=j_{G,\dirac}(e^{\lambda-\rho_G})$ of Example
\ref{example;holomorphic-induction}.  The following statement results:
for each level $1$ character $\lambda$ of $T^{(\sigma)}$ there is at
most one $v\in W_G$ such that $v(\lambda)-\rho_G$ is $G$-dominant; if
no such $v$ exists, then $j_{G,*}(e^\lambda)=0$; if $v$ exists, then
\begin{equation}\label{equation;borel-weil-bott}
j_{G,*}(e^\lambda)=\det(v)[V_G(v(\lambda)-\rho_G,\sigma)].
\end{equation}
This statement applies to the group $H^{(\sigma+\omega_M)}$.  Taking
an $H$-dominant level $1$ character $\mu$ of $T^{(\sigma+\omega_M)}$
and putting $\lambda=\mu+\rho_H$, we find $v=1$ and
$$[V_H(\mu,\sigma+\omega_M)]=j_{H,*}(e^{\mu+\rho_H}).$$
Applying $i_*$ to both sides and using Theorem \ref{theorem;functor}
yields
$$
i_*([V_H(\mu,\sigma+\omega_M)])=i_*j_{H,*}(e^{\mu+\rho_H})
=j_{G,*}(e^{\mu+\rho_H}).
$$
The result now follows from \eqref{equation;borel-weil-bott} plus the
observation that, if $w\in W_G$ maps the strictly $H$-dominant
character $\lambda=\mu+\rho_H$ to a strictly $G$-dominant element
$w(\lambda)$, then $w\in W^H$.
\end{proof}

\begin{example}\label{example;irreducible}
The character $\rho_M$ is a highest weight of the spinor module $S$
(see~\eqref{equation;weight}) and hence is $H$-dominant.  Therefore
$i_*([V_H(\rho_M,\omega_M)])=1$ by Theorem~\ref{theorem;irreducible}.
\end{example}

In the $\Spin$ case Kostant \cite{kostant;cubic-dirac-multiplet},
\cite{kostant;bott-borel-weil-multiplet} has established a refinement
of Theorem \ref{theorem;irreducible}, which amounts to a vanishing
theorem for the kernel or the cokernel of the Dirac operator (taken
with respect to a very special connection), and which yields a
$\Z/2\Z$-graded generalization of the classical Borel-Weil-Bott
theorem.  Kostant's assumption of the existence of a $\Spin$-structure
on $G/H$ can be removed by means of our techniques, but we will not
pursue that avenue here.

%%%%%%%%%%%%%%%%%%%%%%%%%%%%%%%%%%%%%%%%%%%%%%%%%%%%%%%%%%%%%%%%%%%%%%%%
\subsection{Duality and induction}\label{subsection;duality}
%%%%%%%%%%%%%%%%%%%%%%%%%%%%%%%%%%%%%%%%%%%%%%%%%%%%%%%%%%%%%%%%%%%%%%%%

There is a natural pairing between twisted representation modules
defined by means of twisted $\Spin^c$-induction.  As before, we let
$\tau$ be a central extension of $H$ by $\U(1)$.  Define
$$\ca{P}\colon R(H,\tau)\times R(H,\omega_M-\tau)\longto R(G)$$
by $\ca{P}(a_1,a_2)=i_*(a_1a_2)$ for $a_1\in R(H,\tau)$ and $a_2\in
R(H,\omega_M-\tau)$.  It follows from the $R(G)$-linearity of the
induction map $i_*$ that $\ca{P}$ is $R(G)$-bilinear.  Recall that
$A\spcheck$ denotes the dual of an $R(G)$-module $A$.  We call the
pairing \emph{nonsingular} if the two $R(G)$-linear maps
$$
\ca{P}^\sharp\colon R(H,\tau)\longto R(H,\omega_M-\tau)\spcheck,\qquad
R(H,\omega_M-\tau)\longto R(H,\tau)\spcheck
$$
induced by $\ca{P}$ are isomorphisms.  The map $\ca{P}^\sharp$ fits
into a diagram
\begin{equation}\label{equation;summand}
\vcenter{\xymatrix{
R(T)\ar[r]^-{\ca{P}_\dirac^\sharp}\ar@<0.5ex>[d]^\pi&
R(T)\spcheck\ar@<0.5ex>[d]^{\pi\spcheck}\\
R(H,\tau)\ar[r]^-{\ca{P}^\sharp}\ar@<0.5ex>[u]^\iota&
R(H,\omega_M-\tau)\spcheck\ar@<0.5ex>[u]^{\iota\spcheck}.
}}
\end{equation}
Here $\ca{P}_\dirac$ is the $R(G)$-bilinear pairing on $R(T)$ defined
by holomorphic induction (cf.\ Example
\ref{example;holomorphic-induction}),
$$
\ca{P}_\dirac(b_1,b_2)=j_{G,\dirac}(b_1b_2)=j_{G,*}(e^{\rho_G}b_1b_2)
$$
for $b_1$, $b_2\in R(T)$.  Choose a character $\mu$ of level $1$ of
the extended torus $T^{(\tau)}$.  Then $e^\mu$ is a free generator of
the $R(T)$-module $R(T,\tau)$ (see Example \ref{example;torus}), and
we define $\iota$, $\pi$, $\iota\spcheck$, and $\pi\spcheck$ to be the
compositions of the following $R(H)$-linear maps:
\begin{gather}
\label{equation;iota}
\iota\colon\xymatrix{R(H,\tau)\ar[r]^-{j_H^*}&
R(T,\tau)\ar[r]^-{e^{-\mu}}&R(T),}
\\
\label{equation;pi}
\pi\colon\xymatrix{R(T)\ar[r]^-{e^{\rho_H+\mu}}&
R(T,\omega_{H/T}+\tau)\ar[r]^-{j_{H,*}}&R(H,\tau),}
\\
\notag
\iota\spcheck\colon\xymatrix@C=35pt{R(H,\omega_M-\tau)\spcheck
\ar[r]^-{^t\mspace{-1mu}(j_{H,*})}&
R(T,\omega_{G/T}-\tau)\spcheck\ar[r]^-{^t\mspace{-1mu}(e^{\rho_G-\mu})}&
R(T)\spcheck,}
\\
\notag
\pi\spcheck\colon\xymatrix@C=35pt{R(T)\spcheck
\ar[r]^-{^t\mspace{-1mu}(e^{-\rho_M+\mu})}&
R(T,\omega_M-\tau)\spcheck\ar[r]^-{^t\mspace{-1mu}(j_H^*)}&
R(H,\omega_M-\tau)\spcheck.}
\end{gather}
Here $e^\lambda$ denotes ``multiplication by $e^\lambda$'' and
$^t\mspace{-1mu}(e^\lambda)$ the transpose operator, etc.  We show
next that the diagram \eqref{equation;summand} commutes in two
different ways and that the $\iota$'s are sections of the $\pi$'s.

\begin{lemma}\label{lemma;commute}
We have
$$
\iota\spcheck\circ\ca{P}^\sharp=\ca{P}_\dirac^\sharp\circ\iota,\qquad
\pi\spcheck\circ\ca{P}_\dirac^\sharp=\ca{P}^\sharp\circ\pi,\qquad
\pi\circ\iota=\id,\qquad\pi\spcheck\circ\iota\spcheck=\id.
$$
In particular the $R(H)$-module $R(H,\tau)$ is isomorphic to a direct
summand of $R(T)$.
\end{lemma}

\begin{proof}
Let $a\in R(H,\tau)$ and $b\in R(T)$.  Then
$$
\iota\spcheck(\ca{P}^\sharp(a))(b)
=\ca{P}^\sharp(a)(j_{H,*}(e^{\rho_G-\mu}b))
=i_*(aj_{H,*}(e^{\rho_G-\mu}b))
$$
and
$$
\ca{P}_\dirac^\sharp(\iota(a))(b)
=j_{G,*}(e^{\rho_G}e^{-\mu}j_H^*(a)b)
=i_*(j_{H,*}(e^{\rho_G-\mu}j_H^*(a)b))
=i_*(aj_{H,*}(e^{\rho_G-\mu}b)),
$$
where we used functoriality of induction, $j_{G,*}=i_*\circ j_{H,*}$
(Theorem \ref{theorem;functor}).  Therefore
$\iota\spcheck\circ\ca{P}^\sharp=\ca{P}_\dirac^\sharp\circ\iota$.  The
proof that
$\pi\spcheck\circ\ca{P}_\dirac^\sharp=\ca{P}^\sharp\circ\pi$ is
similar.  Let $a\in R(H)$.  Then
$$
\pi(\iota(a))=j_{H,*}(e^{\rho_H+\mu}e^{-\mu}j_H^*(a))
=j_{H,*}(e^{\rho_H}j_H^*(a))=j_{H,*}(e^{\rho_H})a=a,
$$
because $j_{H,*}(e^{\rho_H})=1$
% (e.g.\ by Theorem \ref{theorem;irreducible}).
The proof that $\pi\spcheck\circ\iota\spcheck=\id$ is similar.
\end{proof}

We can now establish a basic structure result for the twisted modules
$R(H,\tau)$, which generalizes theorems of Pittie
\cite{pittie;homogeneous-vector} and Steinberg \cite{steinberg;pittie}
and Kazhdan and Lusztig
\cite{kazhdan-lusztig;deligne-langlands-hecke}.

\begin{theorem}\label{theorem;nonsingular}
Assume that $\pi_1(G)$ is torsion-free.  Then the pairing $\ca{P}$ is
nonsingular and $R(H,\tau)$ is a free $R(G)$-module of rank
$\abs{W^H}$.
\end{theorem}

\begin{proof}
Kazhdan and Lusztig
\cite[Proposition~1.6]{kazhdan-lusztig;deligne-langlands-hecke} showed
that the map $\ca{P}_\dirac^\sharp$ is an isomorphism.  Together with
Lemma \ref{lemma;commute} this proves that $\ca{P}^\sharp$ is also an
isomorphism.  Upon replacing $\tau$ with $\omega_M-\tau$ we see that
the map
$$R(H,\omega_M-\tau)\longto R(H,\tau)\spcheck$$
induced by $\ca{P}$ is likewise an isomorphism.  Therefore $\ca{P}$ is
nonsingular.  The Pittie-Steinberg theorem
\cite{pittie;homogeneous-vector}, \cite{steinberg;pittie} says that
$R(H)$ is a free $R(G)$-module of rank $\abs{W^H}$.  By Lemma
\ref{lemma;commute}, $R(H,\tau)$ is isomorphic to a direct summand of
$R(T)$, so applying Pittie-Steinberg to $H=T$ we find that $R(H,\tau)$
is a projective $R(G)$-module.  Moreover, since $R(T)$ is finitely
generated and $R(G)$ is a Noetherian ring, $R(H,\tau)$ is finitely
generated.  Steinberg proved also that $R(G)$ is the tensor product of
a polynomial algebra and a Laurent polynomial algebra over $\Z$.  The
assertion that $R(H,\tau)$ is free now follows from the theorem of
Quillen \cite{quillen;projective-polynomial} and Suslin
\cite{suslin;projective-polynomial-free}, which states that finitely
generated projective modules over such rings are free.  (Quillen and
Suslin proved their result for polynomial algebras over a principal
ideal domain and Suslin remarked that his proof works just as well for
mixed polynomial and Laurent polynomial algebras.  See also
\cite[Chapter~V]{lam;serre-problem}.)  Lemma \ref{lemma;rank} (see
Appendix \ref{section;central}) states that
$\rank_{R(H)}(R(H,\tau))=1$.  By the Pittie-Steinberg theorem we
conclude that
$$
\rank_{R(G)}(R(H,\tau))=\rank_{R(G)}(R(H))\cdot\rank_{R(H)}(R(H,\tau))
$$
is equal to $\abs{W^H}$.
\end{proof}

In general it is not true that the twisted module $R(H,\tau)$ is free
over $R(H)$.  (See Example \ref{example;not-free}.)  Steinberg
\cite{steinberg;pittie} has constructed an explicit basis of the
$R(G)$-module $R(H)$.  We do not know if there is a similar
construction for a basis of $R(H,\tau)$.

The next result, which is in essence
\cite[Theorem~III]{bott;homogeneous-differential} , highlights the
contrast between maximal rank subgroups and other subgroups of $G$
(cf.\ Theorem \ref{theorem;linear}\eqref{item;zero}): in the maximal
rank case \emph{every} twisted induction map (at least for $\sigma=0$)
arises from twisted equivariant elliptic operators.

\begin{theorem}\label{theorem;dirac-induction}
Assume that $\pi_1(G)$ is torsion-free.  Then the group of twisted
induction maps $R(H,\tau)\spcheck$ is a free $R(G)$-module of rank
$\abs{W^H}$.  Every twisted induction map $f\in R(H,\tau)\spcheck$ is
of the form $f=\ca{P}^\sharp(a)$ for a unique $a\in
R(H,\omega_M-\tau)$.  Choose $H^{(\omega_M-\tau)}$-modules $V_0$ and
$V_1$ of level $1$ such that $a=[V_0]-[V_1]$.  Let $D_0$ and $D_1$ be
the twisted Dirac operators defined by the equivariant Clifford
modules $S\otimes V_0$, resp.\ $S\otimes V_1$.  Then
$f=i_{D_0}-i_{D_1}$.
\end{theorem}

\begin{proof}
The first two assertions follow immediately from Theorem
\ref{theorem;nonsingular}.  We have $i_{D_0}(b)=i_*([V_0]b)$ and
$i_{D_1}(b)=i_*([V_1]b)$ for all $b\in R(H,\tau)$ by Example
\ref{example;twisted-dirac}.  Hence
$f(b)=i_*(ab)=i_{D_0}(b)-i_{D_1}(b)$ for all $b$.
\end{proof}

%%%%%%%%%%%%%%%%%%%%%%%%%%%%%%%%%%%%%%%%%%%%%%%%%%%%%%%%%%%%%%%%%%%%%%%%
\section{Applications to K-theory}\label{section;k}
%%%%%%%%%%%%%%%%%%%%%%%%%%%%%%%%%%%%%%%%%%%%%%%%%%%%%%%%%%%%%%%%%%%%%%%%

The results of the previous sections lead to some direct consequences
in equivariant K-theory.  We use the same notation as in
\S\S\,\ref{section;induction}--\ref{section;properties}.  Recall that
$G$ denotes a compact connected Lie group, $H$ a closed connected
subgroup of maximal rank, $T$ a common maximal torus of $H$ and $G$,
and $\tau$ a central extension of $H$ by $\U(1)$.  In addition we
denote by $X$ a compact topological $G$-space,%
\glossary{X@$X$, topological $G$-space}
by $K_G^*(X)$ its $G$-equivariant K-ring, and by $K_H^*(X,\tau)$ its
$\tau$-twisted $H$-equivariant K-group.

%%%%%%%%%%%%%%%%%%%%%%%%%%%%%%%%%%%%%%%%%%%%%%%%%%%%%%%%%%%%%%%%%%%%%%%%
\subsection{The K\"unneth theorem and duality}
\label{subsection;poincare}
%%%%%%%%%%%%%%%%%%%%%%%%%%%%%%%%%%%%%%%%%%%%%%%%%%%%%%%%%%%%%%%%%%%%%%%%

We establish an equivariant K\"unneth formula, Proposition
\ref{proposition;kunneth}, which is a minor variation on results of
Hodgkin \cite{hodgkin;equivariant-kunneth}, Snaith
\cite{snaith;kunneth}, McLeod \cite{mcleod;kunneth-equivariant}, and
Rosenberg and Schochet \cite{rosenberg-schochet;equivariant-k-theory},
and which allows us to extend the results of \S\,\ref{section;dirac}
to equivariant K-theory.  This formula implies that the
$G$-equivariant K-group $K_G^*(X)$ is a direct summand of the
(untwisted) $H$-equivariant K-group $K_H^*(X)$, and in Theorem
\ref{theorem;invariant} we identify this direct summand in terms of
linear equations.  Theorem \ref{theorem;duality} is a duality theorem,
which generalizes the nonsingular pairing $\ca{P}$ from representation
rings to equivariant K-theory.

First we extend the definition of twisted induction maps to K-theory.
The action map
$$A\colon G\times X\to X$$
defined by $A(g,x)=gx$ is $G$-equivariant with respect to the left
multiplication action on $G$ and $H$-invariant with respect to the
action $h\cdot(g,x)=(gh^{-1},hx)$.  In the sense of
\cite[\S\,1]{atiyah-singer;index-elliptic-IV;ann} the triple
$\group{X}=(G\times X,X,A)$ is a $G\times H$-equivariant family of
smooth manifolds over $X$ with fibre $G$ and projection $A$.  The
$H$-action on $\group{X}$ is free and the quotient family
$\tilde{\group{X}}=(G\times^HX,X,\tilde{A})$ is a $G$-equivariant
family of smooth manifolds over $X$ with fibre $M=G/H$ and projection
map $\tilde{A}\colon\tilde{\group{X}}\to X$ given by
$\tilde{A}([g,x])=gx$.

Now let $U$ be a $\Z/2\Z$-graded $H^{(\tau)}$-module of level $1$ and
let
$$
D\colon\Gamma\bigl(G,G\times U^0\bigr)\longto\Gamma\bigl(G,G\times
U^1\bigr)
$$
be a twisted equivariant elliptic differential operator on $M$.  Let
$V$ be an $H^{\tau}$-equivariant vector bundle of level $-1$ over $X$.
Then $E=U\otimes V$ is a $\Z/2\Z$-graded $H$-equivariant vector bundle
over $X$.  The pullback $\group{E}=A^*E$ is a $\Z/2\Z$-graded $G\times
H$-equivariant vector bundle over the family $\group{X}$.  It is shown
in \cite[\S\,2]{atiyah;bott-periodicity-index} that there exists a
$G\times H$-equivariant family of differential operators
$$
\group{D}_{V,x}\colon\Gamma\bigl(A^{-1}(x),\group{E}^0|A^{-1}(x)\bigr)
\longto\Gamma\bigl(A^{-1}(x),\group{E}^1|A^{-1}(x)\bigr)
$$
whose symbol is equal to
$$\sym(\group{D}_V)=\sym(D)\otimes\id_V.$$
It follows that the family $\group{D}_V$ is $H$-transversely elliptic.
The $H$-action being free, $\group{D}_V$ descends to a $G$-equivariant
elliptic family
$$
\tilde{\group{D}}_{V,x}\colon
\Gamma\bigl(\tilde{A}^{-1}(x),\tilde{\group{E}}^0|\tilde{A}^{-1}(x)\bigr)
\longto
\Gamma\bigl(\tilde{A}^{-1}(x),\tilde{\group{E}}^1|\tilde{A}^{-1}(x)\bigr)
$$
over $\tilde{\group{X}}$ with coefficients in the quotient bundle
$\tilde{\group{E}}=G\times^HE$.  We define the induction map
$$i_D\colon K_H^*(X,-\tau)\longto K_G^*(X)$$
by $i_D([V])=\index(\tilde{\group{D}}_V)$.  

\begin{remark}
The family of manifolds $\tilde{\group{X}}$ is trivial (a
trivialization is given by mapping a class $[g,x]\in\tilde{\group{X}}$
to the pair $(gH,gx)\in G/H\times X$), but the family of vector
bundles $\tilde{\group{E}}$ is not trivial unless the $H$-equivariant
vector bundle $E$ over $X$ is $G$-equivariant.  Thus
$\tilde{\group{D}}_V$ is a nontrivial family of operators.
\end{remark}

More generally, if $\sigma$ is a central extension of $G$ by $\U(1)$
and $D$ is a $(\sigma,\tau)$-twisted operator, this construction gives
rise to an induction map
$$i_D\colon K_H^*(X,\sigma-\tau)\longto K_G^*(X,\sigma),$$
which is linear over the ring $K_G^*(X)$.  Taking $D$ to be the
twisted $\Spin^c$ Dirac operator $\Dirac$ we get the looked-for map
\begin{equation}\label{equation;induction-k}
i_*=i_\Dirac\colon K_H^*(X,\sigma+\omega_M)\longto K_G^*(X,\sigma).
\end{equation}
Replacing $G$ with $H$ and $H$ with $T$ gives the map $j_{H,*}\colon
K_T^*(X,\tau+\omega_{H/T})\to K_H^*(X,\tau)$, which occurs in the
proof of the following K\"unneth formula.

\begin{proposition}\label{proposition;kunneth}
Assume that $\pi_1(G)$ is torsion-free.  Then the map
$$\phi\colon R(H,\tau)\otimes_{R(G)}K_G^*(X)\longto K_H^*(X,\tau)$$
defined by $\phi(u\otimes b)=u\cdot i^*(b)$ is an isomorphism of
$\Z/2\Z$-graded $R(H)$-modules.  Hence $K_H^*(X,\tau)$ is a free
module over $K_G^*(X)$ of rank $\abs{W^H}$.
\end{proposition}

\begin{proof}
Evidently $\phi$ is an $R(H)$-linear morphism of degree $0$.  To prove
that it is bijective we consider a diagram similar to
\eqref{equation;summand},
$$
\xymatrix{
R(T)\otimes_{R(G)}K_G^*(X)\ar[r]^-\psi\ar@<0.5ex>[d]^{\pi\otimes\id}&
K_T^*(X)\ar@<0.5ex>[d]^\pi\\
R(H,\tau)\otimes_{R(G)}K_G^*(X)\ar[r]^-\phi\ar@<0.5ex>[u]^{\iota\otimes\id}&
K_H^*(X,\tau)\ar@<0.5ex>[u]^\iota.
}
$$
Here $\psi$ is defined by $\psi(v\otimes b)=v\cdot j_G^*(b)$ and the
maps $\iota$ and $\pi$ on the left are defined in
\eqref{equation;iota} and \eqref{equation;pi}.  Analogously the maps
$\iota$ and $\pi$ on the right are defined to be the compositions
\begin{gather*}
\iota\colon\xymatrix{K_H^*(X,\tau)\ar[r]^-{j_H^*}&
K_T^*(X,\tau)\ar[r]^-{e^{-\mu}}&K_T^*(X),}
\\
\pi\colon\xymatrix{K_T(X)\ar[r]^-{e^{\rho_H+\mu}}&
K_T^*(X,\omega_{H/T}+\tau)\ar[r]^-{j_{H,*}}&K_H^*(X,\tau).}
\end{gather*}
As in Lemma \ref{lemma;commute} we have the properties
$$
\iota\circ \phi=\psi\circ(\iota\otimes\id),\qquad\pi\circ
\psi=\phi\circ(\pi\otimes\id),
$$
which follow from the $K_H^*(X)$-linearity of the restriction map
$j_H^*$ and the induction map $j_{H,*}$; and the property
$\pi\circ\iota=\id$, which follows from $j_{H,*}(e^{\rho_H})=1$.  The
map $\psi$ is an isomorphism by the usual equivariant K\"unneth
theorem of \cite{rosenberg-schochet;equivariant-k-theory}.  Hence
$\phi$ is also an isomorphism.  The second statement of the
proposition now follows from Theorem \ref{theorem;nonsingular}.
\end{proof}

\begin{example}\label{example;kunneth}
Assume that $\pi_1(G)$ is torsion-free.  Take $X=G$ and let $G$ act by
right multiplication.  Then $K_H^*(X,\tau)\cong K^*(X/H,\tau)$ because
the action is free, so Proposition \ref{proposition;kunneth} gives an
isomorphism
$$K^*(G/H,\tau)\cong R(H,\tau)\otimes_{R(G)}\Z.$$
Taking $X=G/H$ and letting $H$ act by right multiplication gives an
isomorphism
$$K_H^*(G/H,\tau)\cong R(H,\tau)\otimes_{R(G)}R(H).$$
\end{example}

The equivariant K\"unneth formula expresses $H$-equivariant K-the\-ory
in terms of $G$-equivariant K-theory, but it also enables us to do the
reverse.  Let $\ca{E}=\End_{R(G)}(R(H))$ be the ring of $R(G)$-linear
endomorphisms of $R(H)$ and let $I(\ca{E})$ be the left ideal
$$I(\ca{E})=\{\,\Delta\in\ca{E}\mid\Delta(1)=0\,\}.$$
For any left $\ca{E}$-module $A$ let $A^{I(\ca{E})}$ denote the
subgroup of elements annihilated by $I(\ca{E})$.  The K\"unneth
isomorphism, Proposition \ref{proposition;kunneth}, shows that
$K_H^*(X)$ is in a natural way a left $\ca{E}$-module if $\pi_1(G)$ is
torsion-free.  A version of the next theorem for the maximal torus
$H=T$ was proved in \cite[Theorem~4.6]%
{harada-landweber-sjamaar;divided-differences-character}.  (However,
the version for the maximal torus is stronger: it is less restrictive
in that it does not require the assumption on $\pi_1(G)$, and it is
more explicit in that there is a description of the left ideal
$I(\ca{E})$ in terms of divided difference operators.)

\begin{theorem}\label{theorem;invariant}
Assume that $\pi_1(G)$ is torsion-free.  Then the map $i^*\colon
K_G^*(X)\to K_H^*(X)$ is an isomorphism onto $K_H^*(X)^{I(\ca{E})}$.
\end{theorem}

\begin{proof}
The group $R(H)$ is an $\ca{E}$-$R(G)$-bimodule.  By the
Pittie-Steinberg theorem \cite{pittie;homogeneous-vector},
\cite{steinberg;pittie}, $R(H)$ is free of finite rank $\abs{W^H}$ as
a $R(G)$-module and therefore it is a progenerator of the category of
$R(G)$-modules.  Hence, by the first Morita equivalence theorem (see
e.g.\ \cite[\S\,18]{lam;modules-rings}), the functor 
$$\lie{G}\colon B\mapsto R(H)\otimes_{R(G)}B$$
is an equivalence from the category of left $R(G)$-modules to the
category of left $\ca{E}$-modules, whose inverse is the functor
$$\lie{F}\colon A\mapsto\Hom_{\ca{E}}(R(H),A).$$
As in the proof of \cite[Theorem
4.6]{harada-landweber-sjamaar;divided-differences-character} (which
deals with the case $H=T$) one shows that the functor $\lie{F}$ is
isomorphic to the functor $\lie{I}\colon A\mapsto A^{I(\ca{E})}$.  We
conclude that $B\cong\lie{I}(\lie{G}(B))$ for all $R(G)$-modules $B$.
In particular we can take $B=K_G^*(X)$.  Then $\lie{G}(B)\cong
K_H^*(X)$ by Proposition \ref{proposition;kunneth}, so
$B\cong\lie{I}(K_H^*(X))=K_H^*(X)^{I(\ca{E})}$.
\end{proof}

The pairing $\ca{P}$ defined in \S\,\ref{subsection;duality}
generalizes to a bi-additive pairing 
$$
\ca{P}_X\colon K_H^*(X,\tau)\times K_H^*(X,\omega_M-\tau)\longto
K_G^*(X)
$$
defined by $\ca{P}_X(a_1,a_2)=i_*(a_1a_2)$ for $a_1\in K_H^*(X,\tau)$
and $a_2\in K_H^*(X,\omega_M-\tau)$.  It follows from the naturality
of $i_*$ that
\begin{equation}\label{equation;natural-pairing}
\ca{P}_X(f^*(a_1),f^*(a_2))=f^*\ca{P}_Y(a_1,a_2)
\end{equation} 
for $a_1\in K_H^*(Y,\tau)$ and $a_2\in K_H^*(Y,\omega_M-\tau)$, where
$f\colon X\to Y$ is any $G$-equivariant continuous map.

\begin{theorem}\label{theorem;duality}
Assume that $\pi_1(G)$ is torsion-free.  Then the pairing $\ca{P}_X$
is nonsingular.  Hence
$$
K_H^*(X,\tau)\cong
\Hom_{K_G^*(X)}\bigl(K_H^*(X,\omega_M-\tau),K_G^*(X)\bigr)
$$
as $\Z/2\Z$-graded left $K_G^*(X)$-modules.
\end{theorem}

\begin{proof}
This follows from the nonsingularity of the pairing $\ca{P}$ (Theorem
\ref{theorem;nonsingular}), the K\"unneth theorem (Proposition
\ref{proposition;kunneth}) and the naturality property
\eqref{equation;natural-pairing}.  (See \cite[Proposition
5.1]{harada-landweber-sjamaar;divided-differences-character} for the
case $H=T$.)
\end{proof}

\begin{example}\label{example;poincare}
Assume that $\pi_1(G)$ is torsion-free.  Let $\tau=0$.  Theorem
\ref{theorem;duality} contains as special cases various forms of
Poincar\'e duality in K-theory.  For instance, taking $X=G$ with $H$
acting by right multiplication gives a nonsingular pairing
$$K^*(G/H)\times K^*(G/H,\omega_{G/H})\longto\Z.$$
Taking $X=G/H$ with $H$ acting by left multiplication gives a
nonsingular pairing
$$K_H^*(G/H)\times K_H^*(G/H,\omega_{G/H})\longto R(H).$$
If $G/H$ has an invariant $\Spin^c$-structure, then $\omega_{G/H}=0$,
so we have a nonsingular $\Z$-valued pairing on $K^*(G/H)$ and a
nonsingular $R(H)$-valued pairing on $K_H^*(G/H)$.
\end{example}

%%%%%%%%%%%%%%%%%%%%%%%%%%%%%%%%%%%%%%%%%%%%%%%%%%%%%%%%%%%%%%%%%%%%%%%%
\subsection{GKRS multiplets}\label{subsection;multiplet}
%%%%%%%%%%%%%%%%%%%%%%%%%%%%%%%%%%%%%%%%%%%%%%%%%%%%%%%%%%%%%%%%%%%%%%%%
 
In this section we show how each $T$-equivari\-ant K-class on a
$G$-space $X$ gives birth to a litter of $T$-equivariant classes
parametrized by $W^H$, which we call a
\emph{Gross\--Kostant\--Ramond\--Sternberg multiplet}.  This is an
extension of the notion of a multiplet introduced in
\cite{gross-kostant-ramond-sternberg} from a one-point space to an
arbitrary $G$-space $X$.  Generalizing a result of
\cite{gross-kostant-ramond-sternberg}, we show that the alternating
sum of a multiplet, when mapped to ordinary K-theory, vanishes.

We start by rewriting \eqref{equation;anti-invariant} in the
``opposite'' way.  Define $J_M^\op\in\Z[W_G]$ by%
\glossary{Jmop@$J_M^\op$, ``opposite'' of $J_M$}
$$J_M^\op=\sum_{w\in W^H}\det(w)w^{-1}.$$
By decomposing $w\in W$ as $w=w''(w')^{-1}$ with $w'\in W^H$ and
$w''\in W_H$ we obtain
\begin{equation}\label{equation;opposite-antisymmetrizer}
J_G=J_HJ_M^\op.
\end{equation}
Let us write $\partial_G=j_G^*\circ j_{G,*}$.  This is an
$R(G)$-linear operator from $R(T,\omega_{G/T})$ to $R(T)$, and the
Weyl character formula for $G$ states that $\partial_G(a)=J_G(a)/\d_G$
for all $a\in R(T,\omega_{G/T})$.  Similarly, for the group $H$ we
have $\partial_H(b)=J_H(b)/\d_H$ for all $b\in R(T,\omega_{H/T})$.
Combining this with \eqref{equation;opposite-antisymmetrizer} and
substituting Lemma \ref{lemma;weight-euler}\eqref{item;euler-twist}
yields
$$
\partial_G(a)=\frac{J_G(a)}{\d_G}=\frac1{j_H^*(\e(\Dirac_M)^*)}
\frac{J_H\bigl(J_M^\op(a)\bigr)}{\d_H}
=\frac1{j_H^*(\e(\Dirac_M)^*)}\partial_H\bigl(J_M^\op(a)\bigr)
$$
for all $a\in R(T,\omega_{G/T})$.  This amounts to an identity of
$R(G)$-linear operators, namely
\begin{equation}\label{equation;gkrs}
j_H^*(\e(\Dirac_M)^*)\partial_G=\partial_H\circ J_M^\op,
\end{equation}
where the first operator on the left is multiplication by
$j_H^*(\e(\Dirac_M)^*)$.  Except for the twists by the various
orientation systems, which makes it work at the group level as opposed
to the Lie algebra level, this identity is formula (1) in
\cite{gross-kostant-ramond-sternberg}.

Because the induction maps $j_{G,*}$ and $j_{H,*}$ are defined in
K-theory (see \eqref{equation;induction-k}), the operators
$\partial_G$ and $\partial_H$ make sense in K-theory.  The operator
$\partial_G$ maps $K_T^*(X,\omega_{G/T})$ to $K_T^*(X)$ and
$\partial_H$ maps $K_T^*(X,\omega_{G/T})$ to
$$
K_T^*(X,\omega_{G/T}-\omega_{H/T})=K_T^*(X,\omega_M).
$$
The next lemma means that we can substitute in the identity
\eqref{equation;gkrs} any class $a\in K_T^*(X,\omega_{G/T})$.

\begin{lemma}\label{lemma;gkrs}
The following diagram commutes:
$$
\xymatrix{
K_T^*(X,\omega_{G/T})\ar[r]^-{\partial_G}\ar[d]_-{J_M^\op}&
K_T^*(X)\ar[d]^-{j_H^*(\e(\Dirac_M)^*)}
\\
K_T^*(X,\omega_{G/T})\ar[r]^-{\partial_H}&K_T^*(X,\omega_M).
}
$$
\end{lemma}

\begin{proof}
First assume that $\pi_1(G)$ is torsion-free.  Then the result follows
by combining \eqref{equation;gkrs} with the K\"unneth formula
(Proposition \ref{proposition;kunneth}) and the $K_G^*(X)$-linearity
of $\partial_G$ and $\partial_H$.  If $\pi_1(G)$ is not torsion-free,
we choose a covering $\phi\colon\tilde{G}\to G$ of $G$ by a compact
connected $\tilde{G}$ such that $\pi_1(\tilde{G})$ is torsion-free,
and we let $\tilde{T}$ be the maximal torus $\phi^{-1}(T)$ of
$\tilde{G}$.  Any extension $T^{(\tau)}$ of $T$ induces an extension
$\tilde{T}^{(\tau)}$ of $\tilde{T}$.  It follows from
\cite[Lemma~2.4]{snaith;kunneth} that the pullback map
$$
\phi^*\colon K_T^*(X,\tau)\longto K_{\tilde{T}}^*(X,\tau)
$$
is injective.  Since the kernel of $\phi$ is a central subgroup of
$\tilde{G}$ and a $W$-invariant subgroup of $\tilde{T}$, the pullback
map is natural with respect to each of the maps occurring in the
identity \eqref{equation;gkrs}:
$\phi^*\circ\partial_G=\partial_{\tilde{G}}\circ\phi^*$, etc.
Therefore the commutativity of the diagram for the group $G$ follows
from the commutativity for $\tilde{G}$.
\end{proof}

For $w\in W^H$ and $a\in K_T^*(X,\omega_{G/T})$ put
$$a_w=\partial_H(w^{-1}(a))\in K_T^*(X,\omega_M).$$
We call the $W^H$-tuple $(a_w)_{w\in W^H}$ the \emph{multiplet}
generated by $a$.  Let $f\colon1\to T$ be the trivial homomorphism,
which induces the forgetful map
$$f^*\colon K_T^*(X,\omega_M)\longto K^*(X,\omega_M).$$

\begin{theorem}\label{theorem;gkrs}
Suppose that $H\ne G$.  Let $(a_w)_{w\in W^H}$ be a multiplet in
$K_T^*(X,\omega_M)$.  Then $\sum_{w\in W^H}\det(w)f^*(a_w)=0$.
\end{theorem}

\begin{proof}
Since $H\ne G$, we have $\ca{R}_M\ne\emptyset$ and hence
$$
f^*j_H^*(\e(\Dirac_M)^*)
=\prod_{\alpha\in\ca{R}_M^+}f^*\bigl(e^{-\alpha/2}-e^{\alpha/2}\bigr)
=\prod_{\alpha\in\ca{R}_M^+}(1-1)=0
$$
by Lemma \ref{lemma;weight-euler}\eqref{item;euler-twist}.  Therefore
$$
\sum_{w\in W^H}\det(w)f^*(a_w)=f^*\partial_HJ_M^\op(a)
=f^*\bigl(j_H^*(\e(\Dirac_M)^*)\partial_G(a)\bigr)=0
$$
by Lemma \ref{lemma;gkrs}.
\end{proof}

This result reduces to that of \cite{gross-kostant-ramond-sternberg}
by pulling back a multiplet on $X$ to a point, i.e.\ by applying the
augmentation map $K_T^*(X,\omega_M)\to R(T,\omega_M)$.  It must be
said that this K-theory version of the multiplet theorem is much
weaker than the original version and the later version of
\cite{kostant;cubic-dirac-multiplet}.  For instance, there is no
telling whether the elements of a multiplet are distinct or even
nonzero.  (However, suppose that $\pi_1(G)$ is torsion-free and that
$a$ is of the form $a=e^\lambda j_G^*(b)$, where $\lambda$ is a
strictly $G$-dominant character and $b\in K_G^*(X)$ is nonzero.  Then
$a_w=\partial_H(w^{-1}(e^\lambda))j_G^*(b)$.  The elements
$\partial_H(w^{-1}(e^\lambda))$ are restrictions to $T$ of irreducible
$H$-modules with distinct highest weights, and so by the K\"unneth
formula the multiplet elements $a_w$ are distinct and nonzero.)

%%%%%%%%%%%%%%%%%%%%%%%%%%%%%%%%%%%%%%%%%%%%%%%%%%%%%%%%%%%%%%%%%%%%%%%%
\appendix
%%%%%%%%%%%%%%%%%%%%%%%%%%%%%%%%%%%%%%%%%%%%%%%%%%%%%%%%%%%%%%%%%%%%%%%%

%%%%%%%%%%%%%%%%%%%%%%%%%%%%%%%%%%%%%%%%%%%%%%%%%%%%%%%%%%%%%%%%%%%%%%%%
\section{Twisted K-theory}\label{section;twist}
%%%%%%%%%%%%%%%%%%%%%%%%%%%%%%%%%%%%%%%%%%%%%%%%%%%%%%%%%%%%%%%%%%%%%%%%

In this appendix we summarize the necessary facts from twisted
K-theory in a form suited to our purpose.  Like the original treatment
by Donovan and Karoubi \cite{donovan-karoubi;graded-local}, we cover
only K-theory twisted by torsion classes.  We denote by $X$ a compact
topological space and by $G$ and $H$ two (not necessarily connected)
compact Lie groups.  Contrary to our convention elsewhere in the paper
we do not assume $H$ to be a subgroup of $G$.

%%%%%%%%%%%%%%%%%%%%%%%%%%%%%%%%%%%%%%%%%%%%%%%%%%%%%%%%%%%%%%%%%%%%%%%%
\subsection{Twists and twisted vector bundles}\label{subsection;bundle}
%%%%%%%%%%%%%%%%%%%%%%%%%%%%%%%%%%%%%%%%%%%%%%%%%%%%%%%%%%%%%%%%%%%%%%%%

A \emph{twist} of $X$ is a pair $\tau=\bigl(P,H^{(\tau)}\bigr)$, where
$\pr\colon P\to X$ is a principal $H$-bundle over $X$ and $H^{(\tau)}$
is a central extension of $H$ by $\U(1)$.  We regard $P$ as an
$H^{(\tau)}$-space on which the central circle $\U(1)$ acts trivially.
A \emph{$\tau$-twisted vector bundle over $X$} is an
$H^{(\tau)}$-equivariant complex vector bundle over $P$ which is
\emph{of level~$1$}, in the sense that the central circle of
$H^{(\tau)}$ acts on $E$ by scalar multiplication on the fibres.

A \emph{morphism} between $\tau$-twisted vector bundles $E_1$ and
$E_2$ is an $H^{(\tau)}$-equivariant vector bundle homomorphism $E_1\to
E_2$.  We denote by $\lie{Vec}(X,\tau)$ the category of $\tau$-twisted
vector bundles.  This is an additive category, in which there is an
obvious notion of an exact sequence.  Thus we can form the
Grothendieck group of $\lie{Vec}(X,\tau)$, which we will denote by
$K(X,\tau)$.

A twist $\tau$ is a simple example of a \emph{gerbe with band} $\U(1)$
over $X$, and $K(X,\tau)$ is called the \emph{$\tau$-twisted K-group}
of $X$, or the \emph{K-group of $X$ with coefficients in $\tau$}.  See
\cite[\S\,2]{bouwknegt-carey-mathai-murray-stevenson;twisted} or
\cite[\S\,2.5]{tu-xu-laurent-gengoux;twisted-differentiable-stacks}
for more general notions of gerbe and a comparison with other versions
of twisted K-theory.

We can multiply a twisted bundle $F$ by an ordinary vector bundle $E$
on $X$ by the rule $E\cdot F=\pr^*E\otimes F$.  This rule turns
$K(X,\tau)$ into a $K(X)$-module.  More generally, let
$\tau_1=\bigl(P_1,H_1^{(\tau_1)}\bigr)$ and
$\tau_2=\bigl(P_2,H_2^{(\tau_2)}\bigr)$ be two twists of $X$.  The
\emph{sum} of $\tau_1$ and $\tau_2$ is the twist
$\tau=\bigl(P,H^{(\tau)}\bigr)$, where $P$ is the fibred product
$P_1\times_XP_2$, viewed as a principal bundle over $X$ with structure
group $H=H_1\times H_2$, and the central extension $H^{(\tau)}$ is the
quotient of $H_1^{(\tau_1)}\times H_2^{(\tau_2)}$ by the anti-diagonal
copy of $\U(1)$.  Any $\tau_1$- and $\tau_2$-twisted bundles on $X$
can be lifted to $P$ and the tensor product of the lifts is a
$\tau_1+\tau_2$-twisted bundle.  This defines a multiplication law
$$
K(X,\tau_1)\times K(X,\tau_2)\longto K(X,\tau_1+\tau_2).
$$

A twist $\tau$ pulls back under a continuous map $f\colon Y\to X$ in
an evident way, and we have an induced homomorphism
$$f^*\colon K(X,\tau)\longto K(Y,f^*\tau).$$

The twist $\tau$ is \emph{(Morita) trivial} if there exists a
\emph{trivialization}, i.e.\ a principal $H^{(\tau)}$-bundle $Q$ over
$X$ such that $P$ is the quotient of $Q$ by the central circle of
$H^{(\tau)}$.  (If the extension $H^{(\tau)}$ of $H$ is trivial, then
the twist $\tau$ is Morita trivial, but the converse is false.)  If
$\tau$ is Morita trivial, then
\begin{equation}\label{equation;trivial-twist}
K(X,\tau)\cong K_{H^{(\tau)}}(Q)\cong K\bigl(Q/H^{(\tau)}\bigr)\cong
K(X).
\end{equation}
This isomorphism depends on the choice of the trivialization $Q$.

\begin{example}\label{example;orientation}
Let $V$ be an oriented real vector bundle of rank $m$ over $X$
provided with a Riemannian metric.  The \emph{orientation twist} or
\emph{orientation system} associated with $V$ is the pair
$\omega_V=(\SO(V),\Spin^c(m))$.  Here $\SO(V)$ is the oriented
orthogonal frame bundle of $V$, which has structure group $\SO(m)$,
and $\Spin^c(m)$ is the $\Spin^c$-group of the Euclidean space $\R^m$.
The orientation system is Morita trivial precisely when $V$ possesses
a $\Spin^c$-structure, i.e.\ an orientation in K-theory.  If $X$ is an
oriented Riemannian manifold, the \emph{orientation twist $\omega_X$
of $X$} is defined to be the orientation twist of the tangent bundle
of $X$.
\end{example}

\begin{example}\label{example;extension}
Let $\tau$ be a central extension of a compact Lie group $H$ by
$\U(1)$.  By letting $H$ act on itself by right multiplication we can
view $H$ as a principal bundle over the one-point space $X=\pt$.  From
this point of view $\tau$ is nothing but a twist of a point.  As such
it is Morita trivial, and therefore it follows from
\eqref{equation;trivial-twist} that $K(\pt,\tau)\cong\Z$.
\end{example}

Suppose the compact Lie group $G$ acts continuously on $X$.  A twist
$\tau=\bigl(P,H^{(\tau)}\bigr)$ of $X$ is \emph{$G$-equivariant} if
the principal bundle $P$ is $G$-equivariant, i.e.\ equipped with a
$G$-action by bundle maps which lifts the $G$-action on the base $X$
and which commutes with the action of the structure group $H$.  A
\emph{$G$-equivariant $\tau$-twisted vector bundle over $X$} is a
$G\times H^{(\tau)}$-equivariant complex vector bundle over $P$ which
is of level~$1$ with respect to $H^{(\tau)}$.  Such twisted bundles
are the objects of an exact category $\lie{Vec}_G(X,\tau)$, whose
K-group $K_G(X,\tau)$ is the \emph{equivariant} K-group of $X$ with
coefficients in $\tau$.

\begin{example}\label{example;equivariant-orientation}
The orientation twist of an oriented Riemannian vector bundle $V$ is
equivariant with respect to any compact Lie group which acts on $V$ by
orientation-preserving isometric bundle maps.  In particular, the
orientation twist of an oriented Riemannian manifold $X$ is
equivariant with respect to any compact Lie group which acts on $X$ by
orientation-preserving isometries.
\end{example}

\begin{example}\label{example;equivariant-extension}
Let $X=\pt$ and let $\tau$ be as in Example \ref{example;extension}.
We turn $\tau$ into an $H$-equivariant twist by letting $H$ act on $H$
by left multiplication.  Every $H\times H^{(\tau)}$-equivariant vector
bundle $E$ over $H$ trivializes equivariantly to a product bundle
$E\cong H\times U$.  On this product bundle $H$ acts on the base $H$
by left multiplication and trivially on the vector space $U$, and
$H^{(\tau)}$ acts on the base by right multiplication and linearly on
the fibre.  Thus the bundle $E$ is of level~$1$ if and only if the
$H^{(\tau)}$-module $U$ is of level~$1$.  It follows that the category
$\lie{Vec}_H(\pt,\tau)$ is equivalent to the category of level~$1$
$H^{(\tau)}$-modules.  We conclude that $K_H(\pt,\tau)\cong
R(H,\tau)$, the twisted representation module of $H$.
\end{example}

\begin{example}\label{example;equivariant-extension-space}
Generalizing Example \ref{example;equivariant-extension}, we let $X$
be a topological $H$-space and $p\colon X\to\pt$ the constant map.
The $H$-equivariant twist $\tau$ pulls back to the $H$-equivariant
twist $p^*\tau$ on $X$.  We view $X$ as an $H^{(\tau)}$-space on which
the central circle acts trivially.  As in Example
\ref{example;equivariant-extension} one shows that the category
$\lie{Vec}(X,p^*\tau)$ is equivalent to the category of
$H^{(\tau)}$-equivariant level~$1$ vector bundles on $X$.  Thus
$K_H(X,p^*\tau)$ is the Grothendieck group of $H^{(\tau)}$-equivariant
level~$1$ vector bundles on $X$.  To simplify the notation we will
often write this group as $K_H(X,\tau)$.
\end{example}

\begin{example}[induced twists]\label{example;induced-gerbe}
Continuing Example \ref{example;equivariant-extension-space}, we
suppose that $H$ is a subgroup of $G$ and let $i\colon H\to G$ the
inclusion map.  We view the product $G\times X$ as a $G$-equivariant
principal $H$-bundle over the associated bundle $G\times^HX$.  Hence
the pair
$$i_*\tau=\bigl(G\times X,H^{(\tau)}\bigr)$$
is a $G$-equivariant twist of $G\times^HX$, called the twist
\emph{induced} by $\tau$.  If $E$ is a $\tau$-twisted vector bundle
over $X$, then $G\times E$ is a $i_*\tau$-twisted vector bundle over
$G\times^HX$.  The map $E\mapsto G\times E$ defines an equivalence of
categories between $\lie{Vec}(X,\tau)$ and
$\lie{Vec}(G\times^HX,i_*\tau)$.  Thus we have a natural isomorphism
$$K_G(G\times^HX,i_*\tau)\cong K_H(X,\tau).$$
For $X=\pt$ this specializes to $K_G(M,i_*\tau)\cong R(H,\tau)$, where
$M=G/H$.  For simplicity we will often write $K_G(G\times^HX,\tau)$
instead of $K_G(G\times^HX,i_*\tau)$.
\end{example}

%%%%%%%%%%%%%%%%%%%%%%%%%%%%%%%%%%%%%%%%%%%%%%%%%%%%%%%%%%%%%%%%%%%%%%%%
\subsection{Relative twisted K-theory}\label{subsection;relative}
%%%%%%%%%%%%%%%%%%%%%%%%%%%%%%%%%%%%%%%%%%%%%%%%%%%%%%%%%%%%%%%%%%%%%%%%

Relative twisted K-classes are presented by complexes of twisted
vector bundles.  Let $X$ be a compact $G$-space and let
$\tau=\bigl(P,H^{(\tau)}\bigr)$ be a $G$-equivariant twist of $X$.
Consider the category $\lie{Vec}_G^*(X,\tau)$ of bounded complexes
associated with the additive category $\lie{Vec}_G(X,\tau)$.  Thus an
object of $\lie{Vec}_G^*(X,\tau)$ is a $\Z$-graded $G$-equivariant
$\tau$-twisted bundle $E^*$ such that $E^j=0$ for almost all $j$,
furnished with a differential of degree $1$.  Let $Y$ be a closed
$G$-invariant subspace of $X$.  We denote by $\lie{Vec}_G^*(X,Y,\tau)$
the full subcategory of $\lie{Vec}_G^*(X,\tau)$ comprising all objects
$E^*$ with the property that the restriction of $E^*$ to the subspace
$\pr^{-1}(Y)$ of $P$ is an \emph{exact} complex.  The set
$L_G(X,Y,\tau)$ of isomorphism classes of $\lie{Vec}_G^*(X,Y,\tau)$ is
an abelian monoid.  A quotient of $L_G(X,Y,\tau)$ by an appropriate
submonoid (which is defined in the same way as in ordinary K-theory;
see \cite[\S\,3]{segal;equivariant-k-theory;;1968}) is the
\emph{relative} twisted K-group $K_G(X,Y,\tau)$ of $X$.  The Euler
characteristic map $L_G(X,Y,\tau)\to K_G(X,\tau)$ defined by
$[E^*]\mapsto\sum_j(-1)^j[E^j]$ induces a homomorphism
$$K_G(X,Y,\tau)\to K_G(X,\tau),$$
which is an isomorphism if $Y$ is empty.

The group $K_G^0(X,Y,\tau)=K_G(X,Y,\tau)$ is the degree~$0$ part of
the $\Z/2\Z$-graded K-group $K_G^*(X,Y,\tau)$.  The degree~$1$ part is
defined by
$$
K_G^1(X,Y,\tau)
=K_G\bigl(X\times[0,1],(Y\times[0,1])\cup(X\times\{0,1\}),\tau\bigr).
$$
%

%%%%%%%%%%%%%%%%%%%%%%%%%%%%%%%%%%%%%%%%%%%%%%%%%%%%%%%%%%%%%%%%%%%%%%%%
\subsection{The Thom isomorphism}\label{subsection;isomorphism}
%%%%%%%%%%%%%%%%%%%%%%%%%%%%%%%%%%%%%%%%%%%%%%%%%%%%%%%%%%%%%%%%%%%%%%%%

Let $X$ be a compact $G$-space and let $\tau$ be a $G$-equivariant
twist of $X$.  Let $\pi\colon V\to X$ be a $G$-equivariant oriented
real vector bundle of even rank $m=2l$ equipped with an invariant
Riemannian metric, and let $\omega_V$ be the orientation twist of $V$.
We denote the unit ball bundle of $V$ by $BV$, the unit sphere bundle
by $SV$, and the zero section by $\zeta\colon X\to V$.  Let
$P=\SO(V)$.  The spinor module $S=S^0\oplus S^1$ of the Clifford
algebra $\Cl(\R^{2l})$ is a level~$1$ $\Spin^c(2l)$-module, so the
product bundle $E=\pi^*(P)\times S$ is a $\Z/2\Z$-graded
$\pi^*(\omega_V)$-twisted vector bundle over (the total space of) $V$.
Consider the two-term $G\times\Spin^c(2l)$-equivariant complex of
vector bundles
\begin{equation}\label{equation;thom}
E^0\overset\cliff\longto E^1
\end{equation}
defined by placing the term $E^j$ in degree $j$ and for each $v\in V$
letting $\cliff(v)\colon E_v^0\to E_v^1$ be Clifford multiplication by
$v$.  Since $\cliff(v)$ is an isomorphism for $v\ne0$, the complex
\eqref{equation;thom} defines a class in $L_G(BV,SV,\pi^*\omega_V)$,
and hence a class
$$\th(V)\in K_G^0(BV,SV,\pi^*\omega_V),$$
which is called the \emph{Thom class} of $V$.  The \emph{Thom map} is
the map
$$
\zeta_*\colon K_G^*(X,\tau)\to
K_G^*\bigl(BV,SV,\pi^*(\tau+\omega_V)\bigr)
$$
defined by $\zeta_*(a)=\pi^*(a)\th(V)$.  The following result is
\cite[Theorem~IV.6.21]{karoubi;k-theory;;1978}.  See also
\cite[\S\,3.2]{carey-wang;thom} for a discussion closer to our
treatment.

\begin{theorem}\label{theorem;thom}
The Thom map is an isomorphism of graded $K_G^*(X)$-modules.
\end{theorem}

%%%%%%%%%%%%%%%%%%%%%%%%%%%%%%%%%%%%%%%%%%%%%%%%%%%%%%%%%%%%%%%%%%%%%%%%
\section{Central extensions}\label{section;central}
%%%%%%%%%%%%%%%%%%%%%%%%%%%%%%%%%%%%%%%%%%%%%%%%%%%%%%%%%%%%%%%%%%%%%%%%

In this appendix we gather a few elementary facts regarding central
extensions.  The notation is as stated at the beginning of
\S\,\ref{section;induction}.  In particular $G$ denotes a compact
connected Lie group.  In addition $\check{G}$ denotes a
\emph{connected} central extension of $G$ by a \emph{compact} abelian
Lie group $C$,
$$
1\longto C\longto\check{G}\longto G\longto1.
$$

A complex $\check{G}$-module $V$ has \emph{central character}
$\chi\in\X(C)$ if the subgroup $C$ acts on $V$ by $c\cdot v=\chi(c)v$.
Let $\Rep^\chi(\check{G})$ be the category of finite-dimensional
complex $\check{G}$-modules of central character $\chi$.  We call the
Grothendieck group $R^\chi(\check{G})$ of $\Rep^\chi(\check{G})$ the
\emph{$\chi$-twisted representation module} of $G$.  The category
$\Rep^0(\check{G})$ (where $\chi=0$ is the trivial character) is
equivalent to $\Rep(G)$, so the groups $R^0(\check{G})$ and $R(G)$ are
isomorphic.  The tensor product functor
$$
\Rep^{\chi_1}(\check{G})\times\Rep^{\chi_2}(\check{G})\longto
\Rep^{\chi_1+\chi_2}(\check{G})
$$
induces a bi-additive map
$$
R^{\chi_1}(\check{G})\times R^{\chi_2}(\check{G})\longto
R^{\chi_1+\chi_2}(\check{G})
$$
In particular $R^\chi(\check{G})$ is an $R(G)$-module for all $\chi$.

\begin{section-lemma}\label{lemma;extension}
\begin{enumerate}
\item\label{item;isotype-extension}
The $R(G)$-module $R(\check{G})$ is the direct sum of the submodules
$R^\chi(\check{G})$ over all $\chi\in\X(C)$.  Each summand
$R^\chi(\check{G})$ is nonzero.
\item\label{item;invariant-extension}
Let $\check{T}\subset\check{G}$ be the inverse image of $T$.  Then
$\check{T}$ is a maximal torus of $\check{G}$; for each character
$\chi$ of $C$ the submodule $R^\chi(\check{T})$ of $R(\check{T})$ is
preserved by the $W_G$-action; and the restriction homomorphism
$R^\chi(\check{G})\to R^\chi(\check{T})$ is an isomorphism onto
$R^\chi(\check{T})^{W_G}$.
\end{enumerate}
\end{section-lemma}

\begin{proof}
Every $\check{G}$-module $V$ decomposes under the action of $C$ into a
direct sum $\bigoplus_{\chi\in\X(C)}V^\chi$ of isotypical submodules
$V^\chi$.  This decomposition is functorial and defines an equivalence
of categories
$$
\Rep(\check{G})\overset\sim\longto
\bigoplus_{\chi\in\X(C)}\Rep^\chi(\check{G}).
$$
Passing to Grothendieck groups we obtain the direct sum decomposition
in \eqref{item;isotype-extension}.  Each of the submodules
$R^\chi(\check{G})$ is nonzero, because there exists an irreducible
representation of $\check{G}$ with central character $\chi$, for
instance an appropriate subrepresentation of the formally induced
representation $\ind_C^{\check{G}}(\C_\chi)$.  Since $C$ is central in
$\check{G}$, the group $\check{T}$ is a maximal torus of $\check{G}$,
and the homomorphism $\check{G}\to G$ induces isomorphisms of root
systems $\ca{R}_{\check{G}}\cong\ca{R}_G$ and of Weyl groups
$W_{\check{G}}\cong W_G$.  The homomorphism $\check{T}\to T$ is
$W_G$-equivariant.  The $W_G$-action on $\check{T}$ fixes $C$ and
therefore the submodule $R^\chi(\check{T})$ of $R(\check{T})$ is
$W_G$-stable for each $\chi\in\X(C)$.  The isomorphism
$R^\chi(\check{G})\cong R^\chi(\check{T})^{W_G}$ now follows from
$R(\check{G})\cong R(\check{T})^{W_G}$.
\end{proof}

Let $V$ be an irreducible $\check{G}$-module.  By Schur's lemma, the
central subgroup $C$ acts on $V$ by a character $\chi_V$.  Define
\begin{equation}\label{equation;action}
c\cdot[V]=\chi_V(c)[V]
\end{equation}
for $c\in C$.  By linear extension, this formula defines a $C$-action
on the complexified representation ring $R(\check{G})_\C$ by ring
automorphisms.  The ring of $C$-invariants is
\begin{equation}\label{equation;galois-invariants}
(R(\check{G})_\C)^C\cong R(G)_\C.
\end{equation}
  
Recall that the \emph{rank} of a module $A$ over a domain $R$ is the
dimension of the vector space $F\otimes_RA$, where $F$ is the fraction
field of $R$, and is denoted by $\rank_R(A)$.

\begin{section-lemma}\label{lemma;rank}
$\rank_{R(G)}(R^\chi(\check{G}))=1$ for every $\chi\in\X(C)$.
\end{section-lemma}

\begin{proof}
Assume first that $C$ is finite.  Since $G$ and $\check{G}$ are
connected, the rings $R(G)$ and $R(\check{G})$ have no zero divisors,
so we can form the fraction fields $\K$ of $R(G)_\C$ and
$\check{\K}=R(\check{G})_\C\otimes_{R(G)_\C}\K$ of $R(\check{G})_\C$.
It follows from \eqref{equation;galois-invariants} that $\check{\K}$
is a Galois extension of $\K$ with Galois group $C$, which implies
$\dim_\K(\check{\K})=\abs{C}$.  For $\chi\in\X(C)$ let
$$\check{\K}^\chi=R^\chi(\check{G})_\C\otimes_{R(G)_\C}\K,$$
which is a $\K$-linear subspace of $\check{\K}$.  It follows from
Lemma \ref{lemma;extension}\eqref{item;isotype-extension} that as a
vector space over $\K$
$$\check{\K}=\bigoplus_{\chi\in\ca{X}(C)}\check{\K}^\chi,$$
where each of the summands is nonzero.  The number of summands is
$\abs{C}$ because of the fact that $\X(C)\cong C$, and therefore
$\dim_\K\bigl(\check{\K}^\chi\bigr)=1$ for all $\chi$.  Hence
$$
\rank_{R(G)}\bigl(R^\chi(\check{G})\bigr)
=\dim_\K\bigl(\check{\K}^\chi\bigr)=1.
$$
For general $C$ we make the basic observation, which appears to go
back to Shapiro \cite{shapiro;group-extensions-compact}, that the
compact central extension $\check{G}$ is the pushout of a
\emph{finite} central extension
$$
1\longto Z\longto\tilde{G}\longto G\longto1.
$$
One produces $\tilde{G}$ by choosing a Lie algebra homomorphism
$\kappa\colon\g\to\check{\g}$ which splits the exact sequence
$$0\to\lie{c}\to\check{\g}\to\g\to0.$$
One can choose $\kappa$ to be defined over $\Q$; this ensures that it
exponentiates to a Lie group homomorphism $\kappa\colon\tilde{G}\to
\check{G}$, where $\tilde{G}$ is a finite connected covering group of
$G$.  Let $Z$ be the kernel of the covering $\tilde{G}\to G$; then
$\kappa(Z)$ is contained in $C$.  Let $\chi$ be a character of $C$ and
$\kappa^*(\chi)$ its pullback to $Z$.  Then the $R(G)$-module
$R^\chi(\check{G})$ is isomorphic to $R^{\kappa^*(\chi)}(\tilde{G})$,
and we already know that the latter is of rank~$1$.
\end{proof}

\begin{section-example}\label{example;not-free}
Let $\check{G}_1=\check{G}_2=\SU(2)$ and
$$
\check{G}=\check{G}_1\times\check{G}_2\cong \Spin(4),\qquad
C=\{\pm(I,I)\},\qquad G=\check{G}/C\cong\SO(4).
$$
Then $R(\check{G})\cong R(\check{G}_1)\otimes
R(\check{G}_2)\cong\Z[x_1,x_2]$.  Identify $\X(C)$ with
$\Z/2\Z=\{0,1\}$; then $C$ acts on the monomial $x_1^{r_1}x_2^{r_2}\in
R(\check{G})$ with weight $(r_1+r_2)\bmod2$.  Hence $R(\check{G})$ is
the direct sum of the submodules
$$
R(G)\cong R^0(\check{G})=\bigoplus_{r_1+r_2\equiv0}\Z\cdot
x_1^{r_1}x_2^{r_2},\qquad
R^1(\check{G})=\bigoplus_{r_1+r_2\equiv1}\Z\cdot x_1^{r_1}x_2^{r_2},
$$
where the congruences are modulo $2$.  As a ring, 
$$R(G)\cong\Z[y_1,y_2,y_3]/(y_1y_2-y_3^2),$$
where the inclusion $R(G)\to R(\check{G})$ is given by
$$
y_1\longmapsto x_1^2,\qquad y_2\longmapsto x_2^2,\qquad y_3\longmapsto
x_1x_2.
$$
The twisted module $R^1(\check{G})$ is generated by $x_1$ and $x_2$,
which are subject to the single relation $y_3x_1-y_1x_2=0$.  Over the
quotient field of $R(G)$ the two generators are multiples of each
other, $x_2=(y_2y_3^{-1})x_1$ and $x_1=(y_1y_3^{-1})x_2$, so
$R^1(\check{G})$ is of rank $1$, as predicted by Lemma
\ref{lemma;rank}.  However, $R^1(\check{G})$ is not generated by any
single element and is therefore not free.
\end{section-example}

%%%%%%%%%%%%%%%%%%%%%%%%%%%%%%%%%%%%%%%%%%%%%%%%%%%%%%%%%%%%%%%%%%%%%%%%
\section{Shifted anti-invariants}\label{section;anti-invariant}
%%%%%%%%%%%%%%%%%%%%%%%%%%%%%%%%%%%%%%%%%%%%%%%%%%%%%%%%%%%%%%%%%%%%%%%%

This appendix is devoted to the proof of the statement below
(Proposition \ref{proposition;anti-invariant} in the main text).  We
use the notation defined in \S\,\ref{section;dirac}.  In particular,
$G$ is a compact connected Lie group with maximal torus $T$,
$\omega_{G/T}$ is the orientation system of the flag variety $G/T$
defined in \eqref{equation;orientation}, $J_G$ is the antisymmetrizer
\eqref{equation;antisymmetrizer}, and $\d_G$ is the Weyl denominator
\eqref{equation;weyl-denominator}.  Recall that $A^{-W_G}$ denotes the
set of anti-invariant elements of a $W_G$-module $A$.  If $G$ is
semisimple and simply connected, then $\rho_G\in\X(T)$ and therefore
the $R(T)$-module $R(T,\omega_{G/T})$ is $W_G$-equivariantly
isomorphic to $R(T)$ by Lemma
\ref{lemma;weight-euler}\eqref{item;Romega}.  The proposition is then
a standard fact; see e.g.\ \cite[\S\,VI.3,
Proposition~2]{bourbaki;groupes-algebres}.  We will deduce the general
case from this special case.

\begin{proposition*}
\begin{enumerate}
\item\label{item;preserve}
The $W_G$-action on $R\bigl(T^{(\omega_{G/T})}\bigl)$ preserves
$R(T,\omega_{G/T})$.
\item\label{item;free}
The set of anti-invariants $R(T,\omega_{G/T})^{-W_G}$ is a free
$R(G)$-module of rank $1$ generated by $\d_G$.
\item\label{item;antisymmetrizer}
The elements $J_G(e^\lambda)$, with $\lambda\in\rho_G+\X(T)$ strictly
dominant, form a basis of the $\Z$-module $R(T,\omega_{G/T})^{-W_G}$.
\end{enumerate}
\end{proposition*}

\begin{proof}
Put
$$
W=W_G,\qquad J=J_G,\qquad
\omega=\omega_{G/T},\qquad\d=\d_G,\qquad\rho=\rho_G.
$$
We identify $R(T)$ with $\Z[\X(T)]$ and $R(T,\omega)$ with
$R(T)\,e^\rho=\Z[\rho+\X(T)]$ as in Remark \ref{remark;shift}.  The
fact that $\rho-w(\rho)$ is in the root lattice for all $w\in W$
implies that the $W$-action on $\X(T)_\Q$ preserves the affine lattice
$\rho+\X(T)$.  This proves~\eqref{item;preserve}.  Let
$\phi\colon\tilde{G}\to G$ be a compact connected covering group which
is the product $\tilde{G}=C\times\bar{G}$ of a torus $C$ and a simply
connected group $\bar{G}$.  Let $\tilde{T}$ be the maximal torus
$\phi^{-1}(T)$ of $\tilde{G}$, and identify $\X(T)$ with its image
$\phi^*(\X(T))$.  Let $\bar{T}$ be the maximal torus
$\bar{G}\cap\tilde{T}$ of $\bar{G}$.  Since $\bar{G}$ is simply
connected, $\rho\in\X(\bar{T})$ and therefore $\d\in R(\bar{T})$.  We
have
$$
R(\tilde{T})^W=R(C\times\bar{T})^W=\bigl(R(C)\otimes_\Z
R(\bar{T})\bigr)^W=R(C)\otimes_\Z R(\bar{T})^W,
$$
because $W$ acts trivially on $R(C)$ and $R(C)$ is a free abelian
group.  Since $\bar{G}$ is simply connected, it follows from
\cite[\S\,VI.3, Proposition~2]{bourbaki;groupes-algebres} that
$$R(\bar{T})^{-W}=R(\bar{T})^W\cdot\d.$$
Therefore
\begin{multline}\label{equation;generator}
R(\tilde{T})^{-W}=R(C\times\bar{T})^{-W}=\bigl(R(C)\otimes_\Z
R(\bar{T})\bigr)^{-W}=R(C)\otimes_\Z R(\bar{T})^{-W}
\\
=R(C)\otimes_\Z\bigl(R(\bar{T})^W\cdot\d\bigr)=\bigl(R(C)\otimes_\Z
R(\bar{T})^W\bigr)\cdot\d=R(\tilde{T})^W\cdot\d.
\end{multline}
Now let $a\in R(T,\omega)^{-W}\subset R(\tilde{T})^{-W}$.  It follows
from \eqref{equation;generator} that $a=b\d$ for some $b\in
R(\tilde{T})^W$.  We need to argue that $b\in R(T)$.  Let
$K\subset\tilde{T}$ be the kernel of the covering homomorphism
$\phi\colon\tilde{G}\to G$.  This group acts on the complexified
representation ring $R(\tilde{T})_\C$ as in \eqref{equation;action},
and the ring of $K$-invariants $(R(\tilde{T})_\C)^K$ is isomorphic to
$R(T)_\C$.  Since $a$ and $\d$ are in $R(T,\omega)=R(T)e^\rho$, we
have $k\cdot a=\rho(k)a$ and $k\cdot\d=\rho(k)\d$ for all $k\in K$ and
hence
$$
b\d=a=\rho(k)^{-1}k\cdot a=\rho(k)^{-1}k\cdot(b\d)
=\rho(k)^{-1}(k\cdot b)(k\cdot\d)=(k\cdot b)\d.
$$
It follows that $k\cdot b=b$ for all $k\in K$, i.e.\ $b\in R(T)$.
This proves \eqref{item;free}.  \eqref{item;antisymmetrizer} is proved
in exactly the same way as \cite[\S\,VI.3,
Proposition~1]{bourbaki;groupes-algebres}.
\end{proof}

%%%%%%%%%%%%%%%%%%%%%%%%%%%%%%%%%%%%%%%%%%%%%%%%%%%%%%%%%%%%%%%%%%%%%%%%
\section{Homogeneous $\Spin^c$-structures}\label{section;spin}
%%%%%%%%%%%%%%%%%%%%%%%%%%%%%%%%%%%%%%%%%%%%%%%%%%%%%%%%%%%%%%%%%%%%%%%%

In this appendix we review the classification of invariant
$\Spin^c$-structures on equal-rank homogeneous spaces, which is surely
well-known but for which we could not find a reference.  (But see
\cite[\S~2.6]{hirzebruch-slodowy;elliptic} and also Example
\ref{example;spin} below for remarks on the $\Spin$ case.)  See
\cite[Example~4.6]{meinrenken;quantization-conjugacy} and
\cite{landweber-sjamaar;spin} for examples of maximal-rank homogeneous
spaces that do not carry invariant $\Spin^c$-structures.

The notation and the assumptions are as explained at the beginning of
\S\S\,\ref{section;induction} and~\ref{section;dirac}.  Recall that
$G$ denotes a compact connected Lie group, $H$ a closed and connected
subgroup of maximal rank, and $T$ a common maximal torus of $G$ and
$H$.  Recall also that $\m=T_{\bar1}M$ denotes the tangent space at
the identity coset of the homogeneous space $M=G/H$, and $\eta\colon
H\to\SO(\m)$ denotes the tangent representation.  We denote the set of
equivalence classes of $G$-invariant $\Spin^c$-structures on $M$ by
$\lie{Spin}_G^c(M)$.%
\glossary{SpincM@$\lie{Spin}_G^c(M)$, set of equivalence classes of
invariant $\Spin^c$-structures on $M$}

\begin{section-definition}\label{definition;cspinorial}
The orthogonal representation $\eta$ is \emph{c-spinorial} if it lifts
to a homomorphism $H\to\Spin^c(\m)$.  The subgroup $H$ is
\emph{c-spinorial} if $\eta$ is c-spinorial.
\end{section-definition}

Note that $\eta$ is c-spinorial if and only if $M$ possesses a
$G$-invariant $\Spin^c$-struc\-ture, i.e.\ if and only if
$\lie{Spin}_G^c(M)$ is nonempty.  Moreover, liftings of $\eta$ to
$\Spin^c(\m)$ correspond bijectively to elements of $\lie{Spin}_G^c(M)$.

A lifting of $\eta$ to $\Spin^c(\m)$ determines a trivialization $s$
of the central extension
$$
\omega=\omega_M\colon\xymatrix@1@M+3pt{
1\ar[r]&\U(1)\ar[r]&
H^{(\omega)}\ar[r]_-\phi&H\ar[r]\ar@{.>}@<-1ex>[l]_-s&1
}
$$
defined in \eqref{equation;orientation}.  Conversely, given a section
$s$ of $\phi$ we can compose it with the canonical homomorphism
$\hat{\eta}\colon H^{(\omega)}\to\Spin^c(\m)$ to obtain a lifting of
$\eta$,
$$
H\overset{s}\longto
H^{(\omega)}\overset{\hat{\eta}}\longto\Spin^c(\m).
$$
Thus we have a natural one-to-one correspondence between
$\lie{Spin}_G^c(M)$ and the set of trivializations of the orientation
system $\omega$.

The \emph{determinant character} is the character $\det\colon
H^{(\omega)}\to\U(1)$ obtained by pulling back the determinant
character of $\Spin^c(\m)$, which is defined by $\det([z,a])=z^2$.
The homomorphism
$$
\psi=\det\times\phi\colon H^{(\omega)}\longto\U(1)\times H
$$
is a double covering map.  Let $s$ be a section of $\phi$ and let
$\gamma$ be the character $\det\circ s$ of $H$.  Then $s$ is a lifting
homomorphism of $\gamma\times\id$, as in the commutative diagram
$$
\vcenter{\xymatrix{
&H^{(\omega)}\ar[d]^\psi\\
H\ar[r]_-{\gamma\times\id}\ar@{.>}[ur]^{s}&\U(1)\times H,
}}
$$
and so $s=s_\gamma$ is uniquely determined by $\gamma$.

\begin{section-definition}\label{definition;cspinorial-character}
A character $\gamma\in\X(H)$ is \emph{c-spinorial} (relative to the
orthogonal $H$-module $\m$) if the lifting $s_\gamma$ of
$\gamma\times\id$ exists.  We denote the set of $c$-spinorial
characters by $\X(H)^c$.%
\glossary{XHc@$\X(H)^c$, c-spinorial characters of $H$}
\end{section-definition}

The conclusion is that we have natural bijections between the set
$\lie{Spin}_G^c(M)$, the set of trivializations of $\omega$, and the
set $\X(H)^c$.  This proves the first part of the following
proposition.  The map $\nu\colon\X(H)\to\frac12\X(T)$ in the second
part is defined by
$$
\nu(\gamma)
=\frac12j_H^*(\gamma)-\frac12\sum_{\alpha\in\ca{R}_M^+}\alpha
=\frac12j_H^*(\gamma)-\rho_M.
$$

\begin{section-proposition}\label{proposition;spin-c}
\begin{enumerate}
\item\label{item;spin-c}
For a $c$-spinorial character $\gamma$, define a $G$-invariant
principal $\Spin^c(\m)$-bundle over $M$ by
$$\P_\gamma=G\times^H\Spin^c(\m),$$
where the right-hand side denotes the quotient of $G\times\Spin^c(\m)$
by the $H$-action $h\cdot(g,k)=(gh^{-1},\hat{\eta}(s_\gamma(h))k)$.
The map
$$f\colon\X(H)^c\longto\lie{Spin}_G^c(M)$$
which sends $\gamma$ to the equivalence class of $\P_\gamma$ is
bijective.  In particular $\lie{Spin}_G^c(M)$ is nonempty if and only
if $\X(H)^c$ is nonempty.
\item\label{item;lifting}
A character $\gamma$ of $H$ is $c$-spinorial if and only if
$\nu(\gamma)\in\X(T)$.  Hence $\X(H)^c=\nu^{-1}(\X(T))$.
\end{enumerate}
\end{section-proposition}

\begin{proof}
It remains to prove \eqref{item;lifting}.  By covering space theory
and our standing assumption that $H$ is connected, a lifting $s$ of
$\gamma\times\id$ exists if and only if it exists on the level of
fundamental groups, as in the diagram
\begin{equation}\label{equation;fundamental}
\vcenter{\xymatrix{
&\pi_1\bigl(H^{(\omega)}\bigr)\ar[d]^{\psi_*}\\
\pi_1(H)\ar[r]_-{\gamma_*\times\id}\ar@{.>}[ur]^{s_*}
&\pi_1(\U(1))\times\pi_1(H).
}}
\end{equation}
Let $\Y(T)=\Hom(\U(1),T)$ denote the cocharacter group of $T$.  The
fundamental group $\pi_1(H)$ is naturally isomorphic to
$\Y(T)/Q(\ca{R}_H^\vee)$, where $Q(\ca{R}_H^\vee)$ is the coroot
lattice in $\Y(T)$.  (See e.g.\
\cite[\S\,IX.4.6]{bourbaki;groupes-algebres}.)  Since $\psi$ is a
covering map, $\psi_*$ maps the coroot system of $H^{(\omega)}$
bijectively to that of $\U(1)\times H$.  Therefore the lifting problem
\eqref{equation;fundamental} is equivalent to the lifting problem
$$
\xymatrix{
&{\Y\bigl(T^{(\omega)}\bigr)}\ar[d]^{\psi_*}\\
\ar[r]_-{\gamma_*\times\id}\ar@{.>}[ur]^{s_*}
{\Y(T)}&{\Y(\U(1))\oplus\Y(T)}.
}
$$
Since $\X(T)$ and $\Y(T)$ are dual abelian groups, this lifting
problem is equivalent to the dual extension problem
\begin{equation}\label{equation;extend}
\vcenter{\xymatrix{
&{\X\bigl(T^{(\omega)}\bigr)}\ar@{.>}[dl]_{s^*}\\
{\X(T)}&{\X(\U(1))\oplus\X(T)}.\ar[u]_{\psi^*=\det^*\times\phi^*}
\ar[l]^-{\gamma^*\oplus\id}
}}
\end{equation}
By Lemma \ref{lemma;weight-euler}\eqref{item;omega},
$\X(T^{(\omega)})$ is the direct sum of $\phi^*(\X(T))$ and
$\Z\cdot(\eps_0-\rho_M)$.  Let $\X$ be the image of $\psi^*$, which is
a sublattice of $\X\bigl(T^{(\omega)}\bigr)$ of index $2$.  Since
$2\rho_M$ is in $\X(T)$ and $\psi^*$ maps the generator $\eps_0$ of
$\X(\U(1))$ to the determinant character $\det=2\eps_0$, the element
$\eps_0-\rho_M\in\X\bigl(T^{(\omega)}\bigr)$ is a representative of
the nontrivial coset in $\X\bigl(T^{(\omega)}\bigr)\big/\X$.  Because
$2(\eps_0-\rho_M)$ is equal to $\psi^*(\eps_0-2\rho_M)$, the extension
problem \eqref{equation;extend} is soluble if and only if
$$
2\nu(\gamma)=j_H^*(\gamma)-2\rho_M=\gamma^*(\eps_0)-2\rho_M\in\X(T)
$$
is divisible by $2$ in $\X(T)$.  If this is the case, the extension
$s^*=s_\gamma^*$ is uniquely determined by the formula
\begin{equation}\label{equation;section}
s_\gamma^*(\eps_0-\rho_M)=\nu(\gamma)\in\X(T).
\end{equation}
This proves \eqref{item;lifting}.
\end{proof}

We call $\gamma\in\X(H)^c$ the \emph{character} of the
$\Spin^c$-structure $\P_\gamma$.  Let $\gamma\in\X(H)^c$ and
$\chi\in\X(H)$.  Then
$$\nu(\gamma+2\chi)=\nu(\gamma)+j_H^*(\chi)\in\X(T),$$
so $\gamma+2\chi\in\X(H)^c$.  Thus we have an action of the abelian
group $\X(H)$ on the set of $c$-spinorial characters defined by
$\chi\cdot\gamma=\gamma+2\chi$.  Let $\group{L}_\chi=G\times^H\C_\chi$
be the homogeneous complex line bundle on $M$ defined by
$\chi\in\X(H)$.  Recall the natural isomorphisms
\begin{equation}\label{equation;picard}
\X(H)\overset\cong\longto\Pic_G(M)\overset\cong\longto H_G^2(M,\Z),
\end{equation}
where $\Pic_G(M)$ denotes the (topological) Picard group of
isomorphism classes of homogeneous complex line bundles on $M$.  The
first map sends a character $\chi$ to the class of the bundle
$\group{L}_\chi$ and the second map sends the class of a bundle $L$ to
its equivariant Chern class $c_G^1(L)$.  (See e.g.\
\cite[Theorem~C.47]{ginzburg-guillemin-karshon;moment-maps-cobordism}.)
A natural action of $\Pic_G(M)$ on $\lie{Spin}_G^c(M)$ is defined by
$$[L]\cdot[\P]=\bigl[\U(L)\times^{\U(1)_M}\P\bigr],$$
where $\U(L)$ denotes the circle bundle associated to $L$ and we
identify the circle $\U(1)$ with the kernel of
$\psi\colon\Spin^c(\m)\to\SO(\m)$.  (The quotient in the right-hand
side is taken in the category of manifolds over $M$, and
$\U(1)_M=M\times\U(1)$ denotes the constant groupoid over $M$ with
fibre $\U(1)$.)

\begin{section-proposition}\label{proposition;principal-chern-euler}
The notation is as in Proposition {\rm\ref{proposition;spin-c}}.
\begin{enumerate}
\item\label{item;equivariant}
The bijection $f\colon\X(H)^c\to\lie{Spin}_G^c(M)$ is
$\X(H)$-equivariant with respect to the group isomorphism
\eqref{equation;picard}.
\item\label{item;torsor}
$\lie{Spin}_G^c(M)$ is a principal homogeneous space for the abelian
group $\X(H)$.  In particular $\lie{Spin}_G^c(M)$ consists of at most
one element if $H$ is semisimple.
\item\label{item;chern}
Let $\gamma\in\X(H)^c$.  The determinant line bundle of $\P_\gamma$ is
$\group{L}_\gamma=G\times^H\C_\gamma$.  Its equivariant Chern class
$c_1(\group{L}_\gamma)\in H_G^2(M,\Z)\cong\X(H)$ is equal to $\gamma$.
\item\label{item;euler}
Let $\gamma\in\X(H)^c$ and let $\dirac_\gamma$ be the $\Spin^c$ Dirac
operator associated with $\P_\gamma$.  The equivariant Euler class of
$\dirac_\gamma$ is $\e(\dirac_\gamma)=e^{\gamma/2}\e(\Dirac)\in R(H)$,
where $\Dirac$ is the twisted $\Spin^c$ Dirac operator of $M$.
\end{enumerate}
\end{section-proposition}

\begin{proof}
Let $\gamma\in\X(H)^c$ and $\chi\in\X(H)$.  The fibre of
$\U(\group{L}_\chi)\times^{\U(1)_M}\P_\gamma$ over the identity coset
$\bar1\in M$ is
$$
\U(\C_\chi)\times^{\U(1)}\Spin^c(\m),
$$
where $\U(\C_\chi)$ denotes the unit circle in $\C_\chi$.  The map
$$
\U(\C_\chi)\times^{\U(1)}\Spin^c(\m)\longto\Spin^c(\m)
$$
defined by $(z,k)\mapsto zk$ is a diffeomorphism and is equivariant
with respect to the $H$-actions defined on the left-hand side by
$$h\cdot[z,k]=[\chi(h)^{-1}z,\hat{\eta}(s_\gamma(h))k]$$
and on the right-hand side by $h\cdot
k=\hat{\eta}(s_{\gamma+2\chi}(h))k$.  It follows that
$[L_\chi]\cdot[\P_\gamma]=[\P_{\gamma+2\chi}]$, which proves
\eqref{item;equivariant}.  It is easy to verify that the
$\X(H)$-action on $\X(H)^c$ is free and transitive.  Thus
\eqref{item;torsor} follows from \eqref{item;equivariant}.  The
determinant line bundle of $\P_\gamma$ is
$$
\group{L}_\gamma=\P_\gamma\times^{\Spin^c(\m)}\C_{\det}
=\bigl(G\times^H\Spin^c(\m)\bigr)^{\Spin^c(\m)}\C_{\det}
=G\times^H\C_\gamma.
$$
The isomorphism $\X(H)\to H_G^2(M,\Z)$ defined in
\eqref{equation;picard} maps $\gamma$ to $c_G^1(\group{L}_\gamma)$,
which proves \eqref{item;chern}.  By \eqref{equation;euler} and
\eqref{equation;euler-twist}, the Euler class of $\dirac_\gamma$ is
$$
\e(\dirac_\gamma)=(\hat{\eta}\circ s_\gamma)^*\bigl([S^0]-[S^1]\bigr)
=s_\gamma^*\hat{\eta}^*\bigl([S^0]-[S^1]\bigr)
=s_\gamma^*(\e(\Dirac))\in R(H),
$$
where $S$ is the spinor module of $\Cl(\m)$.  Using Lemma
\ref{lemma;weight-euler}\eqref{item;euler-twist} and
\eqref{equation;section} we obtain
$$
j_H^*(\e(\dirac_\gamma))=s_\gamma^*\biggl(e^{\eps_0-\rho_M}
\prod_{\alpha\in\ca{R}_M^+}(1-e^\alpha)\biggl)
=e^{\nu(\gamma)}\prod_{\alpha\in\ca{R}_M^+}(1-e^\alpha)
=j_H^*\bigl(e^{\gamma/2}\e(\Dirac)\bigr),
$$
which establishes \eqref{item;euler}.
\end{proof}

\begin{section-example}[$\Spin$-structures]\label{example;spin}
Let us call the orthogonal representation $\eta\colon H\to\SO(\m)$, or
the subgroup $H$, \emph{spinorial} if $\eta$ lifts to a homomorphism
$\tilde{\eta}\colon H\to\Spin(\m)$.  Thus, $H$ is spinorial if and
only if $M$ has a $G$-invariant $\Spin$-structure, and such a
structure is unique up to equivalence, because $H$ is connected.  A
$\Spin$-structure determines a $\Spin^c$-structure by means of the
homomorphism $\kappa\circ\tilde{\eta}\colon H\to\Spin^c(\m)$, where
$\kappa\colon\Spin(\m)\to\Spin^c(\m)$ is the inclusion map.  The
corresponding $c$-spinorial character of $H$ is
$\gamma=\det\circ\kappa\circ\tilde{\eta}=0$, because $\Spin(\m)$ is
the kernel of the determinant character.  Hence, by Proposition
\ref{proposition;spin-c}\eqref{item;lifting}, $M$ is $G$-invariantly
$\Spin$ if and only if $\rho_M\in\X(T)$.  The Euler class of the
$\Spin$ Dirac operator is equal to that of the twisted $\Spin^c$ Dirac
operator $\Dirac$.
\end{section-example}
  
\begin{section-remark}\label{remark;spin}
If $H$ is c-spinorial and semisimple, then a lifting $H\to\Spin^c(\m)$
takes values in the commutator subgroup $\Spin(\m)$, so $H$ is
spinorial.
\end{section-remark}

\begin{section-example}[almost complex structures]
\label{example;almost-complex}
Up to homotopy, $G$-invariant almost complex structures on $M$
correspond bijectively to $H$-invariant orthogonal complex structures
on $\m$.  Let $J$ be such a structure; then the tangent representation
$\eta\colon H\to\SO(\m)$ takes values in $\U(\m,J)$, the group of
$J$-holomorphic orthogonal maps.  We equip $M$ with the
$\Spin^c$-structure defined by composing this homomorphism with the
canonical injection $\iota$ (see e.g.\
\cite[\S\,D.3.1]{ginzburg-guillemin-karshon;moment-maps-cobordism}),
$$
H\overset\eta\longto\U(\m,J)\overset\iota\longinj\Spin^c(\m).
$$
The associated $c$-spinorial character $\gamma$ is then
$\gamma=\det\circ\iota\circ\eta$.  Since $\det\circ\iota$ is equal to
the usual complex determinant character of $\U(\m,J)$, we have
$j_H^*(\gamma)=\sum_{k=1}^l\beta_k$, where the $\beta_k$ are the
weights of the $T$-action on $\m$ with respect to $J$.  To compute
this, let $\ca{R}_M^+=\{\alpha_1,\alpha_2,\dots,\alpha_l\}$ and let
$J_0$ be the $T$-invariant complex structure on $\m$ given by the
decomposition $\m=\bigoplus_{k=1}^l\m_k$, where
$\m_k\cong\g_\C^{\alpha_k}$.  On $\m_k$ we have $J=c_kJ_0$, where
$c_k=\pm1$.  Hence $\beta_k=c_k\alpha_k$ and
$j_H^*(\gamma)=\sum_{k=1}^lc_k\alpha_k$.
\end{section-example}

\begin{section-example}[complex structures]\label{example;complex}
As a special case of Example \ref{example;almost-complex}, let us take
$H$ to be the centralizer of a subtorus of $T$.  This is the case
precisely when $M$ has an integrable invariant almost complex
structure, as one sees from the well-known identification $M\cong
G_\C/P$ (\cite[\S\,IX.4, Exercice~8]{bourbaki;groupes-algebres}),
where $P$ is the parabolic subgroup of $G_\C$ generated by $H$ and the
negative root spaces.  The corresponding orthogonal almost complex
structure $J$ on $\m$ is positive with respect to the inner product on
$\m$, so $c_k=1$ for all $k$ and hence
$$
j_H^*(\gamma)=2\rho_M,\qquad\nu(\gamma)=0,\qquad
j_H^*(\e(\dirac_\gamma))=\prod_{k=1}^l(1-e^\alpha).
$$
The K\"ahler structure on $M$ and the associated $\Spin^c$-structure
depend on the choice of the basis of $\ca{R}_G$.
\end{section-example}

%%%%%%%%%%%%%%%%%%%%%%%%%%%%%%%%%%%%%%%%%%%%%%%%%%%%%%%%%%%%%%%%%%%%%%%%

 \begin{theglossary}\begin{multicols}{2}

  \item $\bar1$, identity coset $1H\in M$, 3

  \indexspace

  \item $A\times_CB$, fibred product, 3
  \item $\alpha$, root of $G$, 3

  \indexspace

  \item $\ca{B}_G$, basis of $\ca{R}_G$, 3
  \item $BM$, unit ball bundle of cotangent bundle $T^*M$, 9

  \indexspace

  \item $\C_\chi$, one-dimensional $G$-module defined by character $\chi\in\X(G)$, 3
  \item $\Cl(E)$, Clifford algebra of a vector space or vector bundle $E$, 9

  \indexspace

  \item $[D]$, symbol class of operator $D$, 9
  \item $\d_G$, Weyl denominator of $G$, 17
  \item $\Dirac$, twisted $\Spin^c$ Dirac operator, 12
  \item $\dirac$, $\Spin^c$ Dirac operator, 20

  \indexspace

  \item $e^\chi$, class of $\C_\chi$ in $R(G)$, 3
  \item $\e(D)$, Euler class of operator $D$, 9
  \item $\eps_0$, standard generator of $\X(\U(1))\cong\Z$, 15
  \item $\eta$, tangent representation $H\to\GL(\m)$, 3

  \indexspace

  \item $G$, compact connected Lie group, 3
  \item $G\times^HV$, homogeneous vector bundle, 3
  \item $G^{(\sigma)}$, central extension of $G$ by $\U(1)$, 4
  \item $\Gamma$, smooth global sections functor, 4
  \item $\gamma$, c-spinorial character of $H$, 20
  \item $\g$, Lie algebra of $G$, 3
  \item $\g_\C^\alpha$, root space of $\g_\C$, 11

  \indexspace

  \item $H$, closed subgroup of $G$, from \S\,\ref{subsection;maximal} onward connected and containing $T$, 3
  \item $H^{(\tau)}$, central extension of $H$ by $\U(1)$, 6

  \indexspace

  \item $i$, inclusion $H\to G$, 3
  \item $i_!$, formal induction, 4
  \item $i_*$, twisted $\Spin^c$-induction, 14
  \item $i_\dirac$, $\Spin^c$-induction, 20
  \item $i_D$, induction defined by operator $D$, 8
  \item $\ind_H^G$, formal induction, 4

  \indexspace

  \item $j_G$, inclusion $T\to G$, 3
  \item $j_H$, inclusion $T\to H$, 11
  \item $J_G$, antisymmetrizer of $W_G$, 17
  \item $J_M$, relative antisymmetrizer, 17
  \item $J_M^\op$, ``opposite'' of $J_M$, 28

  \indexspace

  \item $K(\lie{E})$, Grothendieck group of category $\lie{E}$, 3

  \indexspace

  \item $M$, homogeneous space $G/H$, 3
  \item $\m$, tangent space $T_{\bar1}M$, 3
  \item $\m^\alpha$, weight space, 11

  \indexspace

  \item $\omega_M$, orientation system of $M$, 13

  \indexspace

  \item $\pi$, projection $T^*M\to M$, 14

  \indexspace

  \item $R(G)$, representation ring, 3
  \item $R(G)^*$, dual $\Z$-module of $R(G)$, 4
  \item $R(G,\sigma)$, twisted representation module, 4
  \item $R(H)^\vee$, dual $R(G)$-module of $R(H)$, 4
  \item $\ca{R}_G$, root system of $G$, 3
  \item $\ca{R}_G^+$, positive roots of $G$, 3
  \item $\ca{R}_M$, weights $\ca{R}_G\setminus\ca{R}_H$ of $\m_\C$, 11
  \item $\ca{R}_M^+$, positive weights $\ca{R}_G^+\setminus\ca{R}_H^+$ of $\m_\C$, 11
  \item $\rho_G$, half-sum of positive roots of $G$, 3
  \item $\rho_M$, half-character $\rho_G-\rho_H$, 13

  \indexspace

  \item $S$, spinor module $S^0\oplus S^1$ of $\Cl(\m)$, 13
  \item $\sigma$, central extension of $G$ by $\U(1)$, 4
  \item $SM$, unit sphere bundle of cotangent bundle $T^*M$, 9
  \item $\lie{Spin}_G^c(M)$, set of equivalence classes of invariant $\Spin^c$-structures on $M$, 36

  \indexspace

  \item $T$, maximal torus of $G$, 3
  \item $\tau$, central extension of $H$ by $\U(1)$, 6

  \indexspace

  \item $\U(1)$, unit circle, 3

  \indexspace

  \item $V_G(\lambda,\sigma)$, irreducible $G^{(\sigma)}$-module of level $1$ with highest weight $\lambda$, 21

  \indexspace

  \item $W_G$, Weyl group of $G$, 3
  \item $W^H$, shortest representatives for $W_G/W_H$, 11

  \indexspace

  \item $X$, topological $G$-space, 24
  \item $\X(G)$, character group $\Hom(G,\U(1))$, 3
  \item $\X(H)^c$, c-spinorial characters of $H$, 36

  \indexspace

  \item $\zeta$, zero section $M\to T^*M$, 9

 \end{multicols}\end{theglossary}

%%%%%%%%%%%%%%%%%%%%%%%%%%%%%%%%%%%%%%%%%%%%%%%%%%%%%%%%%%%%%%%%%%%%%%%%

%%%%%%%%%%%%%%%%%%%%%%%%%%%%%%%%%%%%%%%%%%%%%%%%%%%%%%%%%%%%%%%%%%%%%%%%

\bibliographystyle{amsplain}

\def\cprime{$'$}
\providecommand{\bysame}{\leavevmode\hbox to3em{\hrulefill}\thinspace}
\providecommand{\MR}{\relax\ifhmode\unskip\space\fi MR }
\providecommand{\MRhref}[2]{%
  \href{http://www.ams.org/mathscinet-getitem?mr=#1}{#2}
}
\providecommand{\href}[2]{#2}

%%%%%%%%%%%%%%%%%%%%%%%%%%%%%%%%%%%%%%%%%%%%%%%%%%%%%%%%%%%%%%%%%%%%%%%%

\end{document}